\documentclass{amsart}
\pdfoutput=1

\usepackage[english]{babel}
\usepackage[latin1]{inputenc}
\usepackage{amsmath, amsthm, amssymb, stmaryrd, enumitem}
\usepackage{bm}
\usepackage{tabu}

\usepackage{xcolor}

\usepackage[
	pdftitle={Unweighted Donaldson-Thomas theory of the banana 3-fold with section classes},
	pdfauthor={Oliver Leigh},
	ocgcolorlinks,
	linkcolor=blue,
	citecolor=blue,
	urlcolor=blue]
{hyperref}

\usepackage[all]{xy} 

\usepackage{tikz}
\usepackage{tikz-cd}
\usetikzlibrary{%
  matrix,%
  calc,%
  arrows%
}

\usepackage{caption}
\newcounter{mainTheoremCounter}
\newtheoremstyle{mainTheorem}{}{}{\addtolength{\leftskip}{0em}\itshape}{}{\bfseries}{}{.5em}{\thmname{#1} {#2}}
\theoremstyle{mainTheorem}
\newtheorem{mainTheorem}[mainTheoremCounter]{Theorem}
\newtheorem{mainCorollary}[mainTheoremCounter]{Corollary}

\newtheoremstyle{mainDefinition}{}{}{}{}{\bfseries}{}{.5em}{\thmname{#1} {#2}}
\theoremstyle{mainDefinition}

\newtheoremstyle{mainRemark}{}{}{}{}{\bfseries}{}{.5em}{\thmname{#1} {#2}}
\theoremstyle{mainRemark}

\newcounter{appendixTheoremCounter}
\newtheoremstyle{appendixTheorem}{}{}{\addtolength{\leftskip}{0em}\itshape}{}{\bfseries}{}{.5em}{\thmname{#1} {#2}}
\theoremstyle{appendixTheorem}

\theoremstyle{plain}  

\newtheorem{theorem}{Theorem}[subsection]
\newtheorem{lemma}[theorem]{Lemma}

\theoremstyle{definition} 

\newtheorem{definition}[theorem]{Definition}

\newtheorem{re}[theorem]{}
\newtheorem{remark}[theorem]{Remark}

\newcounter{subequation}[equation]  

\newcommand{\num}{\stepcounter{equation}\tag{\arabic{equation}}}

\newcommand{\C}{\mathbb{C}}

\renewcommand{\P}{\mathbb{P}}

\newcommand{\Hilb}{\mathrm{Hilb}}
\newcommand{\Chow}{\mathrm{Chow}}
\newcommand{\Quot}{\mathrm{Quot}}
\newcommand{\DT}{\textnormal{\textsf{DT}}}
\newcommand{\uDT}{\widehat{\textnormal{\textsf{DT}}}}
\newcommand{\Cyc}{\mathsf{Cyc}}
\newcommand{\HilbCyc}{\mathrm{Hilb}_{\Cyc}}

\renewcommand{\O}{\mathcal{O}}
\newcommand{\Spec}{\mathrm{Spec}}

\newcommand{\Sym}{\mathrm{Sym}}
\newcommand{\Supp}{\mathrm{Supp}}

\newcommand{\cycle}[1]{\mathfrak{#1}}

\begin{document}
\setlength{\parindent}{0cm}

\title[Unweighted DT Theory of the Banana 3-fold w/Section Classes]{Unweighted Donaldson-Thomas theory of the banana 3-fold with section classes}

\author{Oliver Leigh}
\address{School of Mathematics and Statistics, The University of Melbourne, Victoria, 3010, Australia.}
\address{Department of Mathematics, The University of British Columbia, Vancouver, BC, V6T 1Z2, Canada.}
\address{Matematiska institutionen, Stockholms universitet, 106 91 Stockholm, Sweden}
\curraddr{}
\email{oliver.leigh@math.su.se}
\thanks{}

\subjclass[2010]{}

\date{}
\dedicatory{}

\begin{abstract}
We further the study of the Donaldson-Thomas theory of the banana threefolds which were recently discovered and studied by Bryan in \cite{Bryan_Banana}. These are smooth proper Calabi-Yau threefolds which are fibred by Abelian surfaces such that the singular locus of a singular fibre is a non-normal toric curve known as a ``banana configuration''. In \cite{Bryan_Banana} the Donaldson-Thomas partition function for the rank 3 sub-lattice generated by the banana configurations is calculated. In this article we provide calculations with a view towards the rank 4 sub-lattice generated by a section and the banana configurations. We relate the findings to the Pandharipande-Thomas theory for a rational elliptic surface and present new Gopakumar-Vafa invariants for the banana threefold. 

\end{abstract}

\maketitle

\vspace{-0.75cm}
\tableofcontents

\section{Introduction}
\subsection{Donaldson-Thomas Partition Functions}
Donaldson-Thomas theory provides a virtual count of curves on a threefold. It gives us valuable information about the structure of the threefold and has strong links to high-energy physics. \\

For a non-singular Calabi-Yau threefold  $Y$ over $\mathbb{C}$ we let
\[
\Hilb^{\beta,n} (Y) = \Big\{ Z\subset Y\,\,\Big|\,\, [Z] = \beta\in H_2(Y), n=\chi(\O_Z) \Big\}
\]
be the Hilbert scheme of one dimensional proper subschemes with fixed homology class and holomorphic Euler characteristic. We can define the $(\beta,n)$ \textit{Donaldson-Thomas invariant} of $Y$ by:
\[
\DT_{\beta,n}(Y) = 1\cap [\Hilb^{\beta,n} (Y)]^{\mathrm{vir}}.
\]

Behrend proved the surprising result in \cite{Beh} that the Donaldson-Thomas invariants are actually weighted Euler characteristics of the Hilbert scheme: 
\[
\DT_{\beta,n}(Y) =  e(\Hilb^{\beta,n} (Y), \nu) := \sum_{k\in \mathbb{Z}} k \cdot e(\nu^{-1}(k)). 
\]
Here $\nu: \Hilb^{\beta,n}(Y) \rightarrow \mathbb{Z}$ is a constructible function called the \textit{Behrend function} and its values depend formally locally on the scheme structure of $\Hilb^{\beta,n}(Y)$ \cite{Jiang-Motivic}. We also define the \textit{unweighted Donaldson-Thomas} invariants to be:
\[
\uDT_{\beta,n}(Y) =  e(\Hilb^{\beta,n} (Y)). 
\]
These are often closely related to Donaldson-Thomas invariants and their calculation provides insight to the structure of the threefold. Moreover, many important properties of Donaldson-Thomas invariants such as the PT/DT correspondence and the flop formula also hold for the unweighted case \cite{Toda1,Toda2}.\\

The depth of Donaldson-Thomas theory is often not clear until one assembles the invariants into a partition function. Let $\{ C_1, \ldots, C_N\}$ be a basis for $H_2(Y, \mathbb{Z})$, chosen so that if $\beta\in H_2(Y, \mathbb{Z})$ is effective then $\beta =d_1C_1 + \cdots +d_N C_N$ with each $d_i \geq 0$. The Donaldson-Thomas partition function of $Y$ is:
\begin{align*}
Z(Y) :=& \sum_{\beta\in H_2(Y, \mathbb{Z})} \sum_{n\in\mathbb{Z}} \DT_{\beta,n}(Y) Q^\beta p^n\\
:=& \sum_{d_1,\ldots,d_N \geq 0} \sum_{n\in\mathbb{Z}} \DT_{(\sum_i d_iC_i ),n}(Y) Q_1^{d_1}\cdots Q_1^{d_N}p^n.
\end{align*}
We also define the analogous partition function $\widehat{Z}$ for the unweighted Donaldson-Thomas invariants.\\

\begin{remark}
This choice of variable is not necessarily the most canonical as shown in \cite{Bryan_Banana} where the variable $p$ is substituted for $-p$. However, in this article we will be focusing on the unweighted Donaldson-Thomas invariants where this choice makes the most sense. 
\end{remark}

The Donaldson-Thomas partition function is \textit{very} hard to compute. Indeed, for proper Calabi-Yau threefolds, the only known examples of a complete calculation are in computationally trivial cases. However, when we restrict our attention to subsets of $H_2(Y, \mathbb{Z})$ there are many remarkable results. Two interesting cases which are related to the computations in this article are the Schoen (Calabi-Yau) threefold of \cite{Schoen} and the banana (Calabi-Yau) threefold of \cite{Bryan_Banana}.\\

We will employ computational techniques developed in \cite{BK} for studying Donaldson-Thomas theory of local elliptic surfaces. 

\subsection{Donaldson-Thomas Theory of Banana Threefolds}\label{banana_definition_section}

\begin{figure}
  \centering
      \includegraphics[width=0.37\textwidth]{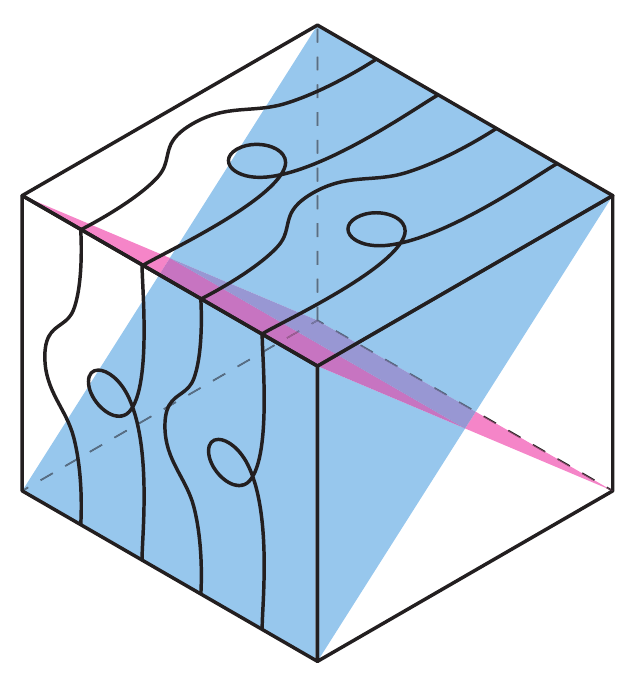}\hspace{1cm}
       \includegraphics[width=0.37\textwidth]{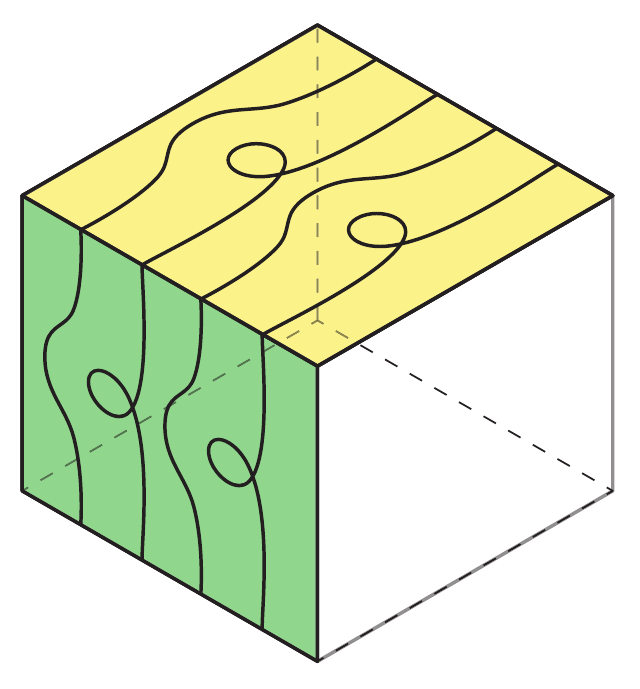}
      \vspace{-1.25em}
  \caption[The banana threefold]{A visual representation of the banana threefold. On the left the diagonal $S_{\Delta}$ and the anti-diagonal $S_{\mathsf{op}}$ are highlighted. On the right the two rational elliptic surfaces $S_{1}$ and $S_{2}$ are highlighted.}
\end{figure}

The banana threefold is of primary interest to us and is defined as follows. Let $\pi: S \rightarrow \P^1$ be a generic rational elliptic surface with a section $\zeta: \P^1 \rightarrow S$. We will take $S$ to be $\P^2$ blown-up at 9 points and $\pi$ given by a generic pencil of cubics. This gives rise to 9 natural choices for $\zeta$ and we choose one. The associated \textit{banana threefold} is the blow-up
\[
X := \mathrm{Bl}_{\Delta} (S\times_{\P^1} S) \num \label{banana_X_definition}
\]
where $\Delta$ is the diagonal divisor in $S\times_{\P^1} S$. The surface $S$ is smooth but the morphism $\pi:S \rightarrow \P^1$ is not. It is singular at 12 points of $S$ which are the nodes of the nodal fibres of $\pi$. This gives rise to 12 conifold singularities of $S\times_{\P^1} S$ that all lie on the divisor $\Delta$. It also, makes $X$ a conifold resolution of $S\times_{\P^1} S$. $X$ is a non-singular simply connected proper Calabi-Yau threefold  \cite[Prop. 28]{Bryan_Banana}.  \\

There is a natural projection $\mathrm{pr}: X\rightarrow \P^1$ and a unique section $\sigma: \P^1\rightarrow X$ arising canonically from $\zeta$.
The generic fibres of the map $\mathrm{pr}: X\rightarrow \P^1$ are Abelian surfaces of the form $E\times E$ where $E=\pi^{-1}(x)$ is the elliptic curve given by the fibre of a point $x\in \P^1$. The projection map $\mathrm{pr}$ also has 12 singular fibres which are non-normal toric surfaces. They are each a compactification of $\C^* \times \C^*$ by a reducible singular curve called a \textit{banana configuration} (c.f. Definition \ref{banana_definition}). Furthermore, the normalisation of a singular $\mathrm{pr}^{-1}(x)$  is isomorphic to $\P^1\times \P^1$ blown up at 2 points \cite[Prop. 24]{Bryan_Banana}.\\

The rational elliptic surface $\pi: S \rightarrow \P^1$, together with the section $\zeta: \P^1 \rightarrow S$, is a Weierstrass fibration. This means that there is a consistent way of choosing Weierstrass coordinates for each fibre (see \cite[III.1.4]{Miranda_Rational}). Thus we have an involution $\iota:S\rightarrow S$ which gives rise to a canonical group law on each fibre where the identity is defined by $\zeta$ and the inverse defined by $\iota$. \\

We will fix  four natural divisors of $X$ for the remainder of this article. The first two arise from considering the natural projections $\mathrm{pr}_i:X \rightarrow S$ and the sections $S_i:S\rightarrow X$ arising from $\zeta$. We denote the corresponding divisors by $S_{1}$ and $S_{2}$. \\

The third and fourth natural divisors of $X$ arise by considering the diagonal $\Delta$ and anti-diagonal $\Delta^{\mathsf{op}}$ (the graph of $\iota$) of $S\times_{\P^1}S$. The anti-diagonal intersects the diagonal in a curve on $\Delta^{\mathsf{op}}$, so it is unaffected by the blow-up. We denote the anti-diagonal divisor in $X$ by $S_{\mathsf{op}}$ and the proper transform of the diagonal by $S_{\Delta}$. The latter is a rational elliptic surface blown-up at the 12 nodal points of the fibres. 

\begin{figure}
  \centering
        \includegraphics[width=0.7\textwidth]{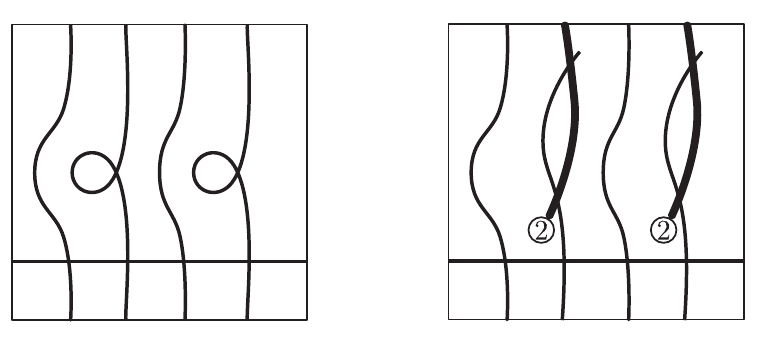}
      \vspace{-1.25em}
  \caption[Natural divisors in the banana threefold]{On the left is a visual representation of the rational elliptic surfaces $S_{1}$, $S_{2}$ and $S_{\mathsf{op}}$. On the right is the diagonal surface  $S_{\Delta}$. Note that the exceptional curves in the fibres of the pencil have order 2.}
\end{figure}

\begin{definition} \label{banana_definition}
A \textit{banana configuration} is a union of three curves $C_1 \cup C_2 \cup C_3$ where $C_i \cong \P^1$ with $N_{C_i/X} \cong \O(-1) \oplus \O(-1)$ and $C_1 \cap C_2 = C_1 \cap C_3 = C_2 \cap C_3 = \{z_1,z_2\}$ where $z_1,z_2\in X$ are distinct points. Also, there exist formal neighbourhoods of $z_1$ and $z_2$ such that the curves $C_i$ become the coordinate axes in those coordinates. We label these curves by their intersection with the natural surfaces in $X$. That is $C_1$ is the unique banana curve that intersects $S_{1}$ at one point. Similarly, $C_2$ intersects $S_{2}$ and $C_3$ intersects $S_{\mathsf{op}}$. 
\end{definition}

\begin{figure}
  \centering
        \includegraphics[width=0.7\textwidth]{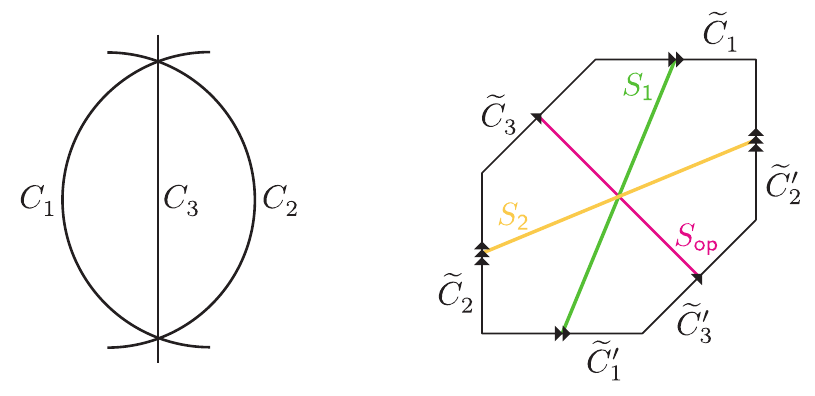}
      \vspace{-1.25em}
  \caption[The banana configuration and the surface containing it]{On the left is a depiction of the banana configuration. On the right is the normalisation of the singular fibre $ F_{\mathsf{ban}}= \mathrm{pr}^{-1}(x)$ with the restrictions of the surfaces $S_{1}$, $S_{2}$, $S_{\mathsf{op}}$.}
\end{figure}

The banana threefold contains 12 copies of the banana configuration. We label the individual banana curves by $C_i^{(j)}$ (and simply $C_i$ when there is no confusion or distinction to be made). We have that $C^{(j_1)}_i \sim C^{(j_2)}_i$ in $H_2(X, \mathbb{Z})$ for each choice of $i,j_1$ and $j_2$. The banana curves $C_1,C_2,C_3$  generate a sub-lattice  $\Gamma_0 \subset H_2(X, \mathbb{Z})$ and we can consider the partition function restricted to these classes:
\[
Z_{\Gamma_0} := \sum_{\beta\in\Gamma_0} \sum_{n\in\mathbb{Z}} \DT_{\beta,n}(X) Q^\beta p^n.
\]
In \cite[Thm. 4]{Bryan_Banana}, this rank three partition function is computed to be:
\[
Z_{\Gamma_0} = \prod_{d_1,d_2,d_3\geq 0} \prod_{k} (1- Q_1^{d_1}Q_2^{d_2}Q_3^{d_3} (-p)^k )^{-12 c(\|\bm{d}\|,k)} \num \label{Bryan_banana_partition_function_main}
\]
where $\bm{d} = (d_1,d_2,d_3)$ and the second product is over $k\in \mathbb{Z}$ unless $\bm{d} = (0,0,0)$ in which case $k>0$. (Note the change in variables from \cite{Bryan_Banana}.) The powers $ c(\|\bm{d}\|,k)$ are defined by
\[
\sum_{a=-1}^\infty \sum_{k\in\mathbb{Z}} c(a,k)Q^a y^k := \frac{\sum_{k\in \mathbb{Z}} Q^{k^2} (-y)^k}{\left(\sum_{k\in \mathbb{Z}+\frac{1}{2}} Q^{2k^2} (-y)^k\right)^2} = \frac{\vartheta_4(2\tau,z)}{\vartheta_1(4\tau,z)^2}
\]
such that $\|\bm{d}\|:= 2d_1d_2 +2 d_2d_3 +2d_3d_1 -d_1^2 -d_2^2-d_3^2$, while $\vartheta_1$ and $\vartheta_4$ are Jacobi theta functions with change of variables $Q=e^{2\pi i \tau}$ and $y=e^{2\pi i z}$.\\

\begin{remark}
 The calculation of (\ref{Bryan_banana_partition_function_main}) uses a motivic method where the values of the Behrend function are explicitly calculated at the contributing points \cite[Prop. 23]{Bryan_Banana}. By removing these weights we can calculate the unweighted partition function $\widehat{Z}_{\Gamma_0}$ directly. In this case, removing the weights corresponds to the change of variables $Q_i \mapsto -Q_i$ and $p\mapsto -p$ in the Donaldson-Thomas partition function. \\
\end{remark}

We can include the class of the section $\sigma$ to generate a larger sub-lattice $\Gamma \subset H_2(X, \mathbb{Z})$. The partition function of this sub-lattice is currently unknown. The purpose of this article is to make progress towards understanding this partition function. We will be calculating the unweighted Donaldson-Thomas theory in the classes:
\[
\beta = \sigma + (0,d_2, d_3) := \sigma + 0\,C_1 +d_2 \,C_2 + d_3\,C_3,
\]
by computing the following partition function
\[
\widehat{Z}_{\sigma + (0,\bullet,\bullet)} :=   \sum_{d_2,d_3\geq 0}\sum_{n\in \mathbb{Z}} \uDT_{\beta,n}(Y) Q_2^{d_2}Q_3^{d_3} p^n,
\]
which we give in terms of the MacMahon functions $M(p,Q) = \prod_{m>0} (1-p^m Q)^{-m}$ and their simpler version $M(p) = M(p,1)$. \\

\begin{mainTheorem}\label{main_DT_calc_theorem_A}
The above unweighted Donaldson-Thomas functions are:
\begin{enumerate}
\item[] $\widehat{Z}_{\sigma + (0,\bullet,\bullet)}$ is:
\[
\widehat{Z}_{(0,\bullet,\bullet)} \frac{ p}{(1-p)^2} \prod_{m>0} \frac{1}{(1-Q_2^m  Q_3^m)^{8}(1-pQ_2^m  Q_3^m)^{2}(1-p^{-1}Q_2^m  Q_3^m)^{2}}
\]
where $\widehat{Z}_{(0,\bullet,\bullet)}$ is the $Q_1^0$ part of the unweighted version of the $\Gamma_0$ partition function (\ref{Bryan_banana_partition_function_main}) and is given by:
\[
M(p)^{24} \prod\limits_{d>0} \dfrac{M(p, Q_2^d Q_3^d)^{24}}{(1-Q_2^d  Q_3^d)^{12} M(p, -Q_2^{d-1}Q_3^d)^{12} M(p, -Q_2^d Q_3^{d-1})^{12}}.\\
\]
\end{enumerate}
\end{mainTheorem}

In the following corollary, the connected unweighted \textit{Pandharipande-Thomas} version of the above formula is identified as the connected version of the Pandharipande-Thomas theory for a rational elliptic surface \cite[Cor. 2]{BK}. 

\vspace{0.2cm}
\begin{mainCorollary} \label{PT_sigma+(0,bullet,bullet)}
The connected unweighted Pandharipande-Thomas partition function is:

\vspace{-0.5cm}
\begin{align*}
\widehat{Z}^{\mathsf{PT},\mathsf{Con}}_{\sigma + (0,\bullet,\bullet)}  
&:=\log \left( \frac{\widehat{Z}_{\sigma +(0,\bullet,\bullet)}}{\widehat{Z}_{ (0,\bullet,\bullet)}|_{Q_i=0}} \right)\\
&= \frac{-p}{(1-p)^2}  \prod_{m>0} \frac{1}{(1-Q_2^m  Q_3^m)^{8}(1-pQ_2^m  Q_3^m)^{2}(1-p^{-1}Q_2^m  Q_3^m)^{2}}. \\
\end{align*}
\end{mainCorollary}

We will also be computing the unweighted Donaldson-Thomas theory in the classes:
\[
\beta = b \,\sigma +  (0,0, d_3), \hspace{0.5cm}\beta = b\,\sigma +  (0,1, d_3) \hspace{0.5cm}\mbox{and}\hspace{0.5cm} \beta = b\,\sigma +  (1,1, d_3)
\]
and the permutations involving $C_1$, $C_2$. So for $i,j\in\{0,1\}$ we define 
\[
\widehat{Z}_{\bullet\sigma + (i,j,\bullet)}:= \sum_{b,d_3\geq 0}\sum_{n\in \mathbb{Z}} \uDT_{\beta,n}(Y) Q_\sigma^b  Q_3^{d_3} p^n.
\]
The formulas  will be given in terms of the functions which are defined for $g\in \mathbb{Z}$:
\[
\psi_g = \psi_g(p) := \left(p^{\frac{1}{2}} - p^{-\frac{1}{2}}\right)^{2g-2} = \left(\frac{p} {(1-p)^{2}}\right)^{1-g}.\\
\]

\vspace{0.2cm}
\begin{mainTheorem}\label{main_DT_calc_theorem_B}
The above unweighted Donaldson Thomas functions are:
\begin{enumerate}
\item $\widehat{Z}_{\bullet \sigma + (0,0,\bullet)}$ is:
\[
 M(p)^{24}  \prod\limits_{m>0}(1+p^m Q_{\sigma})^{m}(1+p^m Q_3)^{12m}.
\]
\item $\widehat{Z}_{\bullet \sigma + (0,1,\bullet)} = \widehat{Z}_{\bullet \sigma + (1,0,\bullet)}$ is:
\[
 \widehat{Z}_{\bullet \sigma + (0,0,\bullet)}  \cdot\Big(\big(12\psi_0 +Q_3 (24\psi_0+12\psi_1)  +Q_3^2 (12\psi_0) \big) + Q_\sigma Q_3 \big(12\psi_0 +2\psi_1\big) \Big)
\]
\item $\widehat{Z}_{\bullet \sigma + (1,1,\bullet)}$ is:
 \begin{align*}
 &\widehat{Z}_{\bullet \sigma + (0,0,\bullet)}   \\
&\cdot \Big( ~\Big((144\psi_{- 1} + 24\psi_0+12\psi_1) +Q_3 (576\psi_{- 1} +384\psi_{0} +72\psi_1 + 12\psi_2)
\\
&\hspace{0.9cm}
+Q_3^2 (864\psi_{- 1} +720\psi_0 + 264\psi_1 +24 \psi_2) \\
&\hspace{0.9cm}
+Q_3^3(576\psi_{- 1} +384\psi_{0} +72\psi_1 + 12\psi_2) + Q_3^4(144\psi_{- 1} + 24\psi_0+12\psi_1) 
\Big)
\\
&\hspace{0.4cm}
+Q_\sigma\Big(\big( 12\psi_0+2\psi_1\big) + Q_3\big(288 \psi_{- 1}+96 \psi_0+44\psi_1 \big) 
\\
&\hspace{1.7cm}
+ Q_3^2\big(576 \psi_{- 1}+600\psi_0+156\psi_1+24 \psi_2\big)\\
&\hspace{1.7cm}+ Q_3^3\big(288 \psi_{- 1}+96 \psi_0+44\psi_1 \big)  + Q_3^4\big(12\psi_0+2\psi_1\big)  \Big)\\
&\hspace{0.4cm}+Q_\sigma^2Q_3^2\Big(144 \psi_{- 1}+48 \psi_0+4\Big)\Big).\\
\end{align*}
\end{enumerate}
\end{mainTheorem}

The connected unweighted \textit{Pandharipande-Thomas} versions of the formulae in theorem \ref{main_DT_calc_theorem_B} contain the same information but are given in a much more compact form. In fact we can present the same information in an even more compact form using the unweighted \textit{Gopakumar-Vafa} invariants $\widehat{n}^g_{\beta}$ via the expansion
{\small\begin{align*}
&\widehat{Z}^{\mathsf{PT}, \mathsf{Con}}_{\Gamma} (X) \\
&:=~
\sum_{\beta\in \Gamma\setminus\{\bm{0}\}}  \sum_{g\geq 0}  \sum_{m>0}~ \widehat{n}^{g}_{\beta}~ \psi_g(p^m) ~(-Q)^{m\beta} \\
&:=~
\hspace{-0.5em} \sum_{\mbox{\tiny $\begin{array}{c}b, d_1,d_2,d_3\geq 0\\ (b,d_1,d_2,d_3)\neq \bm{0}\end{array}$}} \hspace{-0.5em}  \sum_{g\geq 0}  \sum_{m>0} \widehat{n}^{g}_{(b,d_1,d_2,d_2)}~ \psi_g(p^m) ~(- Q_\sigma)^{mb}(- Q_1)^{md_1}( - Q_2)^{md_2}(- Q_3)^{md_3}. 
\end{align*}}

As noted before, these express the same information as the previous generating functions. For $\beta = (d_1,d_2,d_3)$, these invariants are given in \cite[\S A.5]{Bryan_Banana}. We present the new invariants for $\beta = b\sigma + (i,j,d_3)$ where $b>0$. 

\vspace{0.3cm}
\begin{mainCorollary}\label{GV_bsigma+(i,i,bullet)}
Let $i,j\in \{0,1\}$, $b>0$ and $\beta = b\sigma + (i,j,d_3)$. The unweighted Gopakumar-Vafa invariants $\widehat{n}^g_{\beta}$ are given by:
\begin{enumerate}[label={\arabic*)}]
\item If $b>1$ we have $\widehat{n}^g_{\beta} = 0.$ 
\item If $b=1$ then the non-zero invariants are given in the following table:\\
\end{enumerate}
{
\extrarowsep=0.32em
\begin{tabu}{|@{}X[1.5c]@{}|@{}X[c]@{}|@{}X[c]@{}|@{}X[c]@{}|@{}X[c]@{}|@{}X[c]@{}|@{}X[c]@{}|@{}X[c]@{}|@{}X[c]@{}|}
\hline
\multicolumn{9}{|c|}{Table 1: The non-zero $\widehat{n}^g_{\beta}$ for $\beta = \sigma + (i,j,d_3)$ where $i,j\in\{0,1\}$ and $d_3\geq0$.}\\
\hline
 $(d_1,d_2,d_3)$& $(0,0,0)$ & $(0,1,1)$ & $(1,0,1)$ & $(1,1,0)$ & $(1,1,1)$ & $(1,1,2)$ & $(1,1,3)$ & $(1,1,4)$ \\
\hline
$g=0$ & $1$ & $12$ & $12$ & $12$ & $48$ & $216$  & $48$ & $12$\\
\hline
$g=1$ & $0$ & $2$ & $2$ & $2$ & $44$ & $108$  & $44$ & $2$ \\
\hline
$g=2$ & $0$ & $0$ & $0$ & $0$ & $0$ & $24$  & $0$ & $0$ \\
\hline
\end{tabu}
}
\end{mainCorollary}
\mbox{}

\begin{remark}
We note that the values given only depend on the quadratic form $\|\bm{d}\|:= 2d_1d_2+2d_1d_3 +2d_2d_3 - d_1^2-d_2^2-d_3^3$ appearing in the rank 3 Donaldson-Thomas partition function of \cite[Thm. 4]{Bryan_Banana}. However, there is no immediate geometric explanation for this fact. 
\end{remark}

Corollaries \ref{PT_sigma+(0,bullet,bullet)} and \ref{GV_bsigma+(i,i,bullet)} will be proved in section \ref{connected_section}.



\subsection{Notation}
The main notations for this article have been defined above in section \ref{banana_definition_section}. In particular $X$ will always denote the banana threefold as defined in equation (\ref{banana_X_definition}).

\subsection{Future}
The calculation here is for the unweighted Donaldson-Thomas partition function. However, the method of \cite{BK} also provides a route (up to a conjecture \cite[Conj. 21]{BK})  of computing the Donaldson-Thomas partition function. The following are needed in order to convert the given calculation:
\begin{enumerate}[label={\arabic*)}]
\item A proof showing the invariance of the Behrend function under the $(\C^*)^2$-action used on the strata. 
\item A computation of the dimensions of the Zariski tangent spaces for the various strata. 
\end{enumerate}
A comparison of the unweighted and weighted partition functions of the rank 3 lattice of \cite{Bryan_Banana} reveals the likely differences:
\begin{itemize}
\item[] In the variables chosen in this article one can pass from the unweighted to the weighted partition functions by the change of variables $Q_i \mapsto -Q_i$ and $p\mapsto -p$. 
\end{itemize}
Moreover, the conifold transition formula reveals further insight by a comparison with the Donaldson-Thomas partition function of the Schoen threefold with a single section and all fibre classes. The Donaldson-Thomas theory of the Schoen threefold with a section class was shown in \cite{OP} (via the reduced theory of the product of a $K3$ surface with an elliptic curve) to be given by the weight 10 Igusa cusp form.\\ 

As we mentioned previously the Donaldson-Thomas partition function is very hard to compute. So much so that for proper Calabi-Yau threefolds, the only known complete examples are in computationally trivial cases. This is even true conjecturally and even a conjecture for the rank 4 partition function is highly desirable. The work here shows underlying structures that a conjectured partition function must have. 

\subsection{Acknowledgements}I would like to thank both Jim Bryan and Stephen Pietromonaco for very useful discussions. I would also like to thank the anonymous referee for their extremely valuable comments and suggestions. 

\section{Overview of the Computation}
\subsection{Overview of the Method of Calculation}
We will closely follow the method of \cite{BK} developed for studying the Donaldson-Thomas theory of local elliptic surfaces. However, due to some differences in geometry a more subtle approach is required in some areas. In particular, the local elliptic surfaces have a global action which reduces the calculation to considering only the so-called \textit{partition thickened curves}. \\

Our method is based around the following \textit{continuous map}:
\[
\Cyc:  \Hilb^{\beta,n}(X) \rightarrow \Chow^{\beta} (X).
\]
which takes a  one dimensional subscheme to its $1$-cycle. Here $\Chow^{\beta} (X)$ is the Chow \textit{variety} parametrising $1$-cycles in the class $\beta\in H_2(X,\mathbb{Z})$  (as defined in \cite[Thm. I 3.20]{Kollar_Rational}). The fibres of this map are of particular importance and we denote them by $\HilbCyc^{n}\big(X, \cycle{q}\big)$ where $\cycle{q} \in \Chow_1 (X)$. Each $\HilbCyc^{n}\big(X, \cycle{q}\big)$ is a closed subset of $\Hilb^{n,\beta}(X)$ and hence has the natural structure of a reduced subscheme of $\Hilb^{n,\beta}(X)$ (see \cite[Tag 01J3]{stacks_project} for more details).
\begin{remark}
No such morphism exists in the algebraic category. In fact we note from \cite[Thm. I 6.3]{Kollar_Rational} that there is only a morphism from the semi-normalisation $\Hilb^{\beta,n}(X)^{SN} \rightarrow \Chow^n(X)$. However, the semi-normalisation  $\Hilb^{\beta,n}(X)^{SN}$ is homeomorphic to $\Hilb^{\beta,n}(X)$, which gives rise to the above continuous map. 
\end{remark}

\begin{remark}
While no Hilbert-Chow morphism exists for the Chow \textit{variety}, there is a promising theory of relative cycles developed in \cite[Paper IV]{Rydh_Thesis} which allows for the construction of a morphism of functors. These results were used in \cite{Ricolfi_DT/PT} to study the Donaldson-Thomas theory of smooth curves. 
\end{remark}

Broadly, we will be calculating the Euler characteristics $e\big(\Hilb^{\beta,n}(X)\big)$ using the following method:
\begin{enumerate}[label={\arabic*)}]
\item Push forward the calculation to an Euler characteristic on $\Chow^{\beta} (X)$, weighted by the constructible function $(\Cyc_*1)(\mathfrak{q}):= e\big(\HilbCyc^{n}(X, \cycle{q})\big)$. This is further described in sections \ref{Euler_characteristic_review_section} and \ref{pushing_forward_to_chow_section}.
\item Analyse the image of $\Cyc$ and decompose it into combinations of symmetric products where the strata are based on the types of subscheme in the fibres $\HilbCyc^{n}(X, \cycle{q})$. This is done in section \ref{1-cycles_section}. 
\item Compute the Euler characteristic of the fibres $e\big(\HilbCyc^{n}(X, \cycle{q})\big)$ and show that they form a constructible function on the combinations of symmetric products. This is done in section \ref{main_computation_section}.
\item Use the following lemma to give the Euler characteristic partition function. 
\end{enumerate}

\begin{lemma}\label{BK_sym_product_lemma} \cite[Lemma 32]{BK}
Let $Y$ be finite type over $\C$ and let $g: \mathbb{Z}_{\geq 0} \rightarrow \mathbb{Z}(\!(p)\!)$ be any function with $g(0)=1$. Let $G:\Sym^d(Y)\rightarrow \mathbb{Z}(\!(p)\!)$ be the constructible function defined by 
\[
G(\bm{a x}) = \prod_i g(a_i)
\]
where $\bm{ax} = \sum_i a_i x_i \in \Sym^d(Y)$ and $x_i\in Y$ are distinct points. Then 
\[
\sum_{d=0}^\infty e( \Sym^d(Y), G)  q^d = \left( \sum_{a=0}^\infty g(a) q^a \right)^{e(Y)}
\]
where the $G$-weighted Euler characteristic $e( -, G) $ is defined in equation (\ref{weighted_euler_laurent_poly}).
\end{lemma}

\begin{figure}
\centering
\begin{tabu}{X[C]X[C]}
\includegraphics[width=0.36\textwidth]{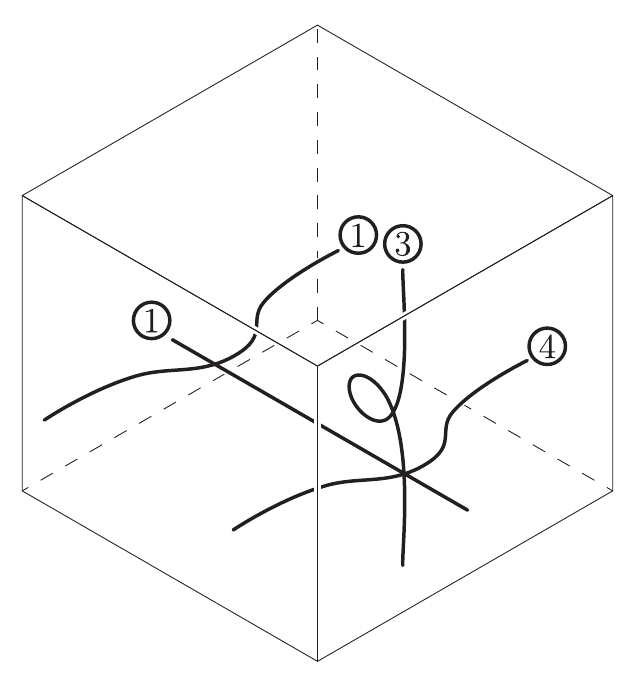}
&
\includegraphics[width=0.36\textwidth]{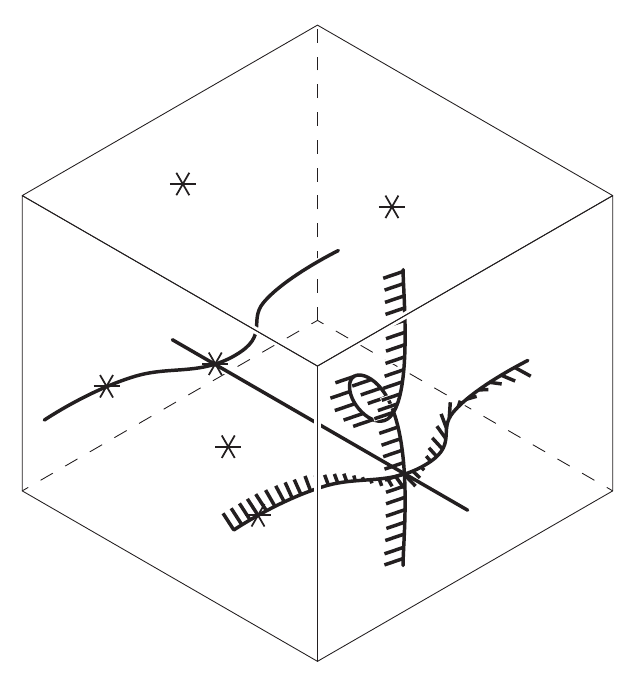}
\\
\includegraphics[width=0.36\textwidth]{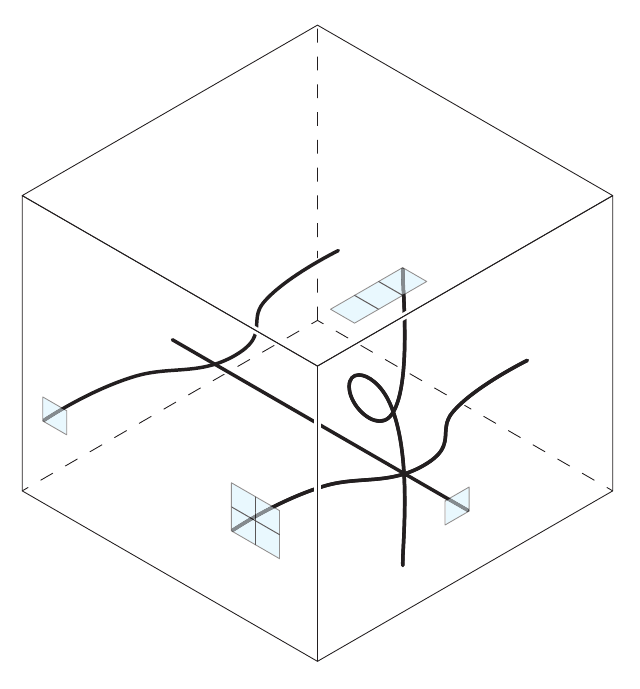}
&
\includegraphics[width=0.36\textwidth]{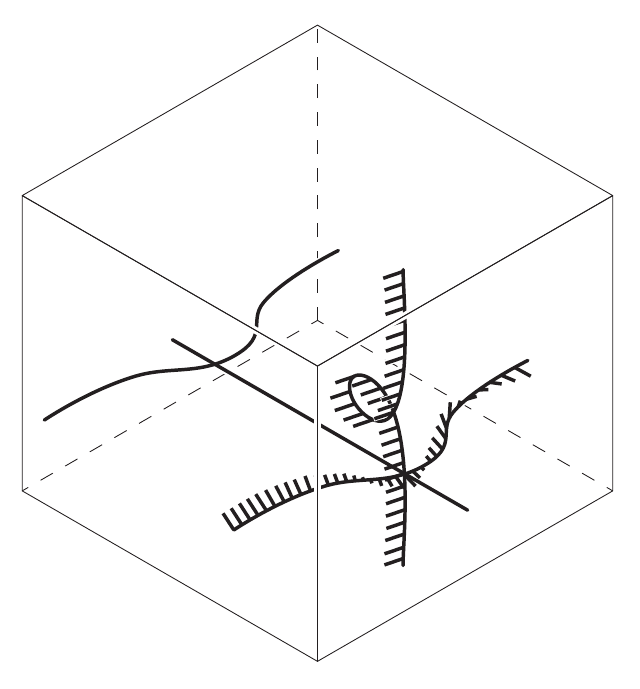}
\end{tabu}
  \caption[The process for reducing to partition thickened curves]{A depiction of the process for reducing to partition thickened curves. Clockwise from the top-left we have: a) Consider a $1$-cycle in the Chow scheme; b) Consider the fibre of the given $1$-cycle; c) Reduce to a computation on the open subset of Cohen-Macaulay subschemes; d) Reduce to a computation on partition thickened schemes. }
\end{figure}

To compute the Euler characteristics of the fibres $(\Cyc_*1)(\cycle{q}):= e\big(\HilbCyc^{n}(X, \cycle{q})\big)$ we use the following method made rigorous in section \ref{theory_computation_section}:
\begin{enumerate}[label={\arabic*)}]
\item Denote the open subset consisting entirely of Cohen-Macaulay subschemes by $\Hilb^{n}_{\mathsf{CM}}(X, \cycle{q}) \subset \HilbCyc^{n}(X, \cycle{q})$, and define the notation $\Hilb^{\blacklozenge}_{\mathsf{CM}}(X, \cycle{q}):=\coprod_{m\in\mathbb{Z}}\Hilb^{m}_{\mathsf{CM}}(X, \cycle{q})$.
\item Consider the constructible map  which takes a subscheme $Z$ to the maximal Cohen-Macaulay subscheme $Z_{\mathsf{CM}}\subset Z$ and denote the constructible map by $\kappa_n: \HilbCyc^{n}(X, \cycle{q}) \rightarrow \Hilb^{\blacklozenge}_{\mathsf{CM}}(X, \cycle{q})$. 
\item Note the equality of the Euler characteristic $e\big(\HilbCyc^{n}(X, \cycle{q})\big)$ and that of the weighted Euler characteristic $e\big(\Hilb^{\blacklozenge}_{\mathsf{CM}}(X, \cycle{q}),(\kappa_n)_*1\big)$ where $(\kappa_n)_*1$ is the constructible function $((\kappa_n)_*1)(\mathsf{p}) := e(\kappa_n^{-1}(\mathsf{p}))$. 
\item Define a $(\C^*)^2$-action on $\Hilb^{\blacklozenge}_{\mathsf{CM}}(X, \cycle{q})$ and show that $\kappa_*1(\mathsf{p}) = \kappa_*1(\alpha\cdot\mathsf{p}) $ meaning $e\big(\HilbCyc^{n}(X, \cycle{q})\big)= e\big(\Hilb^{\blacklozenge}_{\mathsf{CM}}(X, \cycle{q})^{(\C^*)^2},\kappa_*1\big)$. This technique is discussed in section  \ref{action_subsection}. 
\item Identify the $(\C^*)^2$-fixed points $\Hilb^{\blacklozenge}_{\mathsf{CM}}(X, \cycle{q})^{(\C^*)^2}$ as a discrete subset containing partition thickened curves. These neighbourhoods and this action are given explicitly in section \ref{partition_thickened_section}. 
\item Calculate the Euler characteristics $e\big(\Hilb^{\blacklozenge}_{\mathsf{CM}}(X, \cycle{q})^{(\C^*)^2},\kappa_*1\big)$ using the $\Quot$ scheme decomposition and topological vertex method of \cite{BK}. The concept  of this is depicted in figure \ref{vertex_decomp_figure} and described below.  Further technical details are given in  section \ref{Topological_vertex_to_quot_scheme_section}. \\
\end{enumerate}

\begin{figure}
  \centering
      \includegraphics[width=0.95\textwidth]{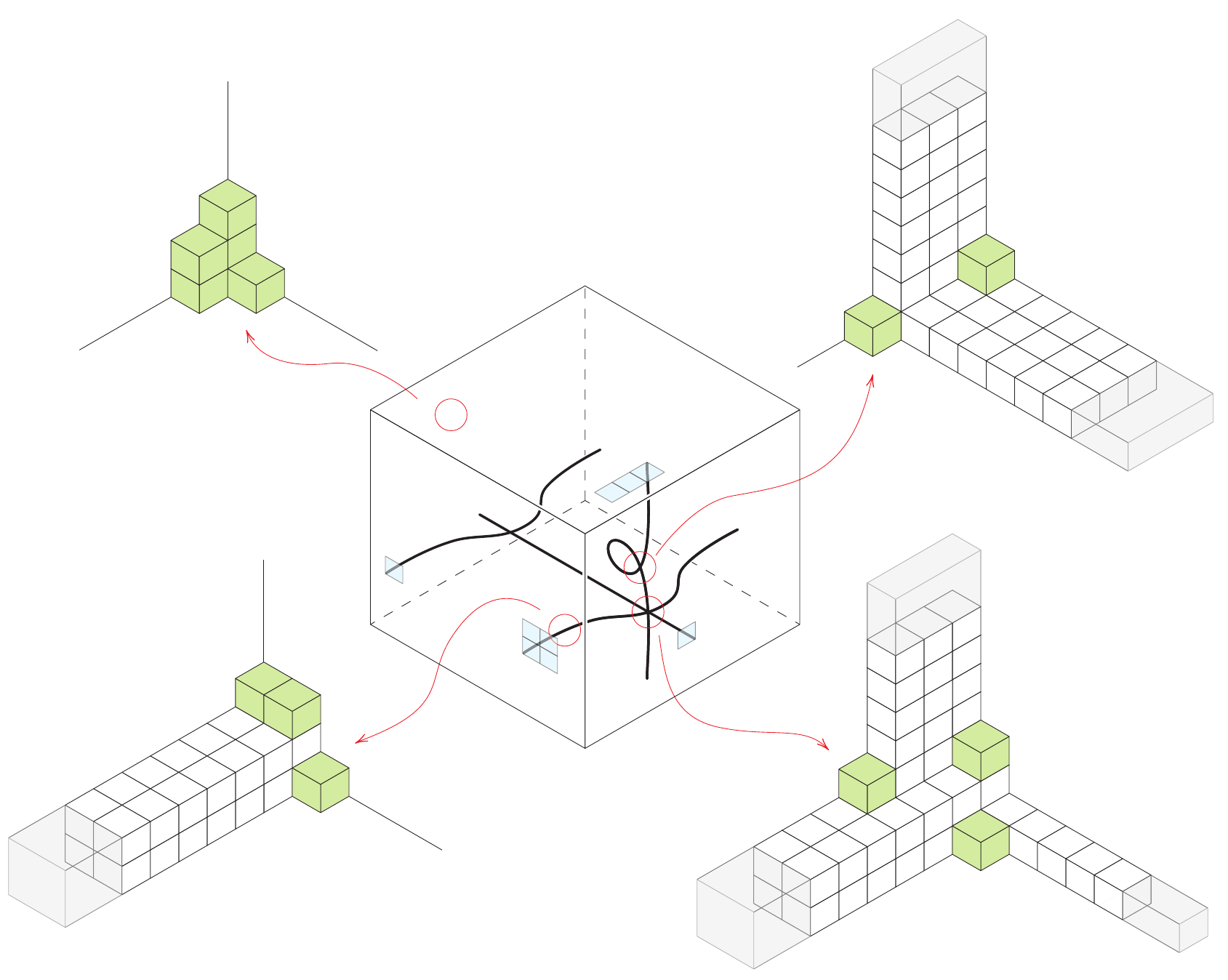}
      \vspace{-1.25em}
  \caption[Euler characteristic from subscheme decompositions]{A depiction of how the topological vertex is applied to calculate Euler characteristic of a given stratum.}\label{vertex_decomp_figure}
\end{figure}

The Euler characteristic calculation of $e\big(\Hilb^{\blacklozenge}_{\mathsf{CM}}(X, \cycle{q})^{(\C^*)^2},\kappa_*1\big)$ for theorems \ref{main_DT_calc_theorem_A}  and \ref{main_DT_calc_theorem_B}  follow similar methods but have different decompositions. 
The calculations are completed by considering the different types of topological vertex that occur for each fixed point in $\Hilb^{m}_{\mathsf{CM}}(X, \cycle{q})^{(\C^*)^2}$ for $m\in\mathbb{Z}$. \\

Since the fixed locus $\Hilb^{m}_{\mathsf{CM}}(X, \cycle{q})^{(\C^*)^2}$ will be a discrete set, we can consider the individual subschemes $C\in \Hilb^{m}_{\mathsf{CM}}(X, \cycle{q})^{(\C^*)^2}$  and their contribution to the Euler characteristic $e\big(\Hilb^{\blacklozenge}_{\mathsf{CM}}(X, \cycle{q})^{(\C^*)^2},\kappa_*1\big)$. To compute the contribution from $C$ we must decompose $X$ as follows:
\begin{enumerate}[label={\arabic*)}]
\item Take the complement $W = X\setminus C$. 
\item Consider, $C^{\diamond}$, the set of singularities of the underlying reduced curve. 
\item Define $C^\circ = C^{\mathsf{red}} \setminus C^\diamond$ to be its complement. 
\end{enumerate}

The curve $C$ will be partition thickened. So each formal neighbourhood of a point $x \in C^{\diamond}$ will give rise to a 3D partition asymptotic to a collection of three partitions (depicted on the righthand side of figure \ref{vertex_decomp_figure}). Similarly, points on $C^\circ $ and $W$ will give rise to 3D partitions asymptotic to collections of three partitions. However, for  $C^\circ $ only one of the three partitions will be non-empty and for $W$ all three partitions will be empty (depicted respectively on the bottom-left and top-left parts of figure \ref{vertex_decomp_figure}). Using techniques from section \ref{Topological_vertex_to_quot_scheme_section} the Euler characteristics can then be determined. \\

This calculation for theorem \ref{main_DT_calc_theorem_A} is finalised in section \ref{main_computation_section_A}. Generalities for the proof of theorem \ref{main_DT_calc_theorem_B} are given in section \ref{main_computation_section_B_prelim} and the individual calculations are given in sections \ref{bs+(0,0,d)_calc_section}, \ref{bs+(0,1,d)_calc_section} and \ref{bs+(1,1,d)_calc_section}.

\subsection{Review of Euler characteristic}\label{Euler_characteristic_review_section}
We begin by recalling some facts about the (topological) Euler characteristic. For a scheme $Y$ over $\C$ we denote by $e(Y)$ the topological Euler characteristic in the complex analytic topology on $Y$. This is independent of any non-reduced structure of $Y$, is additive under decompositions of $Y$ into open sets and their complements, and is multiplicative on Cartesian products. In this way we see that the Euler characteristic defines a ring homomorphism from the Grothendieck ring of varieties to the integers: 
\[
e: K_0(\mathrm{Var}_{\C}) \longrightarrow \mathbb{Z}.
\]
If $Y$ has a $\C^*$-action with fixed locus $Y^{\C^*} \subset Y$ the Euler characteristic also has the property $e(Y^{\C^*}) = e(Y)$ \cite[Cor. 2]{Bia_Fixed}. \\

The interaction of Euler characteristic with constructible functions plays a key role in this article. Recall that a function $\mu: Y\rightarrow \mathbb{Z}(\!(p)\!)$ is constructible if $\mu(Y)$ is finite and $\mu^{-1}(c)$ is the union of finitely many locally closed sets for all non-zero $c\in \mu(Y)$. The $\mu$-weight Euler characteristic is a ring homomorphism
\begin{align}
e(-,\mu): K_0(\mathrm{Var}_{\C}) \longrightarrow \mathbb{Z}(\!(p)\!) \label{weighted_euler_laurent_poly}
\end{align}
defined by $e(Y,\mu) = \sum_{k\in\mathbb{Z}(\!(p)\!)} k\cdot e(\mu^{-1}(k))$. The constant function $1$ and the Behrend function $\nu$ are two canonical examples of constructible functions with images in $\mathbb{Z}\subset \mathbb{Z}(\!(p)\!)$.  Moreover, the usual Euler characteristic is $e(Y) = e(Y, 1)$ where $1$ is the constant function.  \\

For a scheme  $Z$ over $\C$, a constructible map $f:Y\rightarrow Z$ is a finite collection of continuous functions $f_i: Y_i\rightarrow  Z_i$ where $Y=\coprod_{i} Y_i$ is a decomposition into locally closed subsets and $Z_i \subseteq Z$. A constructible homeomorphism is a constructible map such that each $f_i$ is a homeomorphism and $Z=\coprod_{i} Z_i$ is a decomposition into locally closed subsets.  When $f:Y\rightarrow Z$ is a constructible map we define the constructible function $f_*\mu : Z \rightarrow \mathbb{Z}(\!(p)\!)$ by
\[
(f_*\mu)(x) := e( f^{-1}(x),\mu). 
\]
This has the important property $e(Z,f_* \mu)= e(Y,\mu)$. If $\omega:Z\rightarrow \mathbb{Z}(\!(p)\!)$ is another constructible function, then $\mu\cdot \omega$ is a constructible function on $Y\times Z$ and $e(Y\times Z , \mu\cdot \omega) = e(Y , \mu)\cdot e(Z , \omega)$. \\

It will be useful to consider the rings of formal power series in $Q_i$ and formal Laurent series in $p$ with coefficients in $K_0(\mathrm{Var}_{\C})$. An element $P \in K_0(\mathrm{Var}_{\C})[\![Q_i]\!](\!(p)\!)$ is an \textit{indexed} disjoint union of varieties where the indexing is given by monomials. A constructible function $\mu: (\sum_{d_i}\sum_{n} Q_i^{d_1} p^{n} Y_{d_i,n}) \longrightarrow  \mathbb{Z}(\!(p)\!)$ is an indexed collection of constructible functions $\mu_{d_i,n}: Y_{d_i,n} \longrightarrow  \mathbb{Z}(\!(p)\!)$. Moreover, we extend Euler characteristic to a ring homomorphism 
\[
e(-,\mu): K_0(\mathrm{Var}_{\C})[\![Q_i]\!](\!(p)\!) \longrightarrow \mathbb{Z}[\![Q_i]\!](\!(p)\!) 
\]
preserving the indexing $e (\sum_{d_i}\sum_{n} Q_i^{d_1} p^{n} Y_{d_i,n},\mu) := \sum_{d_i}\sum_{n} Q_i^{d_1} p^{n}e( Y_{d_i,n},\mu_{d_i,n})$. \\

Lastly, we extend the definition of constructible map to formal series of varieties in two different ways. Firstly, a constructible map 
$
f: \sum_{n} p^{n} Y_{n} \longrightarrow Z
$
is an indexed collection of constructible maps $f_n: Y_n\rightarrow Z$, and secondly, a constructible map
$
g: \sum_{d_i}\sum_{n} p^{n} Y_{d_i, n} \longrightarrow \sum_{d_i} Z_{d_i}
$
is an indexed collection of constructible maps $g_{d_i}: \sum_{n} Y_{d_i,n}\rightarrow Z_{d_i}$.  We can now also define the push-fowards of constructible functions as before.

\subsection{Pushing Forward to the Chow Variety}\label{pushing_forward_to_chow_section}

Recall that the Chow variety $\Chow^\beta(X)$ is a space parametrising the one dimensional cycles of $X$ in the class $\beta\in H_2(X,\mathbb{Z})$. We then have a constructible map 
\[
\rho_\beta:\sum_n p^n \Hilb^{\beta,n}(X) \rightarrow \Chow^\beta (X).
\]
The strategy for calculating the partition functions is to analyse $\Chow^\beta(X)$ and the fibres of the map $\rho_\beta$. These will often involve the symmetric product, and where possible we will apply lemma \ref{BK_sym_product_lemma}.\\

It will be convenient to employ the following $\bullet$-notations for the Hilbert schemes
\begin{align*}
\Hilb^{\sigma + (0,\bullet, \bullet),\bullet}(X) &:=  \sum_{d_2,d_3\geq0} \sum_{n\in\mathbb{Z}}  Q_\sigma Q_2^{d_2} Q_3^{d_3} p^n \, \Hilb^{\sigma + (0,d_2, d_3),n}(X)\\
\Hilb^{\bullet \sigma + (i,j, \bullet),\bullet}(X) &:=  \sum_{b,d_3\geq0} \sum_{n\in\mathbb{Z}} Q_\sigma^{b}Q_1^i Q_2^j Q_3^{d_3} p^n \, \Hilb^{b \sigma + (i,j, d_3),n}(X)
\end{align*}
 and for the Chow varieties
\begin{align*}
\Chow^{\sigma + (0,\bullet, \bullet)}(X) &:= \sum_{d_2,d_3\geq 0} Q_\sigma Q_2^{d_2} Q_3^{d_3} \Chow^{\sigma + (0,d_2, d_3)}(X)\\
\Chow^{\bullet \sigma + (i,j, \bullet)}(X) &:= \sum_{b,d_3\geq 0}Q_\sigma^{b}Q_1^i Q_2^j Q_3^{d_3} \Chow^{b \sigma + (i,j, d_3)}(X)
\end{align*}
where we have viewed the Hilbert scheme and Chow variety as elements in the Grothendieck ring of varieties, and $i,j \in \{0,1\}$.  The notation $\cycle{q}\in\Chow^{\sigma + (0,\bullet, \bullet)}(X) $ and $\cycle{q}\in\Chow^{\bullet \sigma + (i,j, \bullet)}(X)$ will denote  $\cycle{q}\in\Chow^{\beta}(X)$ for some $\beta\in H_2(X,\mathbb{Z})$. This is what we will mean by the ``points'' of $\Chow^{\sigma + (0,\bullet, \bullet)}(X) $ and $\Chow^{\bullet \sigma + (i,j, \bullet)}(X)$. $\cycle{q}$ will often be given without any associated monomial since that is usually implicitly understood. The $\bullet$-notation is extended to symmetric products by

\vspace{-0.3cm}
\[
\Sym_Q^\bullet( Y) := \sum_{n\in \mathbb{Z}_{\geq 0}} Q^n \Sym^n(Y),
\]
and we use the following notation for elements of the symmetric product
\[
\bm{a}\bm{y} := \sum_i a_i y_i \in \Sym^n( Y) 
\]
where $y_i$ are distinct points on $Y$ and $a_i\in \mathbb{Z}_{\geq0}$. We also denote a tuple of partitions $\bm{\alpha}$ of a tuple of non-negative integers $\bm{a}$ by $\bm{\alpha}\vdash \bm{a}$. As was the case with the Chow variety, we think of  $\bm{a}\bm{y}$ as being a point in $\Sym_Q^\bullet( Y)$.  \\

Using the $\bullet$-notation for the maps $\rho_\beta$ we create the following constructible maps:

\vspace{-0.5cm}
\begin{gather*}
\rho_\bullet: \Hilb^{\sigma + (0,\bullet, \bullet),\bullet}(X) \longrightarrow \Chow^{\sigma + (0,\bullet, \bullet)}(X)\\
\eta^{ij}_\bullet: \Hilb^{\bullet \sigma + (i,j, \bullet),\bullet}(X) \longrightarrow  \Chow^{\bullet \sigma + (i,j, \bullet)}(X)
\end{gather*}
and we also use the notation $\eta_\bullet = \eta^{00}_\bullet + \eta^{01}_\bullet + \eta^{11}_\bullet$. 
The fibres of these maps will be formal sums of subsets of the Hilbert schemes parametrising one dimensional subschemes with a fixed $1$-cycle. Specifically, let $C\subset X$ be a one dimensional subscheme in the class $\beta\in H_2(X)$ with $1$-cycle $\Cyc(C)$. Define $\Hilb^{n}(X,\Cyc(C))\subset \Hilb^{\beta,n}(X)$ to be the closed subset 
\[
\Hilb^{n}(X,\Cyc(C))  = \big\{ Z\in \Hilb^{\beta,n}(X) ~\big|~  \Cyc(Z)=\Cyc(C) \big\}.
\]
The maps $\rho_\bullet$ and $\eta_\bullet$ are explicitly described in lemmas \ref{(1,0,d1,d2)_Chow_lemma}  and \ref{classes_bs+(i,j,d)_main_chow_main_lemma} respectively.

\section{Parametrising Underlying \texorpdfstring{$1$}{1}-cycles} \label{1-cycles_section}

\subsection{Related Linear Systems in Rational Elliptic Surfaces}
In this section we consider some basic results about linear systems on a rational elliptic surface. Some of these result can be found in \cite[\S A.1]{BK}. \\

Recall our notation that $\pi: S \rightarrow \P^1$ is a generic rational elliptic surface with a canonical section $\zeta: \P^1 \rightarrow S$. Consider the following classical results for rational elliptic surfaces from \cite[II.3]{Miranda_Rational}:
\[
\pi_* \O_S \cong \pi_* \O_S(\zeta) \cong \O_{\P^1}
,\hspace{0.7cm}
R^1\pi_* \O_S \cong  \O_{\P^1}(-1)
\hspace{0.7cm}\mbox{and}\hspace{0.7cm}
R^1\pi_* \O_S(\zeta) \cong  0.
\]
After applying the projection formula we have the following:
\begin{align}
\pi_* \O_{S}(dF) \cong \pi_* \O_{S}(\zeta + dF)  \cong \O_{\P^1}(d)  \label{projection_formula_results}
\end{align}
as well as
\begin{align}
R^1\pi_* \O_{S}(dF) \cong \O_{\P^1}(d-1)
\hspace{0.7cm}\mbox{and}\hspace{0.7cm}
R^1\pi_* \O_{S}(\zeta +dF) \cong 0. \label{projection_formula_results_2}
\end{align}

\begin{lemma}\label{H1_lemma}
We have the following isomorphisms:
\[
H^1(S,\O_{S}(dF)) \cong H^0( \P^1,  \O_{\P^1}(d-1))
\hspace{0.7cm}\mbox{and}\hspace{0.7cm}
H^1(S,\O_{S}(\zeta+ dF)) \cong 0.
\]
\end{lemma}
\begin{proof}
The second isomorphism is immediate from the vanishing of $R^i\pi_* \O_{S}(\zeta +dF)$ for $i>0$ (see for example \cite[III Ex. 8.1]{Hartshorne}) and $H^0(\P^1, \O_{\P^1}(d) ) \cong 0$. \\

To show the first isomorphism we consider the following exact sequence arising from the  Leray spectral sequence:
\[
H^1( \P^1, \pi_* \O_{S}(dF) ) \rightarrow H^1(S,\O_{S}(dF)) \rightarrow H^0( \P^1, R^1 \pi_*\O_{S}(dF) ) \rightarrow 0
\]
We have from (\ref{projection_formula_results}) that $H^1( \P^1, \pi_* \O_{S}(dF) ) \cong 0$ and we have the desired isomorphism after considering (\ref{projection_formula_results_2}). 
\end{proof}

\begin{lemma} \label{S_linear_system_lemma}
Consider a fibre $F$ of a point $z\in\P^1$ by the map $S\rightarrow \P^1$ and the image of a section $\zeta:\P^1 \rightarrow S$. Then there are isomorphisms of the linear systems 
\[
| d\,F |_{S} \cong |\zeta + d\,F |_{S} \cong | d\,z |_{\P^1} \hspace{0.7cm}\mbox{and}\hspace{0.7cm} |b\,\zeta + F |_{S} \cong |  z |_{\P^1}. 
\] 
\end{lemma}
\begin{proof}
The isomorphism $ |\zeta + d\,F |_{S} \cong | d\,z |_{\P^1}$  is immediate from the vanishing of $R^i\pi_* \O_{S}(\zeta +dF)$ for $i>0$  and (\ref{projection_formula_results}) (see for example \cite[III Ex. 8.1]{Hartshorne}).\\

We continue by showing $| d\,F |_{S} \cong |\zeta + d\,F |_{S} $. Consider the long exact sequence arising from the divisor sequence for $\zeta$ twisted by $\O_{S}(\zeta +dF)$:
\begin{align*}
0 &\rightarrow H^0(S, \O_{S}(dF)) \overset{f}{\rightarrow}  H^0(S,\O_{S}(\zeta +dF))\overset{}{\rightarrow} H^0(S,\zeta_* \O_{\P^1}(\zeta +dF))\\
&\overset{g}{\rightarrow}  H^0( \P^1,  \O_{\P^1}(d-1)) \rightarrow  0
\end{align*}
where we have applied the results from lemma \ref{H1_lemma}.  From intersection theory we have that $\zeta_* \O_{\P^1}(\zeta +dF) \cong \zeta_* \O_{\P^1}(d-1)$. So $g$ and hence $f$ are isomorphisms.\\

The isomorphism $ |b\,\zeta + F |_{S} \cong |  z |_{\P^1}$ will follow inductively from the divisor sequence for $\zeta$ on $S$:
\[
0\longrightarrow \O_S (k\zeta +F) \longrightarrow \O_S \big( (k+1)\zeta + F\big) \longrightarrow \O_{\zeta}\big( (k+1)\zeta + F\big)\longrightarrow 0.
\]
Intersection theory shows us that $\O_{\zeta}\big( (k+1)\zeta + F\big)$ is a degree $-k$ line bundle on $\P^1$ which shows that its $0$th cohomology vanishes. Hence, we have isomorphisms:
\[
H^0\big(S, \O_S (F) \big)\cong \cdots \cong H^0\big(S, \O_S (b \zeta + F) \big).
\]
\end{proof}

\subsection{Curve Classes and \texorpdfstring{$1$}{1}-cycles in the Threefold}

Recall from definition \ref{banana_definition} that the banana curves $C_i$ are labelled by their unique intersections with the rational elliptic surfaces
\[
S_{1}, \hspace{0.5cm} S_{2} \hspace{0.5cm}\mbox{and}\hspace{0.5cm} S_{\mathsf{op}}. 
\]
These are smooth effective divisors on $X$. Hence a curve $C$ in the class $(d_1,d_2,d_3)$ will have the following intersections with these divisors:
\[
C \cdot S_{1}=d_1, \hspace{0.5cm} C \cdot S_{2}=d_2 \hspace{0.5cm}\mbox{and}\hspace{0.5cm} C \cdot S_{\mathsf{op}}=d_3.
\]

The full lattice $H_2(X,\mathbb{Z})$ is generated by
\[
C_1,~ C_2,~ C_3,~ \sigma_{11}, ~\sigma_{12},~\ldots~,\sigma_{19},~\sigma_{21},~\ldots~,~ \sigma_{99}
\]
where the $\sigma_{ij}$ are the 81 canonical sections of $\mathrm{pr}:X\rightarrow \P^1$ arising from the 9 canonical sections of $\pi:S\rightarrow \P^1$. However, there are 64 relations between the $\sigma_{ij}$'s giving the lattice rank of 20 (see \cite[Prop. 28 and Prop. 29]{Bryan_Banana}).

\begin{lemma}
There are no relations in $H_2(X,\mathbb{Z})$ of the form:
\[
n\cdot \sigma_{i,j} +  d_1C_1 + d_2C_2 + d_3 C_3 = \sum_{(k,l)\neq (i,j)}  a_{k,l}\cdot \sigma_{k,l} +d'_1C_1 + d'_2C_2 + d'_3 C_3
\]
where $n, a_{k,l}, d_t,d'_t \in \mathbb{Z}_{\geq 0}$ for all $k,l\in\{1,\ldots,9\}$ and $t\in\{1,2,3\}$. 
\end{lemma}
\begin{proof}
Any such relation must push forward to relations on $S$ via the projections $\mathrm{pr}_i:X\rightarrow S_i$. However, $S$ is isomorphic to $\P^2$ blown up at 9 points. The exceptional divisors of these blow-ups correspond to the sections $\zeta_i : \P^1 \rightarrow S$. Hence
\[
\mathrm{Pic} \,S \cong \mathrm{Pic}\, \P^2 \times \zeta_1 \times \cdots \times \zeta_9 \cong \mathbb{Z}^{10}
\]
and there are are no relations of this form. 
\end{proof}

The next lemma allows us to consider the curves in our desired classes by decomposing them. 

\begin{lemma}\label{curve_decomposition_lemma} Let $d_1,d_2,d_3, b \in \mathbb{Z}_{\geq 0}$ and $i,j\in \{0,1\}$. 
\begin{enumerate}[label={\arabic*)}]
\item  \label{curve_decomposition_lemma_no_sigma}
Let $C$ be a Cohen-Macaulay curve in the class $ (d_1,d_2, d_3)$. Then the support of $C$ is contained in fibres of the projection map $\mathrm{pr}:X \rightarrow \P^1$. 

\item \label{curve_decomposition_lemma_sigma}
A curve $C$ in the class $\sigma + (d_1,d_2, d_3)$ is of the form 
\[
C = \sigma \cup C_0
\]
where $C_0$ is a curve in the class $(d_1,d_2, d_3)$. 
\item \label{curve_decomposition_lemma_bsigma}
 A curve in the class $b\,\sigma + (i,j,d_3)$ is of the form
\[
C = C_\sigma \cup C_0
\]
where $C_\sigma$ is a curve in the class $b\,\sigma$ and $C_0$ is a curve in the class $(i,j,d_3)$. The same result holds for permutations of $b\,\sigma + (i,j,d_3)$. 
\end{enumerate}
\end{lemma}
\begin{proof}
Consider a curve in one of the given classes and its image under the two projections $\mathrm{pr}_i:X\rightarrow S_i$. For (1) these must be in the classes $|d_1 f_1|$ and $|d_1 f_1|$, for (2) the classes $|\zeta + d_1 F_1|$ and $|\zeta + d_2 F_2|$, and for (3) the classes $|i f_1|$ and $|j f_1|$. Lemma \ref{S_linear_system_lemma} now shows that the curve must have the given form. 
\end{proof}

\subsection{Analysis of \texorpdfstring{$1$}{1}-cycles in Smooth Fibres of \texorpdfstring{$\mathrm{pr}$}{pr}}

Consider a fibre $F_x = \mathrm{pr}^{-1}(x)$ which is smooth. Then there is an elliptic curve $E$ such that $F_x \cong E\times E$. Consider a curve $C$ with underlying 1-cycle contained in $E\times E$, then this gives rise to a divisor $D$ in $E\times E$. Hence we must analyse divisors in $E\times E$ and their classes in $X$. The class of such a curve is determined uniquely by its intersection with the surfaces $S_{1} ,S_{2}$ and $S_{\mathsf{op}}$.

\begin{lemma}\label{ExE_curve_class_lemma1}
Let $C \subset X$ correspond to a divisor $D$ in $E\times E$.
\begin{enumerate}[label={\arabic*)}]
\item \label{ExE_curve_class_lemma1_(0,d_2,d_3)} If $C$ is in the class  $(0,d_2,d_3)$  then $d_2 = d_3$ and $D$ is the pullback of a degree $d_2$ divisor on $E$ via the projection to the second factor. 
\item The result in \ref{ExE_curve_class_lemma1_(0,d_2,d_3)} is true for $(d_1,0,d_3)$ and projection to the first factor. 
\end{enumerate}
\end{lemma}
\begin{proof}
If $C$ is in the class $(0,d_2,d_3)$ then it doesn't intersect with the surface $S_{1}$. When we restrict to $E\times E$ this is the same condition as not intersecting with a fibre of the projection to the second factor. The only divisors that this is true for are those pulled back from $E$ via the projection to the second factor. A divisor of this form will have intersection with $S_{2}$ of $d_2$ and intersection with $S_{\mathsf{op}}$ of $d_2$. Hence we have that $d_2=d_3$. The proof for part 2) is completely analogous. 
\end{proof}

\begin{lemma}\label{ExE_curve_class_lemma2}
Let $C \subset X$ be in the class $(1,1,d)$ and correspond to a divisor $D$ in $E\times E$. Then $d \in \{0,\ldots,4\}$ and occurs in the following situations:
\begin{enumerate}[label={\arabic*)}]
\item If $E$ has $j(E) \neq 0, 1728$ then:
\begin{enumerate}[label={\alph*)}]
\item $d=0$ occurs when $D$ is a translation of the graph $\{(x,-x)\}$. 
\item $d=4$ occurs when $D$ is a translation of the graph $\{(x,x)\}$.
\item $d=2$ occurs when $D$ is the union of a fibre from the projection to the first factor and a fibre from the projection to the second factor.
\end{enumerate}
\item If $j(E)=1728$ and $E \cong \C/i$ we have the cases a) to c) as well as:
\begin{enumerate}
\item[d)] $d=2$ occurs when $D$ is a translation of the graph $\{(x,\pm i x)\}$. 
\end{enumerate}
\item If $j(E)=0$ and $E \cong \C/\tau$ with $\tau=\frac{1}{2}(1+i \sqrt{3})$ we have the cases a) to c) as well as:
\begin{enumerate}
\item[e)] $d=1$ occurs when $D$ is a translation of the graph $\{(x,-\tau x)\}$ or the graph $\{(x, (\tau-1) x)\}$. 
\item[f)] $d=3$ occurs when $D$ is a translation of the graph $\{(x,\tau x)\}$ or the graph $\{(x, (-\tau+1) x)\}$. 
\end{enumerate}
\end{enumerate}
\end{lemma}
\begin{proof}
Denote the projection maps by $\mathsf{p}_i:E\times E \rightarrow E$ and let $C \subset X$ be in the class $(1,1,d)$ and correspond to a divisor $D$ in $E\times E$. Suppose $D$ is reducible. Then from lemma \ref{ExE_curve_class_lemma1} we see that $D$ must be the union $\mathsf{p}_1^{-1}(x_1) \cup \mathsf{p}_2^{-1}(x_2)$ where $x_1,x_2\in E$ are generic points. We also have that $D$ is in the class $(1,1,2)$.\\

Suppose $D$ is irreducible. The surfaces $S_{1}$ and $S_{2}$ intersect $D$ exactly once. So the restrictions $\mathsf{p}_i|_D:D \rightarrow E$ are degree 1 and hence isomorphisms. Thus $D$ is the translation of the graph of an automorphism of $E$. \\

All elliptic curves have the automorphisms $x\mapsto \pm x$. Also we have:
\begin{itemize}
\item if $E \cong \C/i$ ($j$-invariant $j(E)=1728$) the $E$ also has the automorphisms $x\mapsto \pm i x$, and
\item if $E = \C/\tau$ with $\tau=\frac{1}{2}(1+i \sqrt{3})$ ($j$-invariant $j(E)=0$) then $E$ also has the automorphisms $x\mapsto \pm \tau x$ and $x\mapsto \pm (\tau-1) x$. 
\end{itemize}

So to complete the proof we have to calculate the intersections $\#(\Gamma_{\xi} \cap S_{\mathsf{op}})$ where $\Gamma_\xi$ is the graph of an automorphism $\xi$. Also, $S_{\mathsf{op}}|_{F_x} \cong \Gamma_{-1}$ hence we calculate $\#(\Gamma_{\xi} \cap \Gamma_{-1})=\#\{(x,\xi(x)) = (x,-x)\}$ in the surface $F_x$. For all the elliptic curves we have:
\begin{itemize}
\item[(a)] $\#(\Gamma_{1} \cap \Gamma_{-1})$ is given by the four 2-torsion points $\{0,\tfrac{1}{2},\tfrac{1}{2}\tau, \tfrac{1}{2}(1+\tau)\}$.
\item[(b)] $\#(\Gamma_{-1} \cap \Gamma_{-1}) = 0$ since one copy can be translated away from the other. 
\end{itemize}
For $E \cong \C/i$ ($j$-invariant $j(E)=1728$) we have:
\begin{itemize}
\item[(d)] $\#(\Gamma_{\pm i} \cap \Gamma_{-1})$ is given by the two points $\{0,\tfrac{1}{2}(1+\tau)\}$.
\end{itemize}
For $E = \C/\tau$ with $\tau=\frac{1}{2}(1+i \sqrt{3})$ ($j$-invariant $j(E)=0$) we have:
\begin{itemize}
\item[(e)] $\#(\Gamma_{\tau} \cap \Gamma_{-1})$ and $\#(\Gamma_{(1-\tau) i} \cap \Gamma_{-1})$ are both determined by the three points $\{0,\tfrac{1}{3}(1+\tau),\tfrac{2}{3}(1+\tau)\}$.
\item[(f)] $\#(\Gamma_{-\tau} \cap \Gamma_{-1})$ and $\#(\Gamma_{(\tau-1) i} \cap \Gamma_{-1})$ are both given by the single point $\{0\}$.
\end{itemize}
\end{proof}

\subsection{Analysis of \texorpdfstring{$1$}{1}-cycles in Singular Fibres of \texorpdfstring{$\mathrm{pr}$}{pr}}

We denote the fibres of the projection $\mathrm{pr}$ by $F_x := \mathrm{pr}^{-1}(x)$. The singular fibres are all isomorphic so we denote a singular fibre by $F_{\mathsf{ban}}$ and its normalisation by $\nu: \widetilde{F}_{\mathsf{ban}}\rightarrow {F}_{\mathsf{ban}}$. From \cite[Prop. 24]{Bryan_Banana} we have that $\widetilde{F}_{\mathsf{ban}} \cong \mathrm{Bl}_{2\,\mathrm{pt}}(\P^1\times \P^1)$ and if we choose the coordinates on each $\P^1$ so that the $0$ and $\infty$ map to a nodal singularity, then the two points blown-up are $z_1= (0,\infty)$ and $z_2= (\infty,0)$.
\begin{align*}
\xymatrix{
\mathrm{Bl}_{z_1,z_2}(\P^1\times \P^1) \ar[d]_{\mathsf{bl}}\ar[r]^{\hspace{2em}\nu} & F_{\mathsf{ban}} \\
\P^1\times \P^1 &
}
\end{align*}
We also let $N:= \pi^{-1}(x)$ be a nodal elliptic fibre in $S$, and denote the natural projections by $\mathsf{q}_i : F_{\mathsf{ban}} \rightarrow N$ (these are the morphisms $\mathrm{pr}_i:X\rightarrow S$ with restricted domain and codomain). 

\begin{figure}
  \centering
      \includegraphics[width=0.7\textwidth]{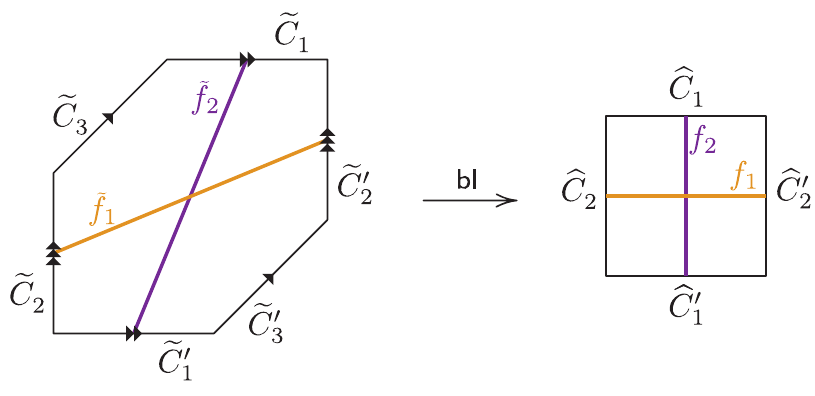}
      \vspace{-1.25em}
  \caption[Normalisation of the singular fibres]{On the left is a depiction of the normalisation $\widetilde{F}_{\mathsf{ban}}$ and on the right is a depiction of $\P^1\times\P^1$. Here $\mathsf{bl}$ is the map blowing up $(0,\infty)$ and $(\infty,0)$. On the right $f_1$ and $f_2$ are generic fibres of the projection maps $\P^1\times \P^1 \rightarrow \P^1$ and on the left $\tilde{f}_1$ and $\tilde{f}_2$ are their proper transforms.}
\end{figure}

\begin{re} Denote the divisors in $\widetilde{F}_{\mathsf{ban}}$ corresponding to the banana curve $C_i$  by $\widetilde{C}_i$ and  $\widetilde{C}'_i$. They are identified in $F_{\mathsf{ban}}$ by
\[
\nu(\widetilde{C}_i) = \nu(\widetilde{C}'_i) = C_i. 
\]
For $i=1,2$ we also denote $\widehat{C}_i = \mathsf{bl}(\widetilde{C}_i)$ and $\widehat{C}'_i = \mathsf{bl}(\widetilde{C}'_i)$  inside $\P^1\times \P^1$. The curve classes in $\widetilde{F}_{\mathsf{ban}}$ are generated by the collection of $\widetilde{C}_i$ and  $\widetilde{C}'_i$'s with the relations:
\[
\widetilde{C}_1 +  \widetilde{C}_3 \sim \widetilde{C}'_1 +  \widetilde{C}'_3 
\hspace{0.5cm}\mbox{and}\hspace{0.5cm}
\widetilde{C}_2 +  \widetilde{C}_3 \sim \widetilde{C}'_2 +  \widetilde{C}'_3.
\]
\end{re}
\begin{re}\label{classes_P1xP1_re}Let $f_1$ and $f_2$ be fibres of the projections $\P^1\times\P^1\rightarrow\P^1$ not equal to any $\widehat{C}_i$ or $\widehat{C}'_i$ and let $\tilde{f}_{1}$ and $\tilde{f}_{2}$ be their proper transforms. Then we also have the relations:
\[
\tilde{f}_1 \sim \widetilde{C}_1 +  \widetilde{C}_3
\hspace{0.5cm}\mbox{and}\hspace{0.5cm}
\tilde{f}_2 \sim \widetilde{C}_2 +  \widetilde{C}_3.
\]
Moreover, if $\widetilde{D}$ is a divisor in $\widetilde{F}_{\mathsf{ban}}$ such that $\nu(\widetilde{D})$ is in the class $(d_1,d_2,d_3)$ then, for $a_i+a'_i = d_i$, we have $D$ is in a class:
\[
a_1\widetilde{C}_1+ a'_1\widetilde{C}'_1 +  a_2\widetilde{C}_2+ a'_2\widetilde{C}'_2+  a_3\widetilde{C}_3+ a'_3\widetilde{C}'_3.
\]
\end{re}

\begin{lemma}\label{FSing_curve_class_lemma1}
Let $C \subset X$ correspond to a divisor $D$ in $F_{\mathsf{ban}}$.
\begin{enumerate}[label={\arabic*)}]
\item $C$ is in the class  $(0,0,d_3)$  if and only if $D$ has $1$-cycle $d_3 C_3$.
\item \label{FSing_curve_class_lemma1_(0,d_2,d_3)} $C$ is in the class  $(0,d_2,d_3)$  if and only if $D$ has $1$-cycle $\widetilde{D} + a_2C_2^{(j)} + a_3 C_3^{(j)}$ where $\tilde{D}$ is the pullback of a degree $a_f$ divisor from the smooth part of $N$ via the projection $\mathsf{q}_i : F_{\mathsf{ban}} \rightarrow N$ such that $a_f+a_2 = D_2$ and $a_f+a_3 = D_3$. Moreover, $\widetilde{D}$ is in the class $(0,a_f,a_f)$. 
\end{enumerate}
\end{lemma}
\begin{proof}
Let $C \subset X$ be a curve in the class $(0,d_2,d_3)$ and correspond to a divisor $D$ in $F_{\mathsf{ban}}$. There exists a divisor $\widetilde{D}$ in $\widetilde{F}_{\mathsf{ban}} \cong \mathrm{Bl}_{z_1,z_2}(\P^1\times \P^1)$ with $\nu(\widetilde{D})=D$. \\

\vspace{-0.1cm}
From the discussion in \ref{classes_P1xP1_re} we have that $\mathsf{bl}(\widetilde{D})$ is in the class of $d_2 f_2$ and is hence in its corresponding linear system. So, $\widetilde{D}$ is the union of the the proper transform of $\mathsf{bl}(\widetilde{D})$ and curves supported at $\widetilde{C}_3$ and $\widetilde{C}'_3$. The result now follows. 
\end{proof}

\begin{lemma}\label{FSing_curve_class_lemma2}
Let $C \subset X$ be an irreducible curve in the class $(1,1,d)$ and correspond to a divisor $D$ in $F_{\mathsf{ban}}$. Then $D$ is the image under $\nu$ of the proper transform under $\mathsf{bl}$ of a smooth divisor in $|f_1 + f_2|$ on $\P^1 \times \P^1$. Moreover, the value of $d$ is determined the intersection of $D$ with points in $\mathsf{P} =\{(0,0),(0,\infty),(\infty,0),(\infty,\infty)\}$. That is, if $D$ intersects
\begin{enumerate}[label={\arabic*)}]
\item $(0,0)$ and $(\infty,\infty)$ only, then $d=2$. 
\item $(0,\infty)$ and $(\infty,0)$ only, then $d=0$. 
\item $(0,0)$ only or $(\infty,\infty)$ only, then $d=2$. 
\item $(0,\infty)$ only or $(\infty,0)$ only, then $d=1$. 
\item no points of $\mathsf{P}$, then $d=2$. 
\end{enumerate}
Moreover, there are no smooth divisors in $|f_1 + f_2|$ on $\P^1 \times \P^1$ that intersect other combinations of these points. 
\end{lemma}
\begin{proof}
Let $C \subset X$ be an irreducible curve in the class $(1,1,d)$ and correspond to a divisor $D$ in $F_{\mathsf{ban}}$. There exists an irreducible divisor $\widetilde{D}$ in $\widetilde{F}_{\mathsf{ban}} \cong \mathrm{Bl}_{z_1,z_2}(\P^1\times \P^1)$ with $\nu(\widetilde{D})=D$. 
$\widetilde{D}$ does not contain either of the exceptional divisor $\widetilde{C}_3$ and $\widetilde{C}'_3$. Hence, it must be the proper transform of a curve in $\P^1 \times \P^1$. \\

From the discussion in \ref{classes_P1xP1_re} we have that $\mathsf{bl}(\widetilde{D})$ is in the class of $f_1+f_2$ and is hence in its corresponding linear system. The only irreducible divisors in $|f_1+f_2|$ are smooth and can only pass through the combinations of points in $\mathsf{P}$ that are given. We refer to the appendix \ref{f1+f2_linear_system_decomp} for the proof of this. The total transform in any divisor in $|f_1+f_2|$ will correspond to a curve in the class $C_1+C_2+2C_2$. Hence the classes of the proper transforms depend on the number of intersections with the set $\{(0,\infty), (\infty,0)\}$. The values are immediately calculated to be those given.
\end{proof}

\subsection{Parametrising \texorpdfstring{$1$}{1}-cycles} For $i\in \{1,2, \mathsf{op}\}$ we use the notation: 
\begin{enumerate}[label={\arabic*)}]
\item  $B_{i} = \{b_{i}^1,\ldots,b_{i}^{12}\}$ is the set of the 12 points in $S_{i}$ that correspond to nodes in the fibres of the projection $\pi_i:=\mathrm{pr}|_{S_i}:S_{i} \rightarrow \P^1$.
\item $ S_{i}^\circ =  S_{i} \setminus  B_{i}$ is the complement of $ B_{i}$ in $S_{i}$
\end{enumerate}

\begin{lemma} \label{(1,0,d1,d2)_Chow_lemma}
In the case $\beta = \sigma + (0,d_2,d_3)\in H_2(X,\mathbb{Z})$ there is the following constructible homeomorphism in $K_0(\mathrm{Var}_{\C})[[Q_2,Q_3]]$:
\[
\Chow^{\sigma + (0,\bullet,\bullet)}(X) \cong Q_\sigma\, \Sym_{Q_2 Q_3}^\bullet (S_{2}^\circ) \times \Sym_{Q_2}^{\bullet}(B_{2})\times \Sym_{Q_3}^{\bullet}(B_{\mathsf{op}}).
\]
Moreover, if the points of $\Chow^{\sigma + (0,\bullet,\bullet)}(X)$ are identified using this constructible homeomorphism, then for $x = (\bm{a}\bm{y}, \bm{m}\bm{b_{2}}, \bm{n}\bm{b_{\mathsf{op}}}) \in \Chow^{\sigma + (0,\bullet,\bullet)}(X)$ the fibre of the cycle map is $\rho_\bullet^{-1}(x) = \HilbCyc^{\bullet}(X,\cycle{q})$ where 
\[
\cycle{q} = \sigma +  \sum_i a_i \mathrm{pr}_2^{-1}(y_i) + \sum_i m_i C_2^{(i)} +\sum_i n_i C_3^{(i)}.
\]
\end{lemma}
\begin{proof}
From lemma \ref{curve_decomposition_lemma} part \ref{curve_decomposition_lemma_sigma} it is enough to consider curves in the class $(0,d_2,d_3)$. Also from \ref{curve_decomposition_lemma} part \ref{curve_decomposition_lemma_no_sigma} we know that the curves are supported on fibres of the map $\mathrm{pr}:X\rightarrow \P^1$. From lemma \ref{ExE_curve_class_lemma1} part \ref{ExE_curve_class_lemma1_(0,d_2,d_3)} we know that the curves supported on smooth fibres of $\mathrm{pr}$ must be thicken fibres of the projection $\mathrm{pr}_2:X \rightarrow S$.  Similarly we know from lemma \ref{FSing_curve_class_lemma1} part \ref{FSing_curve_class_lemma1_(0,d_2,d_3)}  that the curves supported on singular fibres of $\mathrm{pr}$ must be the union of thicken fibres of $\mathrm{pr}_2$ and curves supported on the $C_2$ and $C_3$ banana curves. The result now follows. 
\end{proof}

We also use the notation:
\begin{enumerate}[label={\arabic*)}]
\item $\mathtt{N}_i\subset S_{i}$ are the 12 nodal fibres of $\pi_i:S_{i} \rightarrow \P^1$ with the nodes removed and: 
\begin{itemize}
\item[] $\mathtt{N}_i = \mathtt{N}^\sigma_i ~\amalg~ \mathtt{N}^\emptyset_i$ where   $\mathtt{N}^\sigma_i := \mathtt{N}_i\cap \sigma$ and $\mathtt{N}^\emptyset_i := \mathtt{N}_i\setminus \sigma$.
\end{itemize}
\item $ \mathtt{Sm}_i =  S_{i}^\circ \setminus  \mathtt{N}_i$ is the complement of $\mathtt{N}_i$ in $S_{i}^\circ$ and: 
\begin{itemize}
\item[] $\mathtt{Sm}_i = \mathtt{Sm}^\sigma_i ~\amalg~ \mathtt{Sm}^\emptyset_i$ where $\mathtt{Sm}^\sigma_i := \mathtt{Sm}_i\cap \sigma$ and $\mathtt{Sm}^\emptyset_i := \mathtt{Sm}_i\setminus \sigma$. 
\end{itemize}
\item $\mathtt{J}^{0}$ and $\mathtt{J}^{1728}$ to be the subsets of points $x\in \P^1$ such that $\pi^{-1}(x)$ has $j$-invariant $0$ or $1728$ respectively and $\mathtt{J} = \mathtt{J}^{0}~\amalg ~\mathtt{J}^{1728}$.
\item $\mathtt{L}$ to be the linear system $|f_1+f_2|$ on $\P^1\times \P^1$ with the singular divisors removed where $f_1$ and $f_2$ are fibres of the two projection maps. 
\item $ \widetilde{\mathrm{Aut}}(E):= \mathrm{Aut}(E) \setminus\{\pm1\}$. \\
\end{enumerate}

\begin{remark}
The following lemma should be parsed in the following way. For $i,j\in\{0,1\}$ and $b,d_3\in \mathbb{Z}_{\geq 0}$, a subscheme in the class $\beta = b \sigma + (i,j,d_3)\in H_2(X,\mathbb{Z})$ will have $1$-cycle of the following form:
\[
\cycle{q} = b \sigma + D + \sum_i n_i C_3^{(i)}
\]
where $D$ is reduced and does not contain $\sigma$ or  $C_3^{(i)}$. Then $D$ is in the class $(i,j,n)\in H_2(X,\mathbb{Z})$ for some $n\in \mathbb{Z}_{\geq 0}$.\\

The Chow groups parameterise the different possible $1$-cycles that $D$ can have. Moreover, these possibilities depend on $i$ and $j$:
\begin{itemize}
\item If $i=j=0$ then $D$ is the empty curve. If 
\item If $i=0$ and $j=1$ then $D$ can be either a fibre of the projection $\mathrm{pr}_2$ or  $C_2^{(i)}$. 
\item If $i=j=i$ then $D$ can be reducible or irreducible. If $D$ is reducible then it is some combination of fibres and banana curves. We call the collection where $D$ is irreducible the \textit{diagonals}. \\
\end{itemize}
\end{remark}

\begin{lemma}\label{classes_bs+(i,j,d)_main_chow_main_lemma}
In the cases $\beta = d \sigma + (i,j,d_3)\in H_2(X,\mathbb{Z})$ there is the following constructible homeomorphism in $K_0(\mathrm{Var}_{\C})[[Q_\sigma,Q_3]]$:
\[
\Chow^{\bullet \sigma + (i,j,\bullet)}(X) \cong \sum_{b\in \mathbb{Z}_{\geq 0}} Q_\sigma^b\,\Chow^{ (i,j,\bullet)}(X).
\]
Moreover, using the identification $Q_\sigma^b\,\Chow^{ (i,j,\bullet)}(X)  =  \{b\} \times \Chow^{ (i,j,\bullet)}(X)$ the points of $(b,x)\in \Chow^{\bullet \sigma + (i,j,\bullet)}(X)$ give the fibres $(\eta^{ij}_\bullet)^{-1}(b,x) = \HilbCyc^{\bullet}(X,\cycle{q})$ with $\cycle{q} = b\sigma + \cycle{q}'$ where $\cycle{q}' \in  \Chow^{ (i,j,\bullet)}(X)$. \\

We also have the following decompositions of $\Chow^{ (i,j,\bullet)}(X)$, by constructible homeomorphisms in $K_0(\mathrm{Var}_{\C})[[Q_\sigma,Q_3]]$:
\begin{enumerate}[label={\arabic*)}]
\item\label{classes_bs+(0,0,d)_main_chow_main_lemma} For $i=j=0$ we have the decomposition of $\Chow^{(0,0,\bullet)}(X)$ with parts:
\begin{enumerate}[label={\alph*)}]
\item $Sym_{Q3}^{\bullet}(B_{\mathsf{op}})$.
\end{enumerate}
The corresponding fibres are then $(\eta^{00}_\bullet)^{-1}(x) = \HilbCyc^{\bullet}(X,\cycle{q})$ where:
\begin{enumerate}[label={\alph*)}]
\item If $ x = \bm{n}\bm{b_{\mathsf{op}}}$ then $\cycle{q} = \sum_i n_i C_3^{(i)}$.\\
\end{enumerate}
\item\label{classes_bs+(0,1,d)_main_chow_main_lemma} For $i=0$ and $j=1$ we have a decomposition of $\Chow^{(0,1,\bullet)}(X)$ with parts:
\begin{enumerate}[label={\alph*)}]
\item $Q_2 \,S_{2}^\circ \times \Sym_{Q_3}^{\bullet}(B_{\mathsf{op}})$
\item $Q_2 \,\underset{k=1}{\overset{12}{\amalg}}  \Sym_{Q_3}^\bullet(\{ b_{\mathsf{op}}^{k}\}) \times \Sym_{Q_3}^\bullet(B_{\mathsf{op}}\setminus \{b_{\mathsf{op}}^{k}\})$.
\end{enumerate}
The corresponding fibres are then $(\eta^{01}_\bullet)^{-1}(x) = \HilbCyc^{\bullet}(X,\cycle{q})$ where:
\begin{enumerate}[label={\alph*)}]
\item If $x = (y, \bm{n} \bm{b_{\mathsf{op}}})$ then $\cycle{q} = \mathrm{pr}_2^{-1}(y) + \sum_i n_i C_3^{(i)}$.
\item If $x = (a_k b^k_{\mathsf{op}}, \bm{n} \bm{b_{\mathsf{op}}})$ then $\cycle{q} = C_2^{(k)}+  a_k C_3^{(k)} + \sum_i n_i C_3^{(i)}$.\\
\end{enumerate}

\item\label{classes_bs+(1,1,d)_main_chow_main_lemma} For $i=j=1$ we have a decomposition of $\Chow^{(1,1,\bullet)}(X)$ with parts:
\begin{enumerate}[label={\alph*)}]
\item  $~Q_2Q_3 \,\, S_{1}^\circ \times S_{2}^\circ \times \Sym_{Q3}^{\bullet}(B_{\mathsf{op}}) $
\item  $~Q_2Q_3 \,\underset{k=1}{\overset{12}{\amalg}}  S_{1}^\circ  \times\Sym_{Q3}^\bullet(\{ b_{\mathsf{op}}^{k}\}) \times \Sym_{Q3}^\bullet(B_{\mathsf{op}}\setminus \{b_{\mathsf{op}}^{k}\}) $
\item  $~Q_2Q_3 \,\underset{k=1}{\overset{12}{\amalg}}  S_{2}^\circ  \times\Sym_{Q3}^\bullet(\{ b_{\mathsf{op}}^{k}\}) \times \Sym_{Q3}^\bullet(B_{\mathsf{op}}\setminus \{b_{\mathsf{op}}^{k}\}) $
\item  $Q_2Q_3 \,\hspace{-0.58em}\underset{\mbox{\tiny$\begin{array}{c}k,l=1\\k\neq l\end{array}$}}{\overset{12}{\amalg}} \hspace{-0.58em} \Sym_{Q3}^\bullet(\{ b_{\mathsf{op}}^{k}\})   \times\Sym_{Q3}^\bullet(\{ b_{\mathsf{op}}^{l}\}) \times \Sym_{Q3}^\bullet(B_{\mathsf{op}}\setminus \{b_{\mathsf{op}}^{k},b_{\mathsf{op}}^{l}\})$
\item  $~Q_2Q_3 \,\underset{k=1}{\overset{12}{\amalg}}  \Sym_{Q3}^\bullet(\{ b_{\mathsf{op}}^{k}\}) \times \Sym_{Q3}^\bullet(B_{\mathsf{op}}\setminus \{b_{\mathsf{op}}^{k}\})$
\item  $~ Q_2Q_3 \,\,\mathtt{Diag}^{\bullet}$
\end{enumerate}
where $\mathtt{Diag}^{\bullet}$ will be defined by a further decomposition. The corresponding fibres of a)\,-\,e) are $(\eta^{11}_\bullet)^{-1}(x) = \HilbCyc^{\bullet}(X,\cycle{q})$ where:
\begin{enumerate}[label={\alph*)}]
\item If $x = (y_1,y_2, \bm{n} \bm{b_{\mathsf{op}}})$ then $\cycle{q} = \mathrm{pr}_1^{-1}(y_1)+\mathrm{pr}_2^{-1}(y_2) + \sum_i n_i C_3^{(i)}$.
\item If $x = (y_1, a_k b^k_{\mathsf{op}}, \bm{n} \bm{b_{\mathsf{op}}})$ then $\cycle{q} =  \mathrm{pr}_1^{-1}(y_1) + C_2^{(k)}+ a_k C_3^{(k)} +\sum_i n_i C_3^{(i)}$.
\item If $x = ( y_2, a_k b^k_{\mathsf{op}},\bm{n} \bm{b_{\mathsf{op}}})$ then $\cycle{q} =   \mathrm{pr}_2^{-1}(y_2) + C_1^{(k)}+ a_k C_3^{(k)} +\sum_i n_i C_3^{(i)}$.
\item {\small If $x = ( a_k b^k_{\mathsf{op}}, a_l b^l_{\mathsf{op}},\bm{n} \bm{b_{\mathsf{op}}})$ then  $\cycle{q} =    C_1^{(k)} + C_2^{(l)}+ a_k C_3^{(k)}+a_l C_3^{(l)} +\sum_i n_i C_3^{(i)}$}.
\item If $x = ( a_k b^k_{\mathsf{op}}, \bm{n} \bm{b_{\mathsf{op}}})$ then $\cycle{q} =    C_1^{(k)} + C_2^{(k)}+ a_k C_3^{(k)} +\sum_i n_i C_3^{(i)}$.
\end{enumerate}
For part f), $\mathtt{Diag}^{\bullet}$ is defined by the further decomposition:
\begin{enumerate}
\item[g)] ~$\mathtt{Sm}_{1} \times \Sym_{Q3}^\bullet (B_{\mathsf{op}})$
\item[h)] ~$\mathtt{Sm}_{2} \times \Sym_{Q3}^\bullet (B_{\mathsf{op}}) $
\item[i)] ~$\hspace{-0.3em}\underset{y\in\mathtt{J}}{\amalg} ~E_{\pi(y)} \times \widetilde{\mathrm{Aut}}(E_{\pi(y)}) \times \Sym_{Q3}^\bullet (B_{\mathsf{op}}) $
\item[j)] ~$\hspace{-0.35em}\underset{k=1}{\overset{12}{\amalg}}~ \mathtt{L} \times \Sym_{Q3}^\bullet(\{ b_{\mathsf{op}}^{k}\}) \times \Sym_{Q3}^\bullet(B_{\mathsf{op}}\setminus \{b_{\mathsf{op}}^{k}\})$.
\end{enumerate}
The corresponding fibres of g)\,-\,j) are $(\eta^{11}_\bullet)^{-1}(x) = \Hilb^{n}(X,\cycle{q})$ where:
\begin{enumerate}
\item[g)] If $x = (y ,\bm{n} \bm{b_{\mathsf{op}}})$ then $\cycle{q} =    D_y +\sum_i n_i C_3^{(i)}$ where $D_y$ is the graph of the map $f(z) = z + y|_{E_{\pi(y)}}$ in the fibre $F_{\pi(y)}=E_{\pi(y)}\times E_{\pi(y)}$.
\item[h)] If $x = (y ,\bm{n} \bm{b_{\mathsf{op}}})$ then $\cycle{q} =    D_y +\sum_i n_i C_3^{(i)}$ where $D_y$ is the graph of the map $f(z) = -z + y|_{E_{\pi(y)}}$ in the fibre $F_{\pi(y)}=E_{\pi(y)}\times E_{\pi(y)}$.
\item[i)] If $x = (y ,\bm{n} \bm{b_{\mathsf{op}}})$ then $\cycle{q} =    D_y +\sum_i n_i C_3^{(i)}$ where $D_y$ is the graph of the map $f(z) = A(z) + y|_{E_{\pi(y)}}$ for some $A\in \mathrm{Aut}(E_{\pi(y)})\setminus\{\pm1\}$.
\item[j)] If $x = (z ,a_k b^k_{\mathsf{op}}, \bm{n} \bm{b_{\mathsf{op}}})$ then $\cycle{q} =    \nu(\widetilde{L}_z) + a_k C_3^{(k)}  +\sum_i n_i C_3^{(i)}$ where $\widetilde{L}_z$ is the proper transform of the divisor $L_z$ in $\P^1\times \P^1$ and $\nu$ is the normalisation of the $k$th singular fibre.

\end{enumerate}

\end{enumerate}

\end{lemma}
\begin{proof}
The decomposition $\Chow^{\bullet \sigma + (i,j,\bullet)}(X)  \cong \sum_{b\in \mathbb{Z}_{\geq 0}} Q_\sigma^b\,\Chow^{ (i,j,\bullet)}(X)$ is immediate from lemma \ref{curve_decomposition_lemma} part \ref{curve_decomposition_lemma_bsigma}. Hence it is enough to parametrise the curves in the class $\beta = (i,j,\bullet)$. Also from \ref{curve_decomposition_lemma} part \ref{curve_decomposition_lemma_no_sigma} we know that the curves are supported on fibres of the map $\mathrm{pr}:X\rightarrow \P^1$. We must have that
\begin{align*}
\Cyc(C) =~& a \sigma + D +\sum_{i=1}^{12} m_i C_3^{(i)}. 
\end{align*}
for some minimal reduces curve $D$ in the class $(1,1,n)$ for $n\geq0$ minimal. The possible $D$ curves are described in lemmas \ref{ExE_curve_class_lemma1}, \ref{ExE_curve_class_lemma2}, \ref{FSing_curve_class_lemma1} and  \ref{FSing_curve_class_lemma2}. The result now follows. 
\end{proof}

\begin{remark}
Using lemma \ref{classes_bs+(i,j,d)_main_chow_main_lemma} and the identification  $Q_\sigma^b\,\Chow^{ (i,j,\bullet)}(X)  =  \{b\} \times \Chow^{ (i,j,\bullet)}(X)$ we make the following identification for notational convenience in discussing the points in section \ref{main_computation_section_A}:
\[
\Chow^{\bullet \sigma + (i,j,\bullet)}(X) \cong \mathbb{Z}_{\geq 0}\times \Chow^{ (i,j,\bullet)}(X).
\]
\end{remark}

\section{Techniques for Calculating Euler Characteristic}\label{theory_computation_section}

\subsection{Quot Schemes and their Decomposition}
This section is a summary of required results from \cite{BK}. First we consider the following subset of the Hilbert scheme. 
\begin{definition}\label{Hilbert_subscheme_with_CM_curve_definition}
Let $C\subset X$ be a Cohen-Macaulay subscheme of dimension 1. Consider the Hilbert scheme parameterising one dimensional subschemes $Z\subset X$ with class $[Z]=[C]\in H_2(X,\mathbb{Z})$ and $\chi(\O_Z) = \chi (\O_C) + n$ for some $n\in \mathbb{Z}_{\geq 0}$. This contains the following closed subset:
\[
\Hilb^n(X,C) := \big\{  Z \subset X\mbox{ such that $C\subset Z$ and $I_C/I_Z$ has finite length $n$}   \big\}.
\]
\end{definition}

It is convenient to replace the Hilbert scheme here with a Quot scheme. Recall the Quot scheme $\Quot_X^n(\mathcal{F})$ parametrising quotients $\mathcal{F}\twoheadrightarrow \mathcal{G}$ on $X$, where $\mathcal{G}$ is zero-dimensional of length $n$. It is related to the above Hilbert scheme in the following way.
\begin{lemma}\label{hilb_to_quot_lemma}\cite[Lemma 5]{BK},\cite[Lemma 5.1]{Ricolfi_Local}. 
The following equality holds in $K_0(\mathrm{Var}_{\C})(\!(p)\!)$:
\[
\Hilb^\bullet (X,C):=\sum_{n\in\mathbb{Z}_{\geq 0}}\Hilb^n (X,C) =  \Quot^\bullet_X(I_C) := \sum_{n\in\mathbb{Z}_{\geq 0}} \Quot^n_X(I_C).
\]
\end{lemma}
We also consider the following subscheme of these Quot schemes.
\begin{definition}\cite[Def. 12]{BK}
Let  $\mathcal{F}$ be a coherent sheaf on $X$, and $S\subset X$ a locally closed subset.  We define the locally closed subset of $\Quot_X^n(\mathcal{F})$
\[
\Quot_X^n(\mathcal{F},S) := \big\{ [\mathcal{F}\twoheadrightarrow \mathcal{G}] \in \Quot_X^n(\mathcal{F})~\big|~ \Supp_{\mathsf{red}}(\mathcal{G}) \subset S  \big\}.
\]
\end{definition}
This allows us to decompose the Quot schemes in the following way. 
\begin{lemma}\label{quot_decomposition_lemma}\cite[Prop. 13]{BK}
Let  $\mathcal{F}$ be a coherent sheaf on $X$, $S\subset X$ a locally closed subset and $Z\subset X$ a closed subset. Then if $Z\subset S$ and and $n\in\mathbb{Z}_{\geq 0}$ there is a geometrically bijective constructible map:
\[
\Quot_X^{n}(\mathcal{F},S) \longrightarrow \coprod_{n_1+n_2 = n} \Quot_X^{n_1}(\mathcal{F},S\setminus Z) \times \Quot_X^{n_2}(\mathcal{F},Z).
\]
\end{lemma}

\subsection{An Action on the Formal Neighbourhoods}\label{action_subsection}

Let $C\subset X$ be a one dimensional subscheme in the class $\beta\in H_2(X)$ with $1$-cycle $\cycle{q}=\Cyc(C)$. We recall the notation defining $\HilbCyc^{n}(X,\cycle{q})\subset \Hilb^{\beta,n}(X)$ to be the following reduced subscheme
\[
\HilbCyc^{n}(X,\cycle{q}) := \Big\{~[Z]\in \Hilb^{\beta,n}(X) ~|~  \Cyc(Z) = \cycle{q} ~\Big\} .
\]
Furthermore, we define 
\[
\Hilb^{n}_{\mathsf{CM}}(X,\cycle{q}) \subset \HilbCyc^{n}(X,\cycle{q})
\] 
to be the open subset containing Cohen-Macaulay subschemes of $Z$. 
\begin{lemma}\label{action_on_CM_lemma}
Suppose $Z\subset X$ is a one dimensional Cohen-Macaulay subscheme such that:
\begin{enumerate}[label={\arabic*)}]
\item $Z$ has the decomposition $Z = C \cup \bigcup_i Z_i$ where $C$ is reduced, $Z_i \cap Z_j =\emptyset $ for $i\neq j$ and for each $i$ we have $C\cap Z_i$ is finite. 
\item There are formal neighbourhoods $V_i$ of $Z_i$ in $X$ such that $(\C^*)^2$ acts on each and fixes $Z_i^{\mathsf{red}}$. 
\item $C \cap (\bigcup_i V_i)$ is invariant under the $(\C^*)^2$-action on $V_i$. 
\end{enumerate}
Then there is a $(\C^*)^2$-action on $\Hilb^{n}_{\mathsf{CM}}(X,\Cyc(Z))$ such that if $\alpha \in (\C^*)^2$ and $Y\in \Hilb^{n}_{\mathsf{CM}}(X,\Cyc(Z))$ then:
\[
\alpha\cdot Y ~= ~\widetilde{C}  ~\cup ~ \alpha \cdot (Y|_{\bigcup_i V_i}).
\]
\end{lemma}

\begin{proof}
Let $\beta=[Z]\in H_2(X,\mathbb{Z})$ and use the simplifying notation $\mathcal{H}= \Hilb^{\beta,n}(X)$ and $H:=\Hilb^{n}_{\mathsf{CM}}(X,\Cyc(Z))$. The composition $H\hookrightarrow \Hilb^{n}_{\Cyc}(X,\Cyc(Z)) \hookrightarrow \mathcal{H}$ defines an immersion $H \hookrightarrow \mathcal{H}$ expressing $H$ as a locally closed (reduced) subscheme of $\mathcal{H}$. Moreover, the immersion also defines the following flat family over $H$:
\[
\begin{array}{c}\xymatrix@R=1em{
\mathcal{Z} \ar@{^(->}[rr]_{\iota}\ar[rd]_{h}&& X \times H \ar[ld]^{p_2}\\
& H &
}\end{array}
\]

We consider $\mathcal{Z} \cap (\bigcup_i Z_i^{\mathsf{red}}\times H)$ and define $\mathcal{C} := \overline{\mathcal{Z} \setminus (\mathcal{Z} \cap (\bigcup_i Z_i^{\mathsf{red}}\times H))}$, the scheme-theoretic closure of the scheme-theoretic complement. Also, we denote by $\mathcal{Z}^\dagger := \widehat{\mathcal{Z}}_{\overline{\mathcal{Z}\setminus \mathcal{C}}}$ the formal completion of $\overline{\mathcal{Z}\setminus \mathcal{C}}$ in $\mathcal{Z}$. This gives a decomposition of $\mathcal{Z}$ by
\[
\mathcal{Z} = \mathcal{C} \cup \mathcal{Z}^{\dagger}.
\]

For all closed points $x\in H$, the fibres of the composition $\mathcal{C} \hookrightarrow \mathcal{Z} \rightarrow H$ have property $(\mathcal{C})_x = C$ as subschemes of $X$. Hence, $\mathcal{C} \cap( C\times H)$ contains all of the closed points of $C\times H$. Thus, $C\times H = \mathcal{C} \cap( C\times H)$ and we have the following immersion over $H$
\[
\begin{array}{c}\xymatrix@R=1em{
C \times H \ar@{^(->}[rr]_{\alpha}\ar[rd]_{f_1}&&\mathcal{C}  \ar[ld]^{f_2}\\
& H &
}\end{array}
\]
where both of $f_1$ and $f_2$ are proper with $f_1$ being flat. Also, since $H$ is Noetherian, \cite[Tag 01TX, Tag 05XD]{stacks_project} shows that there is an open set $U\subset H$ containing all the closed points such that $\alpha_U$ is an isomorphism. Hence, $C \times H = \mathcal{C}$ as subschemes of $X\times H$. \\

Now, using a similar argument to the previous paragraph, we have that the underlying reduced schemes of $\mathcal{Z}^\dagger$ and $\bigcup_i Z_i^{\mathsf{red}}\times H$ are equal as subschemes in $X\times H$. This means that they both have the same formal completion in $X\times H$.\\

The formal completion of $\bigcup_i Z_i^{\mathsf{red}}\times H$ in $X\times H$ is given by $\bigcup_i V_i\times H$, so we have  inclusion $\mathcal{Z}^\dagger \hookrightarrow \bigcup_i V_i\times H$. Furthermore, we have an inclusion 
\[
\mathcal{Z} \hookrightarrow \mathcal{C} \cup \bigcup_i (V_i\times H) = (C \cup \bigcup_i V_i) \times H.
\]

Now, by letting $G:= (\mathbb{C}^*)^2$ and $W:=C \cup \bigcup_i V_i$  we have the following diagram
\[
\xymatrix@C=3em{
G\times \mathcal{Z}\ar@{^(->}[r]^{\mathrm{id}\oplus \iota\hspace{0.5cm}}\ar[rd]_{\mathrm{id}\oplus h}& G\times ( W \times H )\ar[r]^{\beta}\ar[d]^{\mathrm{id}\oplus  q_2} & G\times (W \times H ) \ar@{^(->}[r]^{\hspace{0.75cm}{}^{\mathrm{id}\oplus (j\oplus \mathrm{id})}\hspace{0.6cm}}\ar[d]^{\mathrm{id}\oplus q_2}& G\times (X \times H) \ar[d]^{\mathrm{id}\oplus p_2}\\
& G\times H \ar@{=}[r]& G\times H\ar@{=}[r]& G\times H
}
\]
where $\beta$ is defined by $(g,(x,y))\mapsto (g,(g\cdot x,y))$, $j: W \hookrightarrow X$ is the natural inclusion and $q_2: W\times H \rightarrow H$ is the projection onto the second factor. Taking the composition of the top row defines the following flat family in $\mathcal{H}$ over $G\times H$:
\[
\zeta:= \left(\begin{array}{c}\xymatrix@R=1em{
G\times \mathcal{Z} \ar@{^(->}[rr]_{}\ar[rd]_{}&& X \times (G\times H) \ar[ld]^{}\\
& G \times H &
}\end{array}\right)
\]
The flat family $\zeta$ defines a morphism $\Phi:G\times H \rightarrow \mathcal{H}$.\\

We have that $G\times H$ is reduced, so the scheme theoretic image $\mathrm{Sch.Im.}(\Phi)$ is also reduced. Moreover, every closed point of the $\mathrm{Sch.Im.}(\Phi)$ is contained in $H$. Hence, we have $\mathrm{Sch.Im.}(\Phi) \subseteq H$. Thus, $\Phi$ defines a morphism $G\times H \rightarrow H$. It is now straightforward to show that this morphism satisfies the identity and compatibility axioms of a group action. 
\end{proof}

\begin{remark}
In the case where $Z$ is smooth, an analysis similar to that in the proof of lemma \ref{action_on_CM_lemma}  was carried out in \cite{Ricolfi_DT/PT}. However, the analysis there is scheme-theoretic. Moreover the equality $\kappa^{-1}(z) = \Quot^{\bullet}_X(I_Z)$, which will appear in the proof of the next lemma (lemma \ref{kappa_constructible_invariant_lemma}), was proven scheme theoretically in the case of $Z$ begin smooth.\\
\end{remark}

Define $\Hilb^{\blacklozenge}_{\mathsf{CM}}(X, \cycle{q}):=\coprod_{m\in\mathbb{Z}}\Hilb^{m}_{\mathsf{CM}}(X, \cycle{q})$ and consider the constructible map
\[
\kappa: \HilbCyc^{\bullet}(X,\cycle{q}) \longrightarrow \Hilb^{\blacklozenge}_{\mathsf{CM}}(X,\cycle{q})
\]
where $Z\subset X$ is mapped to the maximal Cohen-Macaulay subscheme $Z_{\mathsf{CM}}\subset Z$ (also forgeting the indexing variable $p$). Then for $z\in \Hilb^{\blacklozenge}_{\mathsf{CM}}(X,\cycle{q})$ corresponding to $Z\subset X$ we have
\[
\kappa^{-1}(x) 
= \sum_{m\in\mathbb{Z}} p^m\, \Hilb^{m-\chi(\O_Z)}(X,Z)
= p^{\chi(\O_Z)}  \Hilb^{\bullet}(X,Z).
\]
Moreover, we have
\begin{align*}
e\Big(\HilbCyc^{\bullet}(X,\cycle{q})\Big) 
&= e\Big(\Hilb^{\blacklozenge}_{\mathsf{CM}}(X,\cycle{q}), \kappa_*1\Big)\\
&= e\Big(\Hilb^{\blacklozenge}_{\mathsf{CM}}(X,\cycle{q})^{(\C^*)^2}, \kappa_*1\Big)  \num \label{Invariant_Euler_char_kappa*}
\end{align*}
where $(\kappa_* 1)(z) := e(\kappa^{-1}(z))$ and the last line comes from the following lemma.

\begin{lemma}\label{kappa_constructible_invariant_lemma}
The constructible function $\kappa_* 1$ is invariant under the $(\C^*)^2$-action. That is if $\alpha \in (\C^*)^2$ and $z \in \Hilb^{n}_{\mathsf{CM}}(X,\cycle{q})$ then $(\kappa_* 1)(z) = (\kappa_* 1)(\alpha\cdot z)$. 
\end{lemma}
\begin{proof}
Let $\alpha \in (\C^*)^2$ and $z \in \Hilb^{n}_{\mathsf{CM}}(X,\cycle{q})$ correspond to $Z\subset X$. Also let $Z_i$ and $V_i$ be as in lemma \ref{action_on_CM_lemma} with $\widetilde{Z}: = \bigcup_i Z_i$ and $\widetilde{V}:=\bigcup_iV_i$. Then the fibre $\kappa^{-1}(z)$ is 
\[
\kappa^{-1}(z) = p^{\chi(\O_Z)} \Hilb^{\bullet}(X,Z) =p^{\chi(\O_Z)} \Quot^{\bullet}_X(I_Z)
\]
where the last equality is in $K_0(\mathrm{Var}_{\C})(\!(p)\!)$ from lemma \ref{hilb_to_quot_lemma}. Also from lemma \ref{quot_decomposition_lemma} we have a geometrically bijective constructible map:
\[
\Quot_X^{n}(I_Z) \longrightarrow \coprod_{n_1+n_2 = n} \Quot_X^{n_1}(I_Z,X\setminus \widetilde{V}) \times \Quot_X^{n_2}(I_Z,\widetilde{V}).
\]
We have $I_{\alpha \cdot Z}|_{X\setminus\widetilde{V}}  = I_{Z}|_{X\setminus \widetilde{V}}$ so $\Quot_X^{n_1}(I_Z,X\setminus\widetilde{V}) \cong \Quot_X^{n_1}(I_{\alpha\cdot Z},X\setminus\widetilde{V})$. Moreover, we have isomorphisms
\[
\Quot_X^{n_2}(I_Z,\widetilde{V}) \cong \Quot_{\widetilde{V}}^{n_2}(I_Z|_{\widetilde{V}}) 
\]
and $Z_{\widetilde{V}} \cong \alpha\cdot Z_{\widetilde{V}}$ so we have an isomorphism 
\[
\Quot_X^{n_2}(I_Z,\widetilde{V}) \cong \Quot_X^{n_2}(I_{\alpha\cdot Z},\widetilde{V})
\]
Taking Euler characteristic now shows that $e\big(\kappa^{-1}(z)\big) = e\big(\kappa^{-1}(\alpha\cdot z)\big)$.
\end{proof}

\begin{re}\label{decomposition_method}
We will now consider a useful tool in calculating Euler characteristics of the form given in (\ref{Invariant_Euler_char_kappa*}). First let  $z \in \Hilb^{n}_{\mathsf{CM}}(X,\cycle{q})$ correspond to $Z\subset X$ such that $Z$ is locally monomial. In other words, for every geometric point $z\in Z$ the restriction of $Z$ to the formal neighbourhood of $z$ in $X$ is of the form $\Spec\, \C[[x,y,z]]/I_z$ where $I_z$ is an ideal generated by monomials in $\C[[x,y,z]]$. Then the fibre $\kappa^{-1}(x)$ is
\[
\kappa^{-1}(x) =p^{\chi(\O_Z)}  \Hilb^{\bullet}(X,Z) =p^{\chi(\O_Z)}  \Quot^{\bullet}_X(I_Z)
\]
where the last equality is in $K_0(\mathrm{Var}_{\C})(\!(p)\!)$ from lemma \ref{hilb_to_quot_lemma}. To compute this fibre we employ the following method:
\begin{enumerate}[label={\arabic*)}]
\item Decompose $X$ by $X  =Z~\amalg~ W$ where $W:= X\setminus Z$ 
\item Let $Z^\diamond$ be set of singularities of $Z^{\mathsf{red}}$. 
\item Let $ \coprod_i Z_i = Z \setminus Z^\diamond$  be a decomposition into irreducible components. 
\end{enumerate}
Then applying Euler characteristic to lemma \ref{quot_decomposition_lemma} we have:
\[
e\big(\Quot_X^{\bullet}(I_Z) \big) = e\big(\Quot_X^{\bullet}(I_Z,W) \big)  \prod_{z\in Z^\diamond} e\big(\Quot_X^{\bullet}(I_Z,\{z\})\big) \prod_i e\big(\Quot_X^{\bullet}(I_Z,Z_i)\big).
\]
\end{re}

\vspace{0.2cm}

\subsection{Partitions and the topological vertex}

We recall the terminology of 2D partitions, 3D partitions and the topological vertex from \cite{ORV, BCY}. A \textit{2D partition} $\lambda$ is an infinite sequence of weakly decreasing integers that are zero except for a finite number of terms. The size of a 2D partition $|\lambda|$ is the sum of the elements in the sequence and the length $l(\lambda)$ is the number of non-zero elements. We will also think of a 2D partition as a subset of $(\mathbb{Z}_{\geq 0})^2$ in the following way:
\[
\lambda ~ \leftrightsquigarrow
~ \{ (i,j) \in (\mathbb{Z}_{\geq 0})^2 ~|~ \lambda_i\geq j  \geq 0 \mbox{ or } i=0\}
\]

A \textit{3D partition} is a subset $\bm{\eta} \subset (\mathbb{Z}_{\geq 0})^3$ satisfying the following condition:
\begin{enumerate}[label={\arabic*)}]
\item $(i,j,k) \in \bm{\eta}$ if and only if one of $i$, $j$ or $k$ is zero or one of $(i-1,j,k)$, $(i,j-1,k)$ or $(i,j,k-1)$  is also in $\bm{\eta}$.
\end{enumerate}

Given a triple of 2D partitions $(\lambda,\mu,\nu)$ we also define a \textit{3D partition asymptotic to $(\lambda,\mu,\nu)$} is a 3D partition $\bm{\eta}$ that also satisfies the conditions:
\begin{enumerate}[label={\arabic*)}]
\item $(j,k) \in \lambda$ if and only if $(i,j,k) \in \bm{\eta}$ for all $i\gg 0$.
\item $(k,i) \in \mu$ if and only if $(i,j,k) \in \bm{\eta}$ for all $j\gg 0$.
\item $(i,j) \in \nu$ if and only if $(i,j,k) \in \bm{\eta}$ for all $k\gg 0$.
\end{enumerate}

\begin{figure}
  \centering
      \includegraphics[width=0.5\textwidth]{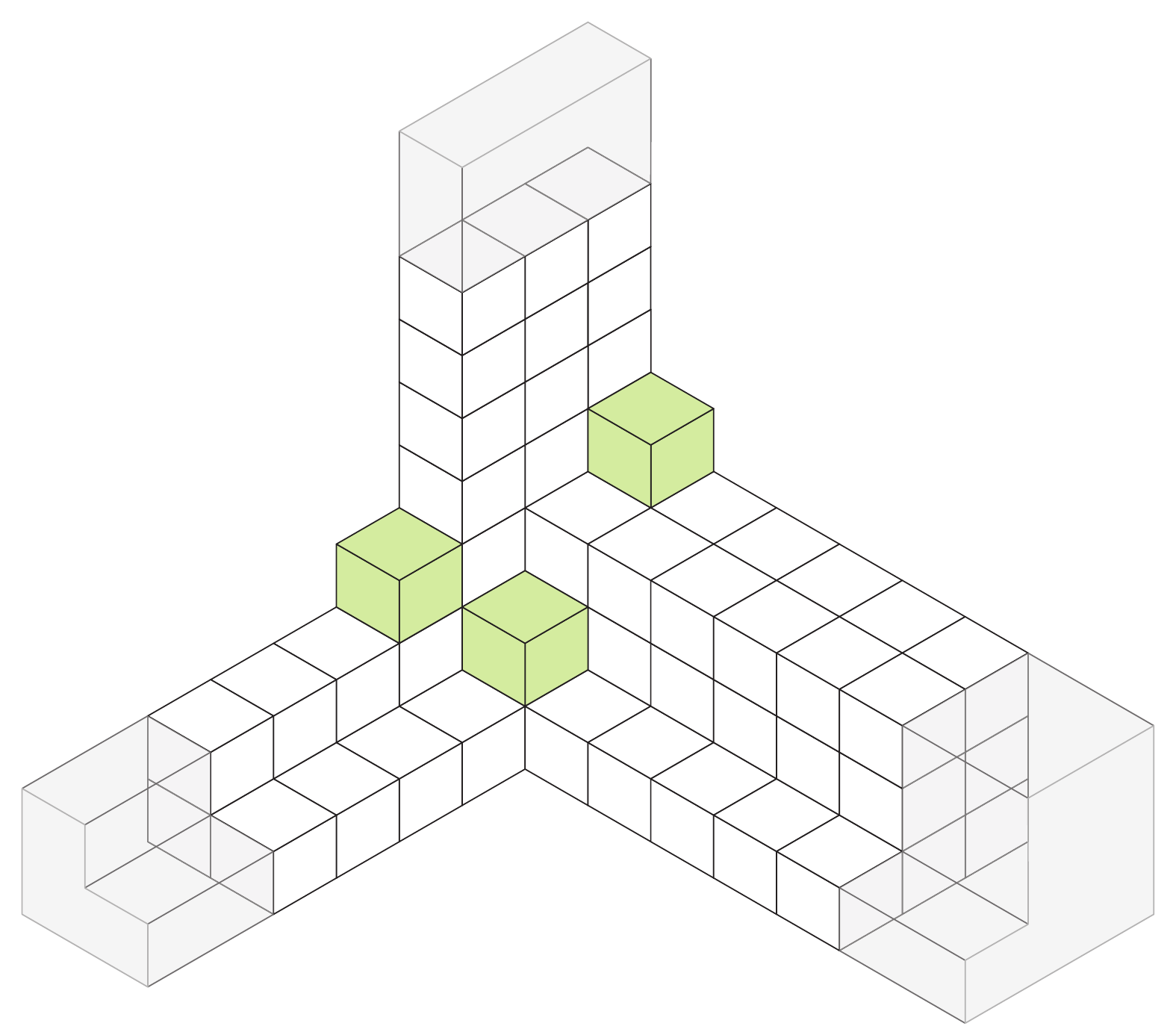}
      \vspace{-1.25em}
  \caption[A 3D partition asymptotic to $\big((2,1), (3,2,2), (1,1,1)\big)$]{A 3D partition asymptotic to $\big((2,1), (3,2,2), (1,1,1)\big)$. The partition containing only the white boxes has renormalised volume $-16$. The partition including the green boxes has renormalised volume $-13$.}
\end{figure}

The \textit{leg of $\bm{\eta}$ in the $i$th direction} is the subset $\{ (i,j,k) \in \bm{\eta} ~|~ (j,k) \in \lambda  \}$. We analogously define the legs of $\bm{\eta}$ in the $j$ and $k$ directions. The \textit{weight} of a point in $\bm{\eta}$ is defined to be
\[
\xi_{\bm{\eta}}(i,j,k) := 1- \#\,\{\mbox{legs of $\bm{\eta}$ containing $(i,j,k)$}\}.
\]
Using this we define the \textit{renormalised volume} of $\bm{\eta}$ by:
\begin{align}
|\bm{\eta}| := \sum_{(i,j,k)\in\bm{\eta}} \xi_{\bm{\eta}}(i,j,k). \label{renormalised_volume_def}
\end{align}

The \textit{topological vertex} is the formal Laurent series:
\[
\mathsf{V}_{\lambda \mu \nu} := \sum_{\bm{\eta}} p^{|\bm{\eta}|} 
\]
where the sum is over all 3D partitions asymptotic to $(\lambda,\mu,\nu)$. An explicit formula for $\mathsf{V}_{\lambda \mu \nu}$ is derived in  \cite[Eq. 3.18]{ORV} to be:
\[
\mathsf{V}_{\lambda\mu\nu} = M(p) p^{-\frac{1}{2}(\|\lambda\|^2+\|\mu^t\|^2+\|\nu\|^2)} S_{\nu^t}(p^{-\rho}) \sum_{\eta} S_{\lambda^t/\eta}(p^{-\nu-\rho})S_{\mu/\eta}(p^{-\nu^t-\rho})
\]

\subsection{Partition Thickened Section, Fibre and Banana Curves} \label{partition_thickened_section}
In this subsection we consider the non-reduced structure of curves in our desired classes. The partition thickened structure will be the fixed points of a $(\C^*)^2$-action. 
\begin{re}\label{canonical_section_coords}
Recall that the section $\zeta \in S$ is the blow-up of a point in $z\in \P^2$. Choose once and for all a formal neighbourhood $\Spec \,\C[\![ s,t]\!]$ of $z\in \P^2$. The blow-up gives the formal neighbourhood of $\zeta\in S$ with 2 coordinate charts:
\[
\C[\![s,t]\!][u]/(t-s u) \cong \C[\![s]\!][u]
\hspace{0.5cm}\mbox{and}\hspace{0.5cm}
\C[\![s,t]\!][v]/(s-t v) \cong \C[\![t]\!][v]
\]
with change of coordinates $s\mapsto tv $ and $u\mapsto v^{-1}$. This gives the formal neighbourhood of $\sigma\in X$ with 2 coordinate charts:
\[
\C[\![s_1,s_2]\!][u]
\hspace{0.5cm}\mbox{and}\hspace{0.5cm}
\C[\![t_1,t_2]\!][v]
\]
with change of coordinates $s_i\mapsto t_i v $ and $u\mapsto v^{-1}$. We call these coordinates the \textbf{canonical formal coordinates} around  $\sigma \in X$. 
\end{re}
\begin{re}\label{canonical_section_coords_relative}
Now consider a reduced curve $D$ in $X$ that intersects $\sigma$ transversely with length 1. When $D$ is restricted to the formal neighbourhood of $\sigma$, it is given by
\[
\C[\![s_1,s_2]\!][u]/(a_0 u-a_1, b_0s_1 -b_1s_2)
\hspace{0.5cm}\mbox{and}\hspace{0.5cm}
\C[\![t_1,t_2]\!][v]/(a_0 -a_1v, b_0t_1 -b_1t_2)
\]
for some $[a_0:a_1],[b_0:b_1] \in \P^1$. We use this to define the change of coordinates:
\[
\tilde{s}_1 \mapsto b_0s_1 -b_1s_2 
\hspace{0.5cm}\mbox{and}\hspace{0.5cm}
\tilde{s}_2 \mapsto b_1s_1 + b_0 s_2 
\]
\[
\tilde{t}_1 \mapsto b_0t_1 -b_1t_2 
\hspace{0.5cm}\mbox{and}\hspace{0.5cm}
\tilde{t}_2 \mapsto b_1t_1 + b_0 t_2 
\]
We call these coordinates the \textbf{canonical formal coordinates relative to $D$}. 
\end{re}

\begin{definition}\label{lambda_thickened_section_defn}
Let $\C[\![s_1,s_2]\!][u]$ and $\C[\![t_1,t_2]\!][v]$ be either the formal canonical coordinates of \ref{canonical_section_coords} or those of \ref{canonical_section_coords_relative}. Then we define
\begin{enumerate}[label={\arabic*)}]
\item The \textbf{canonical $(\C^*)^2$-action} on these coordinates by $(s_1,s_2)\mapsto (g_1s_1, g_2 s_2)$ and $(t_1,t_2)\mapsto (g_1t_1, g_2 t_2)$. 
\item Let $\lambda = (\lambda_1,\ldots,\lambda_l,0,\ldots)$ be a 2D partition. The \textbf{$\lambda$-thickened section} denoted by $\lambda\sigma$ is the subscheme of $X$ defined by the ideal given in the coordinates by
\[
(s_2^{\lambda_1},\ldots, s_1^{l-1} s_2^{\lambda_l},s^{l}) \hspace{0.5cm}\mbox{and}\hspace{0.5cm} (t_2^{\lambda_1},\ldots, t_1^{l-1} t_2^{\lambda_l},t^{l} ). 
\]
\end{enumerate}
\end{definition}

\begin{figure}
  \centering
      \includegraphics[width=0.35\textwidth]{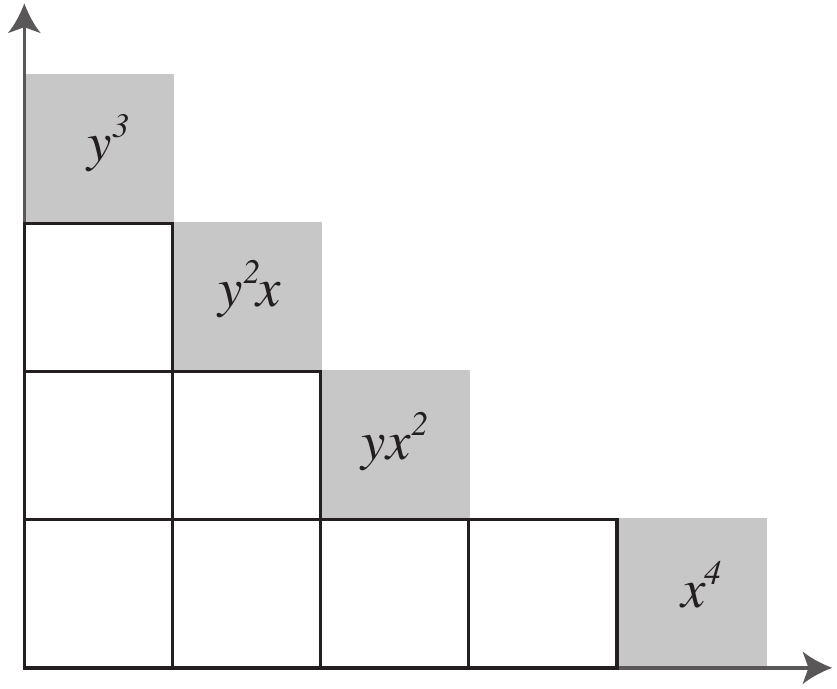}
      \vspace{-1.25em}
  \caption[Subschemes in $\C^2$ and monomial ideals]{Depiction of the subscheme in $\C^2$ given by the monomial ideal $(y^3, y^2 x,y^1x^2,y^1x^3,x^4)$   associated to the partition $(3,2,1,1,0\ldots)$.}
\end{figure}

\begin{re}\label{canonical_C3_coords}
We now consider a canonical formal neighbourhood of the banana curve $C_3$. We follow much of the reasoning from \cite[\S 5.2]{Bryan_Banana}. Let $x\in S$ correspond to a point where $\pi:S\rightarrow \P^1$ is singular. Let formal neighbourhoods in the two isomorphic copies of $S$ be given by 
\[
\Spec \,\C[\![ s_1,t_1]\!]
\hspace{0.5cm}\mbox{and}\hspace{0.5cm}
\Spec \,\C[\![ s_2,t_2]\!]
\]
and the map $S\rightarrow \P^1$ be given by $r\mapsto s_i t_i$. Then the formal neighbourhood of a conifold singularity  in $X$ is given by
\[
\Spec \,\C[\![ s_1,t_1,s_2,t_2]\!]/( s_1t_1 - s_2t_2), 
\]
and the restriction to a fibre of the projection $S\times_{\P^1}S \rightarrow \P^1$ is
\[
\Spec \,\C[\![ s_1,t_1,s_2,t_2]\!]/( s_1t_1, s_2t_2). 
\]
Now, blowing up along $\{s_1 = t_2 = 0\}$ (which is canonically equivalent to blowing up along $\{s_1 -t_1 = s_2-t_2 = 0\}$), we have the two coordinate charts:
\begin{align*}
&\C[\![s_1,t_2, s_2,t_2]\!][u]/(s_1- u t_2,s_2- ut_1) \cong \C[\![t_1,t_2]\!][u],\hspace{0.5cm}\mbox{and}\hspace{0.5cm}\\
&\C[\![s_1,t_2, s_2,t_2]\!][v]/(t_1- v s_2,t_2- vs_1) \cong \C[\![s_1,s_2]\!][v],
\end{align*}
where the change of coordinates is given by $t_1\mapsto v s_2$ , $t_2\mapsto v s_1$  and $u \mapsto v^{-1}$. We call these coordinates the \textbf{canonical formal coordinates} around the banana curve $C_3$. 
\end{re}
\begin{re}
With these coordinates we have:
\begin{enumerate}[label={\arabic*)}]
\item 
The restriction to the fibre of $\mathrm{pr}:X\rightarrow \P^1$ is
\[
\C[\![t_1,t_2]\!][u]/(t_1t_2u) 
\hspace{0.5cm}\mbox{and}\hspace{0.5cm}
\C[\![s_1,s_2]\!][v]/(s_1s_2v).
\]
\item The banana curve $C_3$ is given by
\[
\C[\![t_1,t_2]\!][u]/(t_1,t_2) 
\hspace{0.5cm}\mbox{and}\hspace{0.5cm}
\C[\![s_1,s_2]\!][v]/(s_1,s_2).
\] 
\end{enumerate}

\begin{re}\label{canonical_C3_coords_relative}
Similar to \ref{canonical_section_coords_relative} we also consider canonical relative coordinates for a $C_3$ banana curve. Recall \ref{FSing_curve_class_lemma2} and let $D$ be the image under $\nu:\widetilde{F}_{\mathsf{ban}}\rightarrow F_{\mathsf{ban}}$ of the proper transform under $\mathsf{bl}:\mathrm{Bl}_{(0,\infty),(\infty,0)}(\P^1\times\P^1)\rightarrow \P^1\times \P^1$ of a smooth divisor in $|f_1 + f_2|$ on $\P^1 \times \P^1$. \\

If $D$ intersects $(0,0)$ then the restriction of $D$ to the formal neighbourhood of $C_3$ is given by:
\[
\C[\![s_1,s_2]\!][v] / (s_1 - a s_2, v)
\]
for some $a\in \C^*$. In this case we define \textbf{canonical formal coordinates relative to $D$} around a $C_3$ banana by the following change of coordinates. 
\[
\tilde{s}_1 \mapsto s_1 -as_2 
\hspace{0.5cm}\mbox{and}\hspace{0.5cm}
\tilde{s}_2 \mapsto s_1 + a s_2 
\]
\[
\tilde{t}_1 \mapsto a t_1 +  t_2 
\hspace{0.5cm}\mbox{and}\hspace{0.5cm}
\tilde{t}_2 \mapsto -a t_1 + t_2 
\]
We similarly define the same relative coordinates if for $D$ intersects $(\infty,\infty)$ in the ideal $(-at_1 +  t_2,u)$. Note that these coordinates are compatible if $D$ intersects both $(0,0)$ and $(\infty,\infty)$.  
\end{re}

\begin{definition}\label{lambda_thickened_C3_defn}
Let $\C[\![s_1,s_2]\!][u]$ and $\C[\![t_1,t_2]\!][v]$ be either the canonical coordinates or relative coordinates. 
\begin{enumerate}[label={\arabic*)}]
\item The \textbf{canonical $(\C^*)^2$-action} on these coordinates is defined by 
\[
(s_1,s_2,v)\mapsto (g_1s_1, g_2 s_2,v)
\hspace{0.5cm}\mbox{and}\hspace{0.5cm}
(t_1,t_2, u)\mapsto (g_2 t_1, g_1 t_2,u).
\]
\item Let $\lambda = (\lambda_1,\ldots,\lambda_l,0,\ldots)$ be a 2D partition. The \textbf{$\lambda$-thickened banana curve $C_3$} denoted by $\lambda C_3$ is the subscheme of $X$ defined by the ideal given in the coordinates by
\[
(s_2^{\lambda_1},\ldots, s_1^{l-1} s_2^{\lambda_l},s_1^{l} )
\hspace{0.5cm}\mbox{and}\hspace{0.5cm}
(t_1^{\lambda_1},\ldots, t_2^{l-1} t_1^{\lambda_l},t_2^{l} ). 
\]
(Note the change in coordinates compared to definition \ref{lambda_thickened_section_defn}.)
\end{enumerate}
\end{definition}
\end{re}

\begin{remark}\label{relative_C3_Vertex_remark}
If $D$ intersects both $(0,0)$ and $(\infty,\infty)$  and $\lambda C_3$ is partition thickened in the coordinates relative to $D$. Then ideals for $D\cup \lambda C_3$ at the points $(0,0)$ and $(\infty,\infty)$ are
\[
(s_2^{\lambda_1},\ldots, s_1^{l-1} s_2^{\lambda_l},s_1^{l} ) \cap (s_1,v) 
\hspace{0.35cm}\mbox{and}\hspace{0.35cm}
(t_1^{\lambda_1},\ldots, t_2^{l-1} t_1^{\lambda_l},t_2^{l} )\cap (t_2,u) 
\]
respectively. These both give 3D partitions asymptotic to $(\lambda,\emptyset,\square)$. 
\end{remark}

\begin{lemma} \label{linear_system_divisor_C*_invariant}
Let $D$ be as described in the first paragraph of \ref{canonical_C3_coords_relative}. If let $V$ be the formal neighbourhood of $C_3$ in $X$. If $D$ intersects $(0,0)$ and/or $(\infty,\infty)$ then use the relative coordinates of \ref{canonical_C3_coords_relative}, otherwise use the canonical coordinates of \ref{canonical_C3_coords}. Then $D\cap V$ is invariant under the $(\C^*)^2$-action. 
\end{lemma}
\begin{proof}
We have $D\cap V \neq 0$ if and only if it intersects at least one of $(0,0)$, $(0,\infty)$, $(\infty,0)$, $(\infty,\infty)$. The possible combinations are:
\begin{enumerate}[label={\arabic*)}]
\item \textit{$(0,0)$ and/or $(\infty,\infty)$}: This is by construction of the relative coordinates. 
\item \textit{Exactly one of $(0,\infty)$ or $(\infty,0)$}: Then $D$ is given by the ideal $(v-a, s_1)$ or $(v-a,s_2)$ for some $a\in \C^*$, which are $(\C^*)^2$-invariant. 
\item \textit{ $(0,\infty)$ and $(\infty,0)$}: Then $D$ is given by the ideal $(v-a, s_1s_2)$ for some $a\in \C^*$ which is $(\C^*)^2$-invariant. 
\end{enumerate}
\end{proof}

\begin{re}
It is also shown in \cite[\S 5.2]{Bryan_Banana} that there are the following formal coordinates on $C_2$ compatible with the canonical formal coordinates around $C_3$: 
\[
\C[\![s_1,v]\!] [s_2] 
\hspace{0.5cm}\mbox{and}\hspace{0.5cm}
\C[\![t_1,u]\!] [t_2] 
\]
where the change on coordinates is given by $s_2\mapsto t_2$, $s_1\mapsto t_1t_2$ and $v\mapsto t_2 u$. We can define partition thickenings and a compatible $(\C^*)^2$-action in these coordinates. 
\end{re}

\begin{definition}\label{lambda_thickened_C2_defn}
Let $\C[\![s_1,v]\!] [s_2] $ and $\C[\![t_1,u]\!] [t_2] $ be the above canonical coordinates. 
\begin{enumerate}[label={\arabic*)}]
\item The \textbf{canonical $(\C^*)^2$-action} on these coordinates is defined by: 
\[
(s_1,v,s_2)\mapsto (g_1s_1,v,g_2 s_2)
\hspace{0.5cm}\mbox{and}\hspace{0.5cm}
(t_1,u,t_2)\mapsto (g_2 t_1, u, g_1 t_2).
\] 
\item Let $\mu = (\mu_1,\ldots,\mu_k,0,\ldots)$ be a 2D partition. The \textbf{$\mu$-thickened banana curve $C_2$} denoted by $\mu C_2$ is the subscheme of $X$ defined by the ideal given in the coordinates by
\[
(s_1^{\mu_1},\ldots, v^{k-1} s_1^{\mu_k},v^{k} )
\hspace{0.5cm}\mbox{and}\hspace{0.5cm}
(t_1^{\mu_1},\ldots, u^{k-1} t_1^{\mu_k},u^{k} ). 
\]
(Note the change in coordinates compared to definition \ref{lambda_thickened_C3_defn}.)
\item Let $\lambda = (\lambda_1,\ldots,\lambda_l,0,\ldots)$ be another 2D partition. The \textbf{$(\mu,\lambda)$-thickened banana curve} denoted is the union $\mu C_2+ \lambda C_3$. 
\end{enumerate}
\end{definition}

\begin{remark}
The $C_2$ and $C_3$ banana curves meet in exactly 2 points. At these two points a $(\mu,\lambda)$-thickened banana curve will define define two 3D partitions. One will be asymptotic to $(\mu,\lambda,\emptyset)$ the other will be asymptotic to $(\mu^t,\lambda^t,\emptyset)$ (or equivalently $(\lambda,\mu,\emptyset)$). 
\end{remark}

We will now consider fibres of the projection map $\mathrm{pr}_2:X\rightarrow S$. 

\begin{definition}\label{lambda_thickened_smooth_fibre_defn}
Recall the definition of $\mathrm{pr}_2: X \rightarrow S$ from section \ref{banana_definition_section}. Let $x\in S$ be such that $f_x :=\mathrm{pr}_2^{-1}(x)$ is smooth. Then we define
\begin{enumerate}[label={\arabic*)}]
\item \textbf{Canonical coordinates} on a formal neighbourhood $V_x$ of $f_x$ are formal coordinates $\C[\![s,t]\!]$ of $S$ at $x$ such that  $V_x := f_x \times \Spec\, \C[\![s,t]\!]$ and for $\{x\}\cap \sigma \neq \emptyset$ we have that $\sigma$ restricted to $\C[\![s,t]\!]$ is given by the ideal $(s)$.
\item The \textbf{canonical $(\C^*)^2$-action} on these coordinates by $(s,t)\mapsto (g_1s, g_2 t)$.
\item Let $\lambda = (\lambda_1,\ldots,\lambda_l,0,\ldots)$ be a 2D partition. The \textbf{$\lambda$-thickened smooth fibre} at $x$ denoted by $\lambda f_x$ is the subscheme of $X$ given by the ideal:
\[
(t^{\lambda_1},\ldots, s^{l-1} t^{\lambda_l},s^{l} )
\]
\end{enumerate}
\end{definition}

\begin{re}\label{nodal_fibre_coords_re}
Let $N$ be a nodal fibre of $\pi:S\rightarrow \P^1$, $x\in N\setminus \{\mathrm{node}\}$ and $C=\mathrm{pr}_2^{-1}(x)$. Let $F_{\mathsf{ban}} =\mathrm{pr}_2^{-1}(N)$ and $\widehat{F}_{\mathsf{ban}}$ be its formal completion in $X$. The formal completion of $C$ in $X$ is the same as the formal completion of $C$ in $\widehat{F}_{\mathsf{ban}}$. Let $\widehat{\P^1\times \P^1}$ be the formal neighbourhood of $\P^1\times \P^1$ in the total space of its canonical bundle. It is shown in \cite[Prop. 25]{Bryan_Banana} that there is a natural \'{e}tale map $\widehat{\tau}:\mathrm{Bl}(\widehat{\P^1\times \P^1}) \rightarrow \widehat{F}_{\mathsf{ban}}$ whose restriction to the underlying reduced subschemes is the normalisation map for $F_{\mathsf{ban}}$. \\

Using this and the results of \cite[\S 5.2]{Bryan_Banana} we have charts for the formal completion of the normalisation of $C$ in $\mathrm{Bl}(\widehat{\P^1\times \P^1})$ given by
\[
\C[\![y_0, z_0]\!] [x_0] 
\hspace{0.5cm}\mbox{and}\hspace{0.5cm}
\C[\![y_1, z_1]\!] [x_1] 
\]
with an isomorphism on the complements of $(x_0)$ and $(x_1)$ given by
\[
x_0 \mapsto \frac{1}{x_1},
\hspace{0.5cm}\mbox{and}\hspace{0.5cm}
y_0 \mapsto y_1 x_2^2
\hspace{0.5cm}\mbox{and}\hspace{0.5cm}
z_0 \mapsto z_1
\]
The identification at the node of $C$ is given by the restriction of the morphism $\widehat{\tau}:\mathrm{Bl}(\widehat{\P^1\times \P^1}) \rightarrow \widehat{F}_{\mathsf{ban}}$ is
\[
z_0 = z_1,
\hspace{0.5cm}\mbox{and}\hspace{0.5cm}
x_0 = y_1
\hspace{0.5cm}\mbox{and}\hspace{0.5cm}
y_0 = x_1.
\]
Moreover, these coordinates can be chosen such that $\sigma$ restricted to $\C[\![y_0, z_0]\!] [x_0]$  is given by the ideal $(x_0-1, z_0)$.
\end{re}

\begin{definition}\label{lambda_thickened_nodal_fibre_defn}
Using \ref{nodal_fibre_coords_re} we define
\begin{enumerate}[label={\arabic*)}]
\item \textbf{Canonical coordinates} on a formal neighbourhood $V_C$ of $C$ are formal coordinates given in \ref{nodal_fibre_coords_re}.
\item The \textbf{canonical $(\C^*)^2$-action} on given respectively on $\C[\![y_0, z_0]\!] [x_0] $ and $\C[\![y_1, z_1]\!] [x_1] $ by 
\[
(x_0,y_0,z_0) \mapsto (g_1 x_0,  \tfrac{1}{g_1} y_0, g_2 z_0,)
\hspace{0.5cm}\mbox{and}\hspace{0.5cm}
(x_1,y_1,z_1) \mapsto (\tfrac{1}{g_1} x_1, g_1 y_1, g_2 z_1).
\]

\item Let $\lambda = (\lambda_1,\ldots,\lambda_l,0,\ldots)$ be a 2D partition. The \textbf{$\lambda$-thickened fibre with one node} at $x$ denoted by $\lambda C$ is the subscheme of $X$ given by the ideal which restricted to $\C[\![y_0, z_0]\!] [x_0] $ is
\[
(z_0^{\lambda_1},\ldots, x_0^{l-1} z_0^{\lambda_l},x_0^{l} ) \cap (z_0^{\lambda_1},\ldots, y_0^{l-1} z_0^{\lambda_l},y_0^{l} )
\]
and when restricted to $\C[\![y_1, z_1]\!] [x_1] $ is
\[
 (z_1^{\lambda_1},\ldots, y_1^{l-1} z_1^{\lambda_l},y_1^{l} )  \cap (z_1^{\lambda_1},\ldots, x_1^{l-1} z_1^{\lambda_l},x_1^{l} ).
\]

\end{enumerate}
\end{definition}

\begin{remark}
The partition thickened curves described in this section are easily shown to be the only Cohen-Macaulay subschemes supported in these neighbourhoods that are invariant under the $(\C^*)^2$-action. This is because the invariant Cohen-Macaulay subschemes must be generated by monomial ideals. 
\end{remark}

\begin{lemma}\label{holomorphic_euler_char_of_PTC}
Let $\lambda = (\lambda_1,\ldots,\lambda_l,0,\ldots)$ and $\mu = (\mu_1,\ldots,\mu_k,0,\ldots)$ be 2D partitions, and let $x\in S$ such that $f_x:= \mathrm{pr_2}^{-1}(x)$ contains no banana curves. Then we have the holomorphic Euler characteristics :
\begin{enumerate}[label={\arabic*)}]
\item $\chi(\O_{\lambda \sigma}) = \frac{1}{2} \big( \| \lambda\|^2 + \| \lambda^t\|^2\big)$,
\item $\chi(\O_{\lambda f_x}) = 0$,
\item $\chi(\O_{\sigma \cap \lambda f_x}) = \lambda_1$,
\item $\chi(\O_{\mu C_2~\cup~ \lambda C_3}) = |\bm{\eta}_1| + |\bm{\eta}_2|+  \frac{1}{2} \big( \| \mu \|^2 + \| \mu^t\|^2 + \| \lambda\|^2 + \| \lambda^t\|^2\big)$ where $|\bm{\eta}_i|$ are the renormalised volumes of the minimal 3D partitions associated to $(\mu,\lambda,\emptyset)$ and $(\mu^t,\lambda^t,\emptyset)$. 
\end{enumerate}
\end{lemma}
\begin{proof}
2) and 3) are proved in \cite[Lemma 11]{BK}, the rest are from \cite[Prop. 23]{Bryan_Banana}.
\end{proof}

\subsection[Quot Schemes on \texorpdfstring{$\C^3$}{C3} and the Topological Vertex]{Relation between Quot Schemes on \texorpdfstring{$\C^3$}{C3} and the Topological Vertex}\label{Topological_vertex_to_quot_scheme_section}
This section is predominately a summary of required results from \cite{BK}.  For 2D partitions $\lambda$, $\mu$ and $\nu$ we define the following subscheme of $\C^3$: 
\[
\mathsf{C}_{\lambda,\mu,\nu} = \mathsf{C}_{\lambda,\emptyset,\emptyset} \cup \mathsf{C}_{\emptyset,\mu,\emptyset} \cup \mathsf{C}_{\emptyset,\emptyset,\nu } \subset \Spec \C[r,s,t]
\]
where $\mathsf{C}_{\lambda,\emptyset,\emptyset}$ is defined by the ideal $I_{\lambda,\emptyset,\emptyset} := (t^{\lambda_1},\ldots, t^{l-1} s^{\lambda_l},s^{l} )$, with $\mathsf{C}_{\emptyset,\mu,\emptyset}$ and $\mathsf{C}_{\emptyset,\emptyset, \nu}$ being cyclic permutations of this. In general, we define the ideal $I_{\lambda\mu\nu} = I_{\lambda\emptyset\emptyset}\cap I_{\emptyset\mu\emptyset}\cap I_{\emptyset\emptyset\nu}$.  \\

Now we consider the Quot scheme of length $n$ quotients of $I_{\lambda\mu\nu}$ that are set-theoretically supported at the origin and we employ the following simplifying notation:
\[
\Quot^n(\lambda, \mu, \nu) := \Quot^{n}_{\C^3}(I_{\lambda\mu\nu},\{0\})
\]
A quotient parametrised here will have kernel that is the ideal sheaf of a one-dimensional scheme Z with underlying Cohen-Macaulay (formal) curve $\mathsf{C}_{\lambda,\mu,\nu}$. The embedded points of this scheme are all supported at the origin, but $Z$ doesn't have to be locally monomial. We use the following variation of the notation for the topological vertex:
\[
\widetilde{\mathsf{V}}_{\lambda\mu\nu} := e\big( \Quot^\bullet(\lambda, \mu, \nu)\big) \in \mathbb{Z}[\![p]\!]. 
\]

\begin{lemma} \label{Quot_to_V_lemma}
Let $C$ be a partition thickened section, fibre or $C_3$-banana curve thickened by $\lambda$. Then
\begin{enumerate}[label={\arabic*)}]
\item If $x\in C$ is a smooth point then $e\big(\Quot^{n}_X(I_C, \{x\})\big) = \widetilde{\mathsf{V}}_{\lambda\emptyset\emptyset}$.
\item If $C$ is a thickened nodal fibre then $e\big(\Quot^{n}_X(I_C, \{x\})\big) = \widetilde{\mathsf{V}}_{\lambda \lambda^t\emptyset}$. 
\end{enumerate}
Let $C'$ be a reduced curve intersecting $C$ at $y\in C$ such that $I_{C'} \cap I_{C}$ is locally monomial and there are formal local coordinates $\C[\![r,s,t]\!]$ at $y$ such that:
\begin{enumerate}[label={\arabic*)}]
\item $I_{C'} \cap I_{C} = (t^{\lambda_1},\ldots s^{l-1}t^{\lambda_l},s^l)\cap (r,s)$  then $e\big(\Quot^{n}_X(I_C, \{x\})\big) = \widetilde{\mathsf{V}}_{\lambda \emptyset \square}$.
\item $I_{C'} \cap I_{C} = (t^{\lambda_1},\ldots s^{l-1}t^{\lambda_l},s^l)\cap (r,s) \cap (r,t)$ then $e\big(\Quot^{n}_X(I_C, \{x\})\big) = \widetilde{\mathsf{V}}_{\lambda \square\square}$.
\end{enumerate} 
\end{lemma}
\begin{proof}
The proof is the same as \cite{BK} Lemma 15. 
\end{proof}

\begin{lemma} \label{Quot_to_V_Sym_lemma}
Let $D$ be a one dimensional Cohen-Macaulay subscheme of $X$. 
\begin{enumerate}[label={\arabic*)}]
\item We have:

\vspace{-0.4cm}
\[
e\big(\Quot^{n}_X(I_D, X\setminus D)\big) = \Big(\widetilde{\mathsf{V}}_{\emptyset\emptyset\emptyset}\Big)^{e(X)-e(C)}. 
\]
\item 
Let $\lambda$ be a 2D partition and  $\lambda C \subset D$ be either a partition thickened section, fibre or $C_3$ banana and let $T$ be finite set of points on $C$ such that $C\setminus T$ is smooth. Then
\[
e\big(\Quot^{n}_X(I_D, C\setminus T)\big) = \Big(\widetilde{\mathsf{V}}_{\lambda\emptyset\emptyset}\Big)^{e(C)-e(T)}. 
\]
\end{enumerate}
\end{lemma}
\begin{proof}
The argument is the same as that given for equation (9) in \cite{BK}. 
\end{proof}

The standard $(\C^*)^3$-action on $\C^3$ induces an action on the Quot schemes. The invariant ideals $I \subset \C[r,s,t]$ are precisely those generated by monomials. Also, since there is a bijection between locally monomial ideals and 3D partitions we see that 
\begin{align*}
\widetilde{\mathsf{V}}_{\lambda\mu\nu} 
&= e\big( \Quot^\bullet(\lambda, \mu, \nu)^{(\C^*)^3}\big)\\
&= \sum_{\bm{\eta}} p^{n(\bm{\eta})}
\end{align*}
where we are summing over 3D partitions asymptotic to $(\lambda,\mu,\nu)$ and $n(\bm{\eta})$ is the number of boxes not contained in any legs. Note that that the lowest order term in $\widetilde{\mathsf{V}}_{\lambda\mu\nu}$ is one, which is not true about $\mathsf{V}_{\lambda\mu\nu} $ in general. In fact we have the relationship:
\[
\mathsf{V}_{\lambda\mu\nu}  = p^{|\bm{\eta}_{min}|} \widetilde{\mathsf{V}}_{\lambda\mu\nu} 
\]
where $\bm{\eta}_{min}$ is the 3D partition associated to $\mathsf{C}_{\lambda\mu\nu}$, and $|\cdot|$ is the renormalised volume defined in eqn (\ref{renormalised_volume_def}). 

\begin{lemma} \label{Vtilde_to_V_lemma}
If $\lambda$ is a 2D partition then we have the following equalities:
\begin{enumerate}[label={\arabic*)}]
\item $\mathsf{V}_{\lambda\emptyset\emptyset}  = \widetilde{\mathsf{V}}_{\lambda\emptyset\emptyset} $
\item $\mathsf{V}_{\lambda\square\emptyset}  = p^{-\lambda_1} \widetilde{\mathsf{V}}_{\lambda\square\emptyset}  $
\item $\mathsf{V}_{\lambda\square\square}  = p^{-\lambda_1 -\lambda^t_1} \widetilde{\mathsf{V}}_{\lambda\square\square}  $
\item $\mathsf{V}_{\lambda\lambda^t\emptyset}  = p^{-\|\lambda\|^2} \widetilde{\mathsf{V}}_{\lambda\lambda^t\emptyset}  $
\end{enumerate}
\end{lemma}
\begin{proof}
Parts 1), 2) and 4) are directly from \cite{BK} lemma 17. For part 3), there are $\lambda_1$ boxes that are in the $\lambda$-leg and one of the $\square$-legs. There are $\lambda^t_1$ boxes that are in the $\lambda$-leg and the other $\square$-leg. There is one box that is contained in all three so the renormalised volume is calculated to be
\[
(\lambda_i -1)(1-2) + (\lambda_i^t -1)(1-2) + (1)(1-3) = - \lambda_i -  \lambda^t_i
\]
\end{proof}

\section[Euler Characteristic of the Fibres of the Chow Map]{Calculating the Euler Characteristic from the Fibres of the Chow Map}\label{main_computation_section}
\subsection{Calculation for the class \texorpdfstring{$\sigma + (0,\bullet,\bullet)$}{s + (0,*,*)}}  \label{main_computation_section_A}

We now recall some previously introduced notation: 
\begin{enumerate}[label={\arabic*)}]
\item  $B_{i} = \{b_{i}^1,\ldots,b_{i}^{12}\}$ is the set of the 12 points in $S_{i}$ that correspond to nodes in the fibres of the projection $\pi_i:=\mathrm{pr}|_{S_i}:S_{i} \rightarrow \P^1$.
\item $ S_{i}^\circ =  S_{i} \setminus  B_{i}$ is the complement of $ B_{i}$ in $S_{i}$
\item $\mathtt{N}_i\subset S_{i}$ are the 12 nodal fibres of $\pi_i:S_{i} \rightarrow \P^1$ with the nodes removed and: 
\begin{itemize}
\item[] $\mathtt{N}_i = \mathtt{N}^\sigma_i ~\amalg~ \mathtt{N}^\emptyset_i$ where   $\mathtt{N}^\sigma_i := \mathtt{N}_i\cap \sigma$ and $\mathtt{N}^\emptyset_i := \mathtt{N}_i\setminus \sigma$.
\end{itemize}
\item $ \mathtt{Sm}_i =  S_{i}^\circ \setminus  \mathtt{N}_i$ is the complement of $\mathtt{N}_i$ in $S_{i}^\circ$ and: 
\begin{itemize}
\item[] $\mathtt{Sm}_i = \mathtt{Sm}^\sigma_i ~\amalg~ \mathtt{Sm}^\emptyset_i$ where $\mathtt{Sm}^\sigma_i := \mathtt{Sm}_i\cap \sigma$ and $\mathtt{Sm}^\emptyset_i := \mathtt{Sm}_i\setminus \sigma$. 
\end{itemize}
\end{enumerate}

\begin{figure}
  \centering
      \includegraphics[width=1\textwidth]{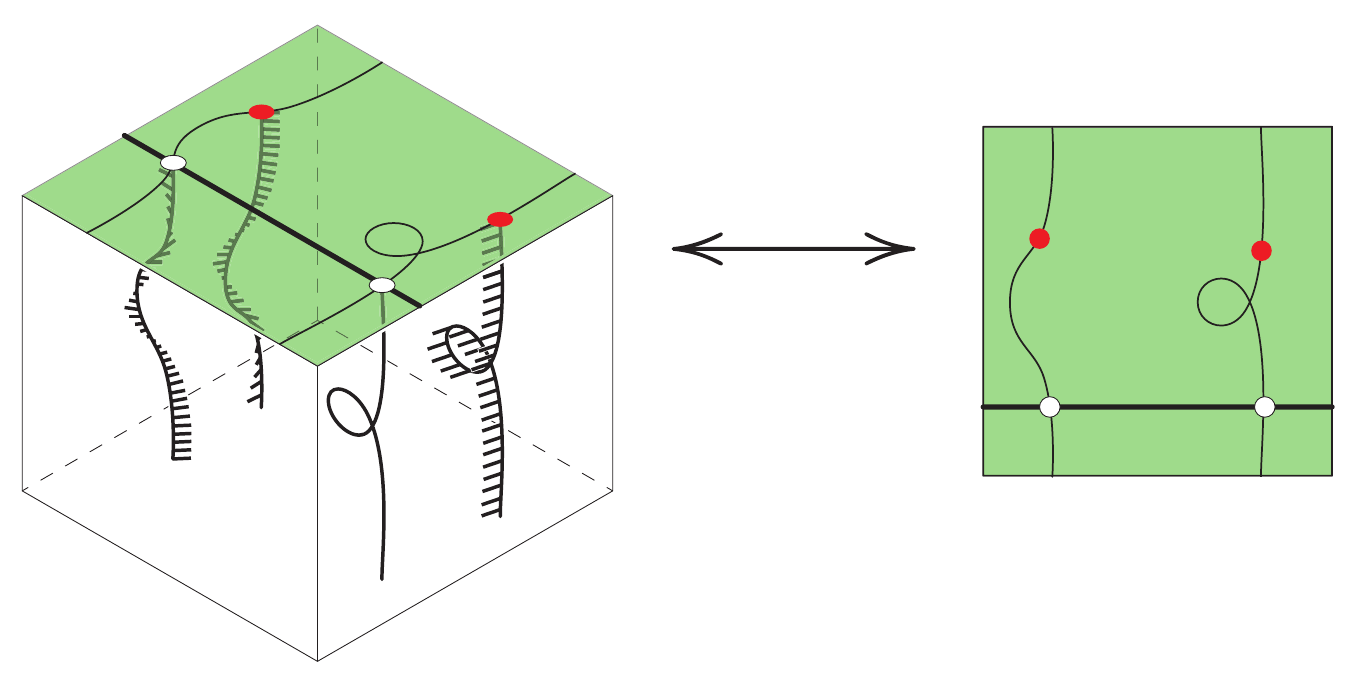}
      \vspace{-1.25em}
  \caption[Decomposition of the Chow sub-scheme for vertical fibres]{Depiction of the decomposition of the Chow sub-scheme that  parametrises the vertical fibres of $\mathrm{pr}_2$. The red dots indicate when the fibres don't intersect the section i.e. $\mathtt{Sm}_2^\emptyset$ and $\mathtt{N}_2^\emptyset$.   The white dots indicate  when the fibres do intersect the section  i.e. $\mathtt{Sm}_2^\sigma$ and $\mathtt{N}_2^\sigma$.}
\end{figure}

Now from lemma \ref{(1,0,d1,d2)_Chow_lemma} we can further decompose $\Chow^{\sigma + (0,\bullet,\bullet)}(X)$ as: 
\begin{align*}
& \Sym_{Q_2Q_3}^\bullet (\mathtt{Sm}^\sigma_2 )\times  \Sym_{Q_2Q_3}^\bullet (\mathtt{N}^\sigma_2 ) \times \Sym_{Q_2Q_3}^\bullet (\mathtt{Sm}^{\emptyset}_2 )  \times  \Sym_{Q_2Q_3}^\bullet (\mathtt{N}^{\emptyset}_2 ) \\
& \times \Sym_{Q_2}^{\bullet}(B_{2})\times \Sym_{Q_3}^{\bullet}(B_{\mathsf{op}}).
\end{align*}

Moreover, if $\cycle{q} = (\bm{a}\bm{x},\bm{c}\bm{y},\bm{d}\bm{z},\bm{l}\bm{w}, \bm{m}\bm{b_{2}}, \bm{n}\bm{b_{\mathsf{op}}}) \in \Chow^{\sigma + (0,\bullet,\bullet)}(X)$ then the fibre is  given by $\rho_\bullet^{-1}(\cycle{q}) \cong \HilbCyc^{\bullet}(X,\cycle{q})$ where 
\begin{align*}
\cycle{q}=~& \sigma +  \sum_i a_i \mathrm{pr}_2^{-1}(x_i)  + \sum_i c_i \mathrm{pr}_2^{-1}(y_i) \\
&+  \sum_i d_i \mathrm{pr}_2^{-1}(z_i) +  \sum_i l_i \mathrm{pr}_2^{-1}(w_i) + \sum_i m_i C_2^{(i)} +\sum_i n_i C_3^{(i)}.
\end{align*}

\begin{re} 
Suppose $C$ is Cohen-Macaulay with the cycle given above. Note that $C$ can be decomposed into a part supported on $C_2 \cup C_3$ and a part supported away from the banana configuration.  This gives the following formal neighbourhoods and  $(\C^*)^2$-actions:
\begin{enumerate}[label={\arabic*)}]
\item Let $U_i$ be the formal neighbourhood of $C_2^{(i)} \cup C_3^{(i)}$ in $X$. These have a canonical $(\C^*)^2$-action described in \ref{lambda_thickened_C3_defn} and \ref{lambda_thickened_C2_defn}. 
\item Let $V_i$ be the formal neighbourhood of $ \mathrm{pr}_2^{-1}(y_i)$ in $X$. These have a canonical $(\C^*)^2$-action described in definition \ref{lambda_thickened_smooth_fibre_defn} and $\sigma \cap V_i$ is either empty of invariant under this action. 
\end{enumerate}
\begin{figure}
  \centering
      \includegraphics[width=0.4\textwidth]{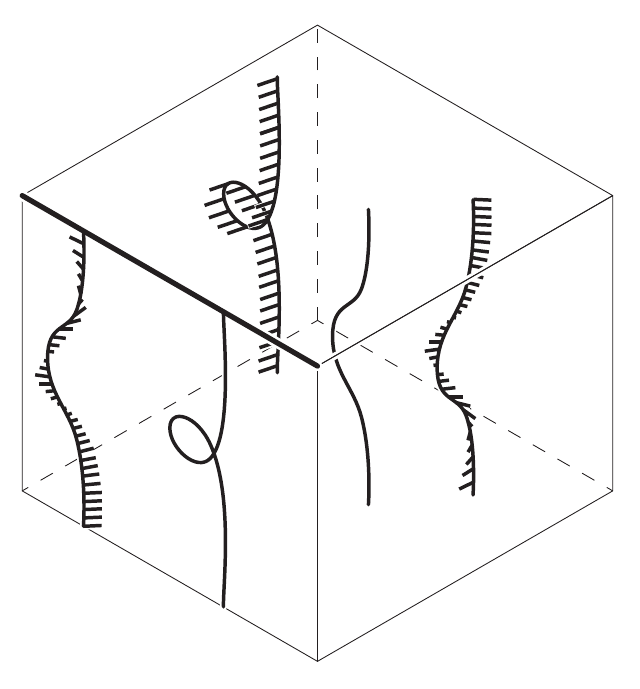}\hspace{1cm}
      \includegraphics[width=0.4\textwidth]{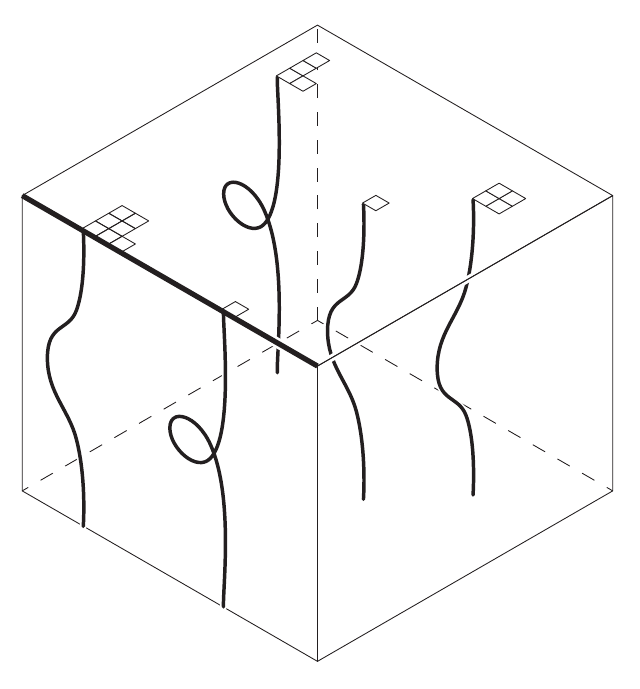}
      \vspace{-1.25em}
  \caption[Cohen-Macaulay subschemes in the fibres of $\rho_\bullet$]{Visual depictions of Cohen-Macaulay subschemes in the fibres of $\rho_\bullet$ that are supported away from the banana configurations. On the left is a general subscheme and on the right is a partition thickened curve which is a fixed point of the $(\C^*)^2$-action.}
\end{figure}
Hence the conditions of lemma \ref{action_on_CM_lemma} are satisfied and there is a $(\C^*)^2$-action defined on $\Hilb^{n}_{\mathsf{CM}}(X,\cycle{q})$. Using  the partition thickened notation introduced in section \ref{partition_thickened_section} we introduce the subschemes:
\[
C_{\bm{\alpha},\bm{\gamma},\bm{\delta},\bm{\lambda},\bm{\mu},\bm{\nu}}:= 
\begin{array}{l}
\sigma 
\cup \bigcup_i \big(\alpha^{(i)}f_{x_i} \big)
\cup \bigcup_i \big(\gamma^{(i)}f_{y_i} \big)
\cup \bigcup_i \big(\delta^{(i)}f_{z_i} \big)\vspace{0.1cm}\\
\cup \bigcup_i \big(\lambda^{(i)}f_{w_i} \big) 
\cup \bigcup_i \big(\mu^{(i)}C_2^{(i)} \big) 
\cup \bigcup_i \big(\nu^{(i)}C_3^{(i)} \big)
\end{array}
\]
and their ideals $I_{\bm{\alpha},\bm{\gamma},\bm{\delta},\bm{\lambda},\bm{\mu},\bm{\nu}}$ in $X$ where $\bm{\alpha}$, $\bm{\gamma}$, $\bm{\delta}$, $\bm{\lambda}$, $\bm{\mu}$ and  $\bm{\nu}$ are tuples of partitions of $\bm{a}$, $\bm{c}$, $\bm{d}$, $\bm{l}$, $\bm{m}$ and  $\bm{n}$ respectively. Then using this notation we  can identify the fixed points of the action as the following formal sum of discrete sets:
\begin{align*}
\Hilb^{\blacklozenge}_{\mathsf{CM}}\big(X,\cycle{q}\big)^{(\C^*)^2} 
=
\hspace{-0.25em}\raisebox{0.4em}{$ \coprod\limits_{\mbox{\tiny$\begin{array}{c} \bm{\alpha}\vdash \bm{a}, ~\bm{\gamma}\vdash \bm{c},~ \bm{\delta}\vdash \bm{d},\\ \bm{\lambda}\vdash \bm{l},~ \bm{\mu}\vdash \bm{m},  ~\bm{\nu}\vdash \bm{n}\end{array}$}}$} 
\hspace{-0.25em}
\Big\{ C_{\bm{\alpha},\bm{\gamma},\bm{\delta},\bm{\lambda},\bm{\mu},\bm{\nu}}\Big\} 
\end{align*}
\end{re}
Using the result of \ref{kappa_constructible_invariant_lemma} we have
\begin{align*}
e\Big(\HilbCyc^{\bullet}(X,\cycle{q})\Big) 
&=
e\Big(\Hilb^{\blacklozenge}_{\mathsf{CM}}(X,\cycle{q})^{(\C^*)^2}, \kappa_*1\Big) \\
&=
\raisebox{0.4em}{$ \sum\limits_{\mbox{\tiny$\begin{array}{c} \bm{\alpha}\vdash \bm{a}, ~\bm{\gamma}\vdash \bm{c},~ \bm{\delta}\vdash \bm{d},\\ \bm{\lambda}\vdash \bm{l},~ \bm{\mu}\vdash \bm{m},  ~\bm{\nu}\vdash \bm{n}\end{array}$}}$} 
e\Big((\Hilb^{\bullet}(X,C_{\bm{\alpha},\bm{\gamma},\bm{\delta},\bm{\lambda},\bm{\mu},\bm{\nu}})\Big)
p^{\chi(\O_{C_{\bm{\alpha},\bm{\gamma},\bm{\delta},\bm{\lambda},\bm{\mu},\bm{\nu}}})}\\
&=
\raisebox{0.4em}{$ \sum\limits_{\mbox{\tiny$\begin{array}{c} \bm{\alpha}\vdash \bm{a}, ~\bm{\gamma}\vdash \bm{c},~ \bm{\delta}\vdash \bm{d},\\ \bm{\lambda}\vdash \bm{l},~ \bm{\mu}\vdash \bm{m},  ~\bm{\nu}\vdash \bm{n}\end{array}$}}$} 
e\Big(\Quot_X^\bullet (I_{\bm{\alpha},\bm{\gamma},\bm{\delta},\bm{\lambda},\bm{\mu},\bm{\nu}})\Big)
p^{\chi(\O_{C_{\bm{\alpha},\bm{\gamma},\bm{\delta},\bm{\lambda},\bm{\mu},\bm{\nu}}})}.
\end{align*}

\vspace{0.2cm}
\begin{re}
Using the decomposition method of \ref{decomposition_method}
following method:
\begin{enumerate}[label={\arabic*)}]
\item Decompose $X$ by $X  =W ~\amalg~ C_{\bm{\alpha},\bm{\gamma},\bm{\delta},\bm{\lambda},\bm{\mu},\bm{\nu}}$ where $W:= X\setminus C_{\bm{\alpha},\bm{\gamma},\bm{\delta},\bm{\lambda},\bm{\mu},\bm{\nu}}$. 
\item Let $C_{\bm{\alpha},\bm{\gamma},\bm{\delta},\bm{\lambda},\bm{\mu},\bm{\nu}}^\diamond$ be set points  given by the following disjoint sets:
\begin{enumerate}[label={\alph*)}]
\item $\sigma^\diamond_\alpha := \sigma \cap C^{\mathsf{red}}_{\bm{\alpha}}$
\item $\sigma^\diamond_\gamma := \sigma \cap C^{\mathsf{red}}_{\bm{\gamma}}$
\item $C^\diamond_{\bm{\gamma}}$ the set of nodes of $C_{\bm{\gamma}}$
\item $C^\diamond_{\bm{\lambda}}$ the set of nodes of $C_{\bm{\lambda}}$
\item $B^\diamond =  \bigcup_i (C^{(i)}_2\cap C^{(i)}_3)$.  
\end{enumerate}
\item Denote the components supported on smooth reduced sub-curves by:
\begin{enumerate}[label={\alph*)}]
\item $\sigma^\circ := \sigma \setminus \sigma^\diamond$
\item $C_{\bm{\alpha}}^\circ :=C_{\bm{\alpha}} \setminus \sigma_\alpha^\diamond$
\item $C_{\bm{\gamma}}^\circ :=C_{\bm{\gamma}} \setminus (\sigma_\gamma^\diamond \cup C_{\bm{\lambda}}^\diamond)$
\item $C_{\bm{\lambda}}^\circ:=C_{\bm{\lambda}} \setminus   C_{\bm{\lambda}}^\diamond$
\item $C_{\bm{\mu}}^\circ:=C_{\bm{\mu}} \setminus B^\diamond$
\item $C_{\bm{\nu}}^\circ:=C_{\bm{\nu}} \setminus B^\diamond$
\end{enumerate}
\end{enumerate}

\end{re}

\vspace{0.2cm}
\begin{re}
Then applying Euler characteristic to lemma \ref{quot_decomposition_lemma} we have:
\begin{align*}
e\big(&\Quot_X^{\bullet}(I_{\bm{\alpha},\bm{\gamma},\bm{\delta},\bm{\lambda},\bm{\mu},\bm{\nu}}) \big)\\ 
=~&
 e\big(\Quot_X^{\bullet}(I_{\bm{\alpha},\bm{\gamma},\bm{\delta},\bm{\lambda},\bm{\mu},\bm{\nu}},W) \big)
 e\big(\Quot_X^{\bullet}(I_{\bm{\alpha},\bm{\gamma},\bm{\delta},\bm{\lambda},\bm{\mu},\bm{\nu}},\sigma^\circ )\big)\\
& e\big(\Quot_X^{\bullet}(I_{\bm{\alpha},\bm{\gamma},\bm{\delta},\bm{\lambda},\bm{\mu},\bm{\nu}},\sigma^\diamond_{\bm{\alpha}})\big) 
e\big(\Quot_X^{\bullet}(I_{\bm{\alpha},\bm{\gamma},\bm{\delta},\bm{\lambda},\bm{\mu},\bm{\nu}},C^\circ_{\bm{\alpha}})\big) \\
& e\big(\Quot_X^{\bullet}(I_{\bm{\alpha},\bm{\gamma},\bm{\delta},\bm{\lambda},\bm{\mu},\bm{\nu}},\sigma^\diamond_{\bm{\gamma}})\big) 
e\big(\Quot_X^{\bullet}(I_{\bm{\alpha},\bm{\gamma},\bm{\delta},\bm{\lambda},\bm{\mu},\bm{\nu}},C^\diamond_{\bm{\gamma}})\big)
e\big(\Quot_X^{\bullet}(I_{\bm{\alpha},\bm{\gamma},\bm{\delta},\bm{\lambda},\bm{\mu},\bm{\nu}},C^\circ_{\bm{\gamma}})\big)\\
& e\big(\Quot_X^{\bullet}(I_{\bm{\alpha},\bm{\gamma},\bm{\delta},\bm{\lambda},\bm{\mu},\bm{\nu}},C_{\bm{\delta}})\big)\\
&e\big(\Quot_X^{\bullet}(I_{\bm{\alpha},\bm{\gamma},\bm{\delta},\bm{\lambda},\bm{\mu},\bm{\nu}},C^\diamond_{\bm{\lambda}})\big)
e\big(\Quot_X^{\bullet}(I_{\bm{\alpha},\bm{\gamma},\bm{\delta},\bm{\lambda},\bm{\mu},\bm{\nu}},C^\circ_{\bm{\lambda}})\big)\\
& e\big(\Quot_X^{\bullet}(I_{\bm{\alpha},\bm{\gamma},\bm{\delta},\bm{\lambda},\bm{\mu},\bm{\nu}},B^\diamond)\big)
 e\big(\Quot_X^{\bullet}(I_{\bm{\alpha},\bm{\gamma},\bm{\delta},\bm{\lambda},\bm{\mu},\bm{\nu}},C^\circ_{\bm{\mu}})\big)
  e\big(\Quot_X^{\bullet}(I_{\bm{\alpha},\bm{\gamma},\bm{\delta},\bm{\lambda},\bm{\mu},\bm{\nu}},C^\circ_{\bm{\nu}})\big)\\
\end{align*}

\end{re}

\vspace{0.2cm}
\begin{lemma}\label{hol_euler_char_of_fibres}
The holomorphic Euler characteristic of $C_{\bm{\alpha},\bm{\gamma},\bm{\delta},\bm{\lambda},\bm{\mu},\bm{\nu}}$ is:
\begin{align*}
\chi(\O_{C_{\bm{\alpha},\bm{\gamma},\bm{\delta},\bm{\lambda},\bm{\mu},\bm{\nu}}})
=~ 1 +  \sum_i \chi\big(\O_{\mu^{(i)} C_2^{(i)} \cup \,\nu^{(i)}C_3^{(i)}}\big) - \sum_i \alpha^{(i)}- \sum_i \gamma^{(i)}
\end{align*}
and we have
\begin{align*}
&p^{\chi\big(\O_{\mu^{(i)} C_2^{(i)} \cup \,\nu^{(i)}C_3^{(i)}}\big)}\widetilde{\mathsf{V}}_{\mu^{(i)}\nu^{(i)}\emptyset}\widetilde{\mathsf{V}}_{(\mu^{(i)})^t(\nu^{(i)})^t\emptyset}  \\
&=
 p^{\frac{1}{2}(\|\mu^{(i)}\|^2+\|(\mu^{(i)})^t\|^2+\|\nu^{(i)}\|^2+\|(\nu^{(i)})^t\|^2)} \mathsf{V}_{\mu^{(i)}\nu^{(i)}\emptyset}\mathsf{V}_{(\mu^{(i)})^t(\nu^{(i)})^t\emptyset}. \\
\end{align*}

\end{lemma}
\begin{proof}
Define the sheaves:
\begin{itemize}
\item[] $\mathcal{F}^{\sigma} := \O_\sigma \oplus \big(\bigoplus_i \O_{\alpha^{(i)} f_{x_i}} \big)\oplus  \big(\bigoplus_i \O_{\beta^{(i)} f_{y_i}}\big)$
\item[] $\mathcal{F}^{\emptyset} := \big(\bigoplus_i \O_{\delta^{(i)} f_{z_i}} \big)\oplus  \big(\bigoplus_i \O_{\lambda^{(i)} f_{w_i}}\big)$
\item[] $\mathcal{F}^{\cap} := \big(\bigoplus_i \O_{\sigma\cap \alpha^{(i)} f_{x_i}} \big)\oplus  \big(\bigoplus_i \O_{\sigma\cap \beta^{(i)} f_{y_i}}\big)$
\item[] $\mathcal{B} := \bigoplus_i \O_{\mu^{(i)} C_2^{(i)} \cup \,\nu^{(i)}C_3^{(i)}}$
\end{itemize}
   
The exact sequence decomposing $C_{\bm{\alpha},\bm{\gamma},\bm{\delta},\bm{\lambda},\bm{\mu},\bm{\nu}}$ is then
\[
0 \longrightarrow \O_{C_{\bm{\alpha},\bm{\gamma},\bm{\delta},\bm{\lambda},\bm{\mu},\bm{\nu}}} 
\longrightarrow 
\mathcal{F}^{\sigma} \oplus \mathcal{F}^{\emptyset} \oplus \mathcal{B} 
\longrightarrow 
  \mathcal{F}^{\cap} 
\longrightarrow 0 
\]
The result now follows from \ref{holomorphic_euler_char_of_PTC} and  the fact that $\chi(\O_{\sigma})=1$.
\end{proof}

\begin{re}

\vspace{0.5cm}
Applying lemmas \ref{hol_euler_char_of_fibres}, \ref{Quot_to_V_lemma} and \ref{Quot_to_V_Sym_lemma} we have:
\begin{align*}
e\big(\Quot&_X^{\bullet}(I_{\bm{\alpha},\bm{\gamma},\bm{\delta},\bm{\lambda},\bm{\mu},\bm{\nu}}) \big)p^{\chi(\O_{C_{\bm{\alpha},\bm{\gamma},\bm{\delta},\bm{\lambda},\bm{\mu},\bm{\nu}}})}\\
=~
 p&\cdot \Big(\widetilde{\mathsf{V}}_{\emptyset\emptyset\emptyset}\Big)^{e(W)}
\Big(\widetilde{\mathsf{V}}_{\square\emptyset\emptyset}\Big)^{e(\sigma^\circ)}\\
&\cdot\prod_{i}\Big(p^{-\alpha_i}\widetilde{\mathsf{V}}_{\alpha^{(i)}\square\emptyset}\Big)
\prod_{i}\Big(\widetilde{\mathsf{V}}_{\alpha^{(i)}\emptyset\emptyset}\Big)^{-1} \\
&\cdot \prod_{i}\Big(p^{-\gamma_i}\widetilde{\mathsf{V}}_{\gamma^{(i)}\square\emptyset}\Big)
\prod_{i}\Big(\widetilde{\mathsf{V}}_{\gamma^{(i)}(\gamma^{(i)})^t\emptyset}\Big)
\prod_{i}\Big(\widetilde{\mathsf{V}}_{\gamma^{(i)}\emptyset\emptyset}\Big)^{-1} \\
&\cdot \prod_{i}\Big(\widetilde{\mathsf{V}}_{\delta^{(i)}\emptyset\emptyset}\Big)^{0} \\
&\cdot \prod_{i}\Big(\widetilde{\mathsf{V}}_{\lambda^{(i)}(\lambda^{(i)})^t\emptyset}\Big)
\prod_{i}\Big(\widetilde{\mathsf{V}}_{\lambda^{(i)}\emptyset\emptyset}\Big)^{0} \\
&\cdot \prod_{i}\Big(p^{\chi\big(\O_{\mu^{(i)} C_2^{(i)} \cup \,\nu^{(i)}C_3^{(i)}}\big)}\widetilde{\mathsf{V}}_{\mu^{(i)}\nu^{(i)}\emptyset}\widetilde{\mathsf{V}}_{(\mu^{(i)})^t(\nu^{(i)})^t\emptyset}\Big) \prod_{i}\Big(\widetilde{\mathsf{V}}_{\mu^{(i)}\emptyset\emptyset}\Big)^{0} \prod_{i}\Big(\widetilde{\mathsf{V}}_{\nu^{(i)}\emptyset\emptyset}\Big)^{0}.\\
\end{align*}

We note that $e(X) = 24$ and $e(\sigma)=2$, so from lemma \ref{Vtilde_to_V_lemma} we now have
 we have 
\begin{align*}
e\big(\Quot&_X^{\bullet}(I_{\bm{\alpha},\bm{\gamma},\bm{\delta},\bm{\lambda},\bm{\mu},\bm{\nu}}) \big)p^{\chi(\O_{C_{\bm{\alpha},\bm{\gamma},\bm{\delta},\bm{\lambda},\bm{\mu},\bm{\nu}}})}\\
=~p&\cdot\Big(\mathsf{V}_{\emptyset\emptyset\emptyset}\Big)^{24}
\Big(\frac{\mathsf{V}_{\square\emptyset\emptyset}}{\mathsf{V}_{\emptyset\emptyset\emptyset}}\Big)^{2}
\prod_{i}\Big(\frac{\mathsf{V}_{\emptyset\emptyset\emptyset}}{\mathsf{V}_{\square\emptyset\emptyset}}\frac{\mathsf{V}_{\alpha^{(i)}\square\emptyset}}{\mathsf{V}_{\alpha^{(i)}\emptyset\emptyset}}\Big)\\
& \cdot\prod_{i}\Big(\frac{\mathsf{V}_{\emptyset\emptyset\emptyset}}{\mathsf{V}_{\square\emptyset\emptyset}}p^{\|\gamma^{(i)}\|^2}\frac{\mathsf{V}_{\gamma^{(i)}\square\emptyset}\mathsf{V}_{\gamma^{(i)}(\gamma^{(i)})^t\emptyset}}{\mathsf{V}_{\emptyset\emptyset\emptyset}\mathsf{V}_{\gamma^{(i)}\emptyset\emptyset}}\Big)
\prod_{i}\Big(p^{\|\lambda^{(i)}\|^2}\frac{\mathsf{V}_{\lambda^{(i)}(\lambda^{(i)})^t\emptyset}}{\mathsf{V}_{\emptyset\emptyset\emptyset}}\Big)\\
&\cdot\prod_{i}\Big( p^{\frac{1}{2}(\|\mu^{(i)}\|^2+\|(\mu^{(i)})^t\|^2+\|\nu^{(i)}\|^2+\|(\nu^{(i)})^t\|^2)}\frac{\mathsf{V}_{\mu^{(i)}\nu^{(i)}\emptyset}\mathsf{V}_{(\mu^{(i)})^t(\nu^{(i)})^t\emptyset}}{\mathsf{V}_{\emptyset\emptyset\emptyset}\mathsf{V}_{\emptyset\emptyset\emptyset}}\Big)\\
\end{align*}

We now define the functions:
\begin{enumerate}[label={\arabic*)}]
\item
$g_{\mathtt{Sm}^\sigma} :  \mathbb{Z}_{\geq 0}\longrightarrow Z(\!(p)\!)$ is defined by  $g_{\mathtt{Sm}^\sigma} (a) =  \frac{\mathsf{V}_{\emptyset\emptyset\emptyset}}{\mathsf{V}_{\square\emptyset\emptyset}}\sum\limits_{\alpha \vdash a}\frac{\mathsf{V}_{\alpha\square\emptyset}}{\mathsf{V}_{\alpha^{(i)}\emptyset\emptyset}}$,
\item
$g_{\mathtt{N}^\sigma}: \mathbb{Z}_{\geq 0} \longrightarrow Z(\!(p)\!)$ is defined by  $g_{\mathtt{N}^\sigma}(c) =  \frac{\mathsf{V}_{\emptyset\emptyset\emptyset}}{\mathsf{V}_{\square\emptyset\emptyset}}\sum\limits_{\gamma\vdash c} p^{\|\gamma^{(i)}\|^2} \frac{\mathsf{V}_{\gamma\square\emptyset}\mathsf{V}_{\gamma\gamma^t\emptyset}}{\mathsf{V}_{\emptyset\emptyset\emptyset}\mathsf{V}_{\gamma\emptyset\emptyset}}$,
\item
$g_{\mathtt{Sm}^\emptyset}:\mathbb{Z}_{\geq 0} \longrightarrow Z(\!(p)\!)$ is defined by  $ g_{\mathtt{Sm}^\emptyset}(d)=  \sum\limits_{\delta \vdash d} 1$,
\item
$g_{\mathtt{N}^{\emptyset}} : \mathbb{Z}_{\geq 0}\longrightarrow Z(\!(p)\!)$ is defined by $g_{\mathtt{N}^{\emptyset}} (l) =\sum\limits_{\lambda \vdash l} p^{\|\lambda^{(i)}\|^2}\frac{\mathsf{V}_{\lambda  \lambda^t\emptyset}}{\mathsf{V}_{\emptyset\emptyset\emptyset}}$,
\item
$g_B:\mathbb{Z}_{\geq 0} \times \mathbb{Z}_{\geq 0} \longrightarrow Z(\!(p)\!)$ is defined by the equation \\
$ g_B(m,n) =\hspace{-0.75em} 
\sum\limits_{\mbox{\tiny $\begin{array}{c}\mu \vdash m \\ \nu \vdash n\end{array}$}} \hspace{-0.75em} 
p^{\frac{1}{2}(\|\mu^{(i)}\|^2+\|\mu^t\|^2+\|\nu\|^2+\|\nu^t\|^2)}\frac{\mathsf{V}_{\mu\nu\emptyset}\mathsf{V}_{\mu^t\nu^t\emptyset}}{\mathsf{V}_{\emptyset\emptyset\emptyset}\mathsf{V}_{\emptyset\emptyset\emptyset}}$.
\end{enumerate}

So the constructible function $(\rho_\bullet)_*1: \Chow^{\sigma + (0,\bullet,\bullet)}(X) \rightarrow Z(\!(p)\!)$ is calculated for $\cycle{q} = (\bm{a}\bm{x},\bm{c}\bm{y},\bm{d}\bm{z},\bm{l}\bm{w}, \bm{m}\bm{b_{2}}, \bm{n}\bm{b_{\mathsf{op}}})$ by:
\begin{align*}
\big((\rho_\bullet)&_*1\big)(\cycle{q}) \\
=~& e\Big(\rho_\bullet^{-1}(\cycle{q} )\Big) \\
=~&
\raisebox{0.4em}{$ \sum\limits_{\mbox{\tiny$\begin{array}{c} \bm{\alpha}\vdash \bm{a}, ~\bm{\gamma}\vdash \bm{c},~ \bm{\delta}\vdash \bm{d},\\ \bm{\lambda}\vdash \bm{l},~ \bm{\mu}\vdash \bm{m},  ~\bm{\nu}\vdash \bm{n}\end{array}$}}$} 
e\Big(\Quot_X^\bullet (I_{\bm{\alpha},\bm{\gamma},\bm{\delta},\bm{\lambda},\bm{\mu},\bm{\nu}})\Big)
p^{\chi(\O_{C_{\bm{\alpha},\bm{\gamma},\bm{\delta},\bm{\lambda},\bm{\mu},\bm{\nu}}})}\\
=~&p\cdot\Big(\mathsf{V}_{\emptyset\emptyset\emptyset}\Big)^{24}
\Big(\frac{\mathsf{V}_{\square\emptyset\emptyset}}{\mathsf{V}_{\emptyset\emptyset\emptyset}}\Big)^{2}
\prod_{i} g_{\mathtt{Sm}^\sigma} (a_i) \prod_{i} g_{\mathtt{N}^\sigma} (c_i) \prod_{i}  g_{\mathtt{Sm}^\emptyset} (d_i) \prod_{i} g_{\mathtt{N}^\emptyset} (l_i) \prod_{i} g_{B} (m_i,n_i). 
\end{align*}
\end{re}

So we can now apply lemma \ref{BK_sym_product_lemma} and reintroduce the formal variables $Q_2$ and $Q_3$ to obtain: 
\begin{align*}
e\Big(\Chow&^{\sigma + (0,\bullet,\bullet)}(X),  (\rho_\bullet)_*1) \Big)\\
=~
  p&\cdot\Big(\mathsf{V}_{\emptyset\emptyset\emptyset}\Big)^{24}
\Big(\frac{\mathsf{V}_{\square\emptyset\emptyset}}{\mathsf{V} _{\emptyset\emptyset\emptyset}}\Big)^{2} 
\Big(\frac{\mathsf{V}_{\emptyset\emptyset\emptyset}}{\mathsf{V}_{\square\emptyset\emptyset}}\sum\limits_{\alpha} (Q_2Q_3)^{|\alpha|} \frac{\mathsf{V}_{\alpha\square\emptyset}}{\mathsf{V}_{\alpha\emptyset\emptyset}}\Big)^{e(\mathtt{Sm}^\sigma)}\\
&\cdot\Big(\frac{\mathsf{V}_{\emptyset\emptyset\emptyset}}{\mathsf{V}_{\square\emptyset\emptyset}}\sum\limits_{\gamma}(Q_2Q_3)^{|\gamma|} p^{\|\gamma\|^2}\frac{\mathsf{V}_{\gamma\square\emptyset}\mathsf{V}_{\gamma\gamma^t\emptyset}}{\mathsf{V}_{\emptyset\emptyset\emptyset}\mathsf{V}_{\gamma\emptyset\emptyset}} \Big)^{e(\mathtt{N}^\sigma)}\\
&\cdot\Big( \sum\limits_{\delta} (Q_2Q_3)^{|\delta|}\Big)^{e(\mathtt{Sm}^\emptyset)}
\Big( \sum\limits_{\lambda}  (Q_2Q_3)^{|\lambda|}p^{\|\lambda\|^2}\frac{\mathsf{V}_{\lambda  \lambda^t\emptyset}}{\mathsf{V}_{\emptyset\emptyset\emptyset}} \Big)^{e(\mathtt{N}^\emptyset)}\\
&\cdot e\Big(\Sym_{Q_2}^{\bullet}(B_{2})\times \Sym_{Q_3}^{\bullet}(B_{\mathsf{op}}), G_B \Big) 
\end{align*}
where $G_B$ is the constructible function
\[
G_B:  \Sym_{Q_2}^{\bullet}(B_{2})\times \Sym_{Q_3}^{\bullet}(B_{\mathsf{op}}) \rightarrow \mathbb{Z}(\!( p)\!)
\]
defined by $G_B(\bm{m b_{2}},\bm{n b_{\mathsf{op}}}):= \prod\limits_{i=1}^{12} g_B(m_i,n_i)$. However, since $B_{2} = \{b_{2}^1,\ldots,b_{2}^{12}\}$ and $B_{\mathsf{op}} = \{b_{\mathsf{op}}^1,\ldots,b_{\mathsf{op}}^{12}\}$ we have:
\[
\Sym_{Q_2}^{\bullet}(B_{2})\times \Sym_{Q_3}^{\bullet}(B_{\mathsf{op}}) \cong \prod_{i=1}^{12} \Sym_{Q_2}^{\bullet}\Big(\{b_{2}^{(i)}\}\Big)\times \Sym_{Q_3}^{\bullet}\Big(\{b_{\mathsf{op}}^{(i)}\}\Big).
\]
Defining $G^{(i)}_B:= G_B|_{ \Sym_{Q_2}^{\bullet}(\{b_{2}^{(i)}\})\times \Sym_{Q_3}^{\bullet}(\{b_{\mathsf{op}}^{(i)}\})}$ gives us:
\begin{align*}
e\Big(\Sym&_{Q_2}^{\bullet}(B_{2})\times \Sym_{Q_3}^{\bullet}(B_{\mathsf{op}}), G_B \Big)  \\
=~& \prod_{i=1}^{12}e\Big( \Sym_{Q_2}^{\bullet}\Big(\{b_{2}^{(i)}\}\Big)\times \Sym_{Q_3}^{\bullet}\Big(\{b_{\mathsf{op}}^{(i)}\}\Big), G^{(i)}_{B} \Big) \\
=~& \left(  \sum_{\mu,\nu} Q_2^{|\mu|}Q_3^{|\nu|} p^{\frac{1}{2}(\|\mu\|^2+\|\mu^t\|^2+\|\nu\|^2+\|\nu^t\|^2)} \frac{\mathsf{V}_{\mu \nu\emptyset}\mathsf{V}_{\mu^t\nu^t\emptyset}}{\mathsf{V}_{\emptyset\emptyset\emptyset}\mathsf{V}_{\emptyset\emptyset\emptyset}}\right)^{12}
\end{align*}

\begin{re}
Applying the vertex formulas of lemmas \ref{vertex_trace_formulas}, \ref{vertex_squares_lemma} and  \ref{banana_subpartition_functions} we have
\begin{align*}
e\Big(\Chow&^{\sigma + (0,\bullet,\bullet)}(X),  (\rho_\bullet)_*1) \Big)\\
=~&
  M(p)^{24}
\frac{p}{\left(1-p\right)^{2} }
\left(\prod\limits_{d>0}\dfrac{(1-Q_2^dQ_3^d)}{(1-p Q_2^dQ_3^d)(1-p^{-1}Q_2^dQ_3^d)}\right)^{-10}\\
&\cdot\left(\prod\limits_{d>0}\dfrac{M(p,Q_2^dQ_3^d)}{(1-p Q_2^dQ_3^d)(1-p^{-1}Q_2^dQ_3^d)} \right)^{12}\\
&\cdot\left( \prod\limits_{d>0}\dfrac{1}{(1-Q_2^dQ_3^d)}\right)^{10}
\left( \prod\limits_{d>0}\dfrac{M(p,Q_2^dQ_3^d)}{(1-Q_2^dQ_3^d)}\right)^{-12}\\
&\cdot\left(\prod\limits_{d>0} \dfrac{M(p,Q_2^d Q_3^d)^{2}}{(1-Q_2^d  Q_3^d) M(p,-Q_2^{d-1}Q_3^d) M(p,-Q_2^d Q_3^{d-1})}\right)^{12} \\
=~&
\frac{  M(p)^{24}}{(1-p)^2}\prod_{d>0} \dfrac{1}{(1-Q_2^dQ_3^d)^8(1-p Q_2^dQ_3^d)^2(1-p^{-1}Q_2^dQ_3^d)^2}\\
&\cdot\left(\prod\limits_{d>0} \dfrac{M(p,Q_2^d Q_3^d)^{2}}{(1-Q_2^d  Q_3^d) M(p,-Q_2^{d-1}Q_3^d) M(p,-Q_2^d Q_3^{d-1})}\right)^{12} \\
\end{align*}
\end{re}

Which completes the proof of theorem \ref{main_DT_calc_theorem_A}. \\

\subsection{Preliminaries for classes of the form \texorpdfstring{$\bullet\sigma + (i,j,\bullet)$}{*s + (i,j,*)}}\label{bulletsigma+(i,j,bullet)_cycle_preliminaries}\label{main_computation_section_B_prelim}

We recall from lemma \ref{classes_bs+(i,j,d)_main_chow_main_lemma} that there is a decomposition of  $\Chow^{\bullet\sigma+(i,j,\bullet)}(X)$ such that for any point $\cycle{q}\in \Chow^{\bullet\sigma+(i,j,\bullet)}(X)$ the fibre is 
\[
(\eta_\bullet)^{-1}(\cycle{q}) \cong \HilbCyc^{\bullet}(X,\Cyc(C))
\]
for some one dimensional subscheme $C$ of $X$ with 
\begin{align}
\Cyc(C)= \cycle{q} =~& a \sigma + D +\sum_{i=1}^{12} m_i C_3^{(i)}  \label{bulletsigma+(i,j,bullet)_cycle_notation}
\end{align}
where $D$ is a one dimensional \textit{reduced} subscheme of $X$. We see from lemma  \ref{classes_bs+(i,j,d)_main_chow_main_lemma} that the intersection of $D$ with $\sigma$ has length $0$, $1$ or $2$. We consider the following formal neighbourhoods around components of $C$: 
\begin{figure}
  \centering
      \includegraphics[width=1\textwidth]{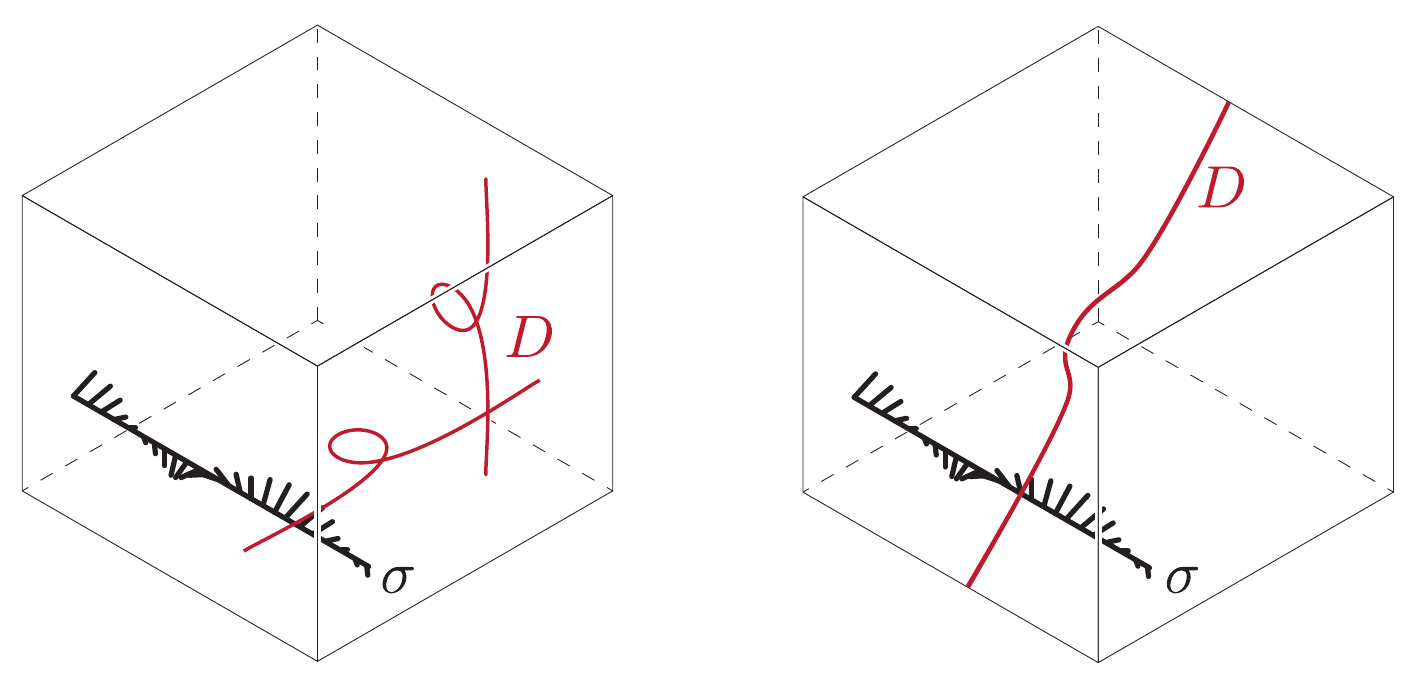}
      \vspace{-1.25em}
  \caption[Cohen-Macaulay subschemes in the fibres of $\eta_\bullet$]{Depiction of two typical curves (away from $C_3$) in the class $b\sigma +(1,1,d)$. }
\end{figure}
\begin{enumerate}[label={\arabic*)}]
\item Let $U_i$ be the formal neighbourhood of $C_3^{(i)}$ in $X$. These have a canonical $(\C^*)^2$-action described in \ref{lambda_thickened_C3_defn} and the $(\C^*)^2$-invariance of $D\cap U_i$ is shown in lemma \ref{linear_system_divisor_C*_invariant}. 
\item Let $V$ be the formal neighbourhood of $\sigma$ in $X$ with the coordinates:
\begin{enumerate}[label={\alph*)}]
\item If $\#(D\cap\sigma) = 0,2$ the let $V$ have the canonical coordinates of \ref{canonical_section_coords} of and $(\C^*)^2$-action described in \ref{lambda_thickened_section_defn}.
\item If $\#(D\cap\sigma) = 1$ the let $V$ have the canonical coordinates of \ref{canonical_section_coords_relative} of and $(\C^*)^2$-action described in \ref{lambda_thickened_section_defn}.
\end{enumerate}
\end{enumerate}
By construction the restrictions of $D$ to these neighbourhoods are invariant under these actions. Hence the conditions of lemma \ref{action_on_CM_lemma} are satisfied and there is a $(\C^*)^2$-action defined on $\Hilb^{n}_{\mathsf{CM}}(X,\Cyc(C))$. We introduce the notation for subschemes of $X$: 
\[
C_{\alpha,\bm{\mu}} = C_{\alpha,\mu^{(1)},\ldots,\mu^{(12)}} =~ \alpha \sigma ~\cup~ D~ \cup \bigcup^{12}_{i=1}\mu_i C_3^{(i)}
\]
and their ideals $I_{\alpha,\bm{\mu}}$. Then using this notation we  can identify the fixed points of the action as the following discrete set:
\begin{align*}
\Hilb^{\blacklozenge}_{\mathsf{CM}}\big(X,\cycle{q}\big)^{(\C^*)^2} 
=
\coprod_{\alpha\vdash a,~\bm{\mu}\vdash\bm{m}}
\Big\{~ C_{\alpha,\bm{\mu}}~\Big\}.
\end{align*}

Using the result of \ref{kappa_constructible_invariant_lemma} we have
\begin{align*}
e\Big(\HilbCyc^{\bullet}(X,\cycle{q})\Big) 
&=
e\Big(\Hilb^{\blacklozenge}_{\mathsf{CM}}(X,\cycle{q})^{(\C^*)^2}, \kappa_*1\Big) \\
&=
\sum_{\alpha\vdash a,~\bm{\mu}\vdash\bm{m}}
 e\Big((\Hilb^{\bullet}(X,C_{\alpha,\bm{\mu}}\Big)p^{\chi(\O_{C_{\alpha,\bm{\mu}}})}\\
&=
\sum_{\alpha\vdash a,~\bm{\mu}\vdash\bm{m}}
e\Big(\Quot_X^\bullet (I_{\alpha,\bm{\mu}}\Big)p^{\chi(\O_{C_{\alpha,\bm{\mu}}})}.
\end{align*}

Where the holomorphic Euler characteristic $\chi(\O_{C_{\alpha,\bm{\mu}}})$ is given by the following lemma. 

\begin{lemma}
The holomorphic Euler characteristic of $C_{\alpha,\bm{\mu}}$ is:
\begin{align*}
\chi(\O_{C_{\alpha,\bm{\mu}}}) 
=~  \chi(\O_D) &+ \Big(\chi(\O_{\alpha \sigma})- |D\cap  {\alpha \sigma}| \Big)\\
&+ \Big(\sum_{i=1}^{12} \chi(\O_{\mu^{(i)} C_3^{(i)}}) - \sum_{i=1}^{12} |D \cap \mu^{(i)} C_3^{(i)}| \Big).
\end{align*}
\end{lemma}
\begin{proof}
This is immediate from the exact sequence decomposing $C_{\alpha,\bm{\mu}}$ into irreducible components:
\[
0 \rightarrow \O_{C_{\alpha,\bm{\mu}}} \rightarrow \O_D\oplus \O_{\alpha \sigma}\oplus  \bigoplus_i \O_{\mu^{(i)} C_3^{(i)}}  \rightarrow \O_{D\,\cap\, \alpha \sigma} \oplus  \bigoplus_i  \O_{D\,\cap\, \mu^{(i)} C_3^{(i)}}   \rightarrow 0 
\]
\end{proof}

\begin{re}
Using the decomposition method of \ref{decomposition_method}
we take the following steps:
\begin{enumerate}[label={\arabic*)}]
\item Decompose $X$ by $X  =W ~\amalg~ C_{\alpha,\bm{\mu}}$ where $W:= X\setminus C_{\alpha,\bm{\mu}}$. 
\item Let $C_{\alpha,\bm{\mu}}^\diamond$ be set points  given by the following disjoint sets:
\begin{enumerate}[label={\alph*)}]
\item $D^\diamond$ is the set of nodes of $D\setminus ( \sigma \cup_i C_3^{(i)})$.
\item $D^*$ is the set singularities of $D\setminus ( \sigma \cup_i C_3^{(i)})$ that are locally isomorphic it the coordinate axes in $\C^3$.  
\item $\sigma^\diamond := \sigma \cap D$,
\item $B_i^\diamond =  (C^{(i)}_3\cap D)$ for $i\in\{1,\ldots,12\}$,  
\end{enumerate}
Note that $D^\diamond \cup D^*$ is the set of singularities of $D\setminus ( \sigma \cup_i C_3^{(i)})$. 
\item Denote the components supported on smooth reduced sub-curves by:
\begin{enumerate}[label={\alph*)}]
\item $D^\circ = D  \setminus (D^\diamond \cup D^*)$,
\item $\sigma^\circ := \sigma \setminus \sigma^\diamond$,
\item $B_i^\circ = \C_3^{(i)}  \setminus B_i^\diamond$ for $i\in\{1,\ldots,12\}$.\\
\end{enumerate}
\end{enumerate}
\end{re}

\begin{re}\label{quot_decomposition_(i,j,bullet)}
Then applying Euler characteristic to lemma \ref{quot_decomposition_lemma} we have:
\begin{align*}
 e\Big(\Quot&_X^\bullet (I_{\alpha,\bm{\mu}}\Big) p^{\chi(\O_{C_{\alpha,\bm{\mu}}}) } \\
=~&  e\Big(\Quot_X^\bullet (I_{\alpha,\bm{\mu}},W\Big) \num\label{euler_decomp_1}\\
&p^{\chi(\O_{D}) } e\Big(\Quot_X^\bullet (I_{\alpha,\bm{\mu}},D^\circ \Big) e\Big(\Quot_X^\bullet (I_{\alpha,\bm{\mu}},D^\diamond \Big)e\Big(\Quot_X^\bullet (I_{\alpha,\bm{\mu}},D^* \Big)  \num\label{euler_decomp_2}\\
& p^{\chi(\O_{\alpha \sigma})- |D\cap  {\alpha \sigma}| }e\Big(\Quot_X^\bullet (I_{\alpha,\bm{\mu}},\sigma^\circ \Big) e\Big(\Quot_X^\bullet (I_{\alpha,\bm{\mu}},\sigma^\diamond \Big) \num \label{euler_decomp_3}\\
&\prod_{i=1}^{12}p^{ \chi(\O_{\mu^{(i)} C_3^{(i)}}) - |D \cap \mu^{(i)} C_3^{(i)}| } e\Big(\Quot_X^\bullet (I_{\alpha,\bm{\mu}},B_i^\circ \Big)  e\Big(\Quot_X^\bullet (I_{\alpha,\bm{\mu}},B_i^\diamond \Big) \num\label{euler_decomp_4}
\end{align*}\\
\end{re}

\begin{re} \label{quot_decomposition_(i,j,bullet)_V_tilde}
We have that $e(X)=24$ and $e(\sigma)=e(C_3^{(i)})=2$. So the Euler characteristic of $W$ is:
\begin{align*}
e(W)
&= e(X) - e(\sigma) -\sum_{i=1}^{12} e(C_3^{(i)}) - e(D^\circ) - e(D^\diamond)- e(D^*)\\
&= -2 - e(D^\circ) - e(D^\diamond)- e(D^*)
\end{align*}

Hence now have from lemma \ref{Quot_to_V_Sym_lemma} that lines (\ref{euler_decomp_1})-(\ref{euler_decomp_2}) from above will be:
\begin{align*}
\Psi(D):=&~p^{\chi(\O_{D}) } 
\Big(\widetilde{\mathsf{V}}_{\emptyset\emptyset\emptyset}\Big)^{e(W)}
\Big(\widetilde{\mathsf{V}}_{\square\emptyset\emptyset}\Big)^{e(D^\circ)}
\Big(\widetilde{\mathsf{V}}_{\square\square\emptyset}\Big)^{e(D^\diamond)}
\Big(\widetilde{\mathsf{V}}_{\square\square\square}\Big)^{e(D^*)}\\
=&~p^{\chi(\O_{D}) } 
\Big(\mathsf{V}_{\emptyset\emptyset\emptyset}\Big)^{-2}
\left(\frac{\mathsf{V}_{\square\emptyset\emptyset}}{\mathsf{V}_{\emptyset\emptyset\emptyset}}\right)^{e(D^\circ)}
\left(p\frac{\mathsf{V}_{\square\square\emptyset}}{\mathsf{V}_{\emptyset\emptyset\emptyset}}\right)^{e(D^\diamond)}
\left(p^2\frac{\mathsf{V}_{\square\square\square}}{\mathsf{V}_{\emptyset\emptyset\emptyset}}\right)^{e(D^*)}
\end{align*}
The intersection of $D$ and $\alpha \sigma$ will determine line (\ref{euler_decomp_3}) from above. From lemma \ref{Quot_to_V_Sym_lemma} and lemma \ref{Vtilde_to_V_lemma} it will be one of:
\begin{enumerate}[label={\arabic*)}]

\item $p^{\frac{1}{2}(\|\alpha\|^2+\| \alpha^t\|^2)} \Big(\widetilde{\mathsf{V}}_{ \emptyset \emptyset \alpha}\widetilde{\mathsf{V}}_{ \emptyset \emptyset \alpha}\Big) = p^{\frac{1}{2}(\|\alpha\|^2+\| \alpha^t\|^2)} \Big(\mathsf{V}_{ \emptyset \emptyset \alpha}\mathsf{V}_{ \emptyset \emptyset \alpha^t}\Big)$

\item $p^{\frac{1}{2}(\|\alpha\|^2+\| \alpha^t\|^2)-l(\alpha^t)} \Big(\widetilde{\mathsf{V}}_{   \square \emptyset  \alpha}\widetilde{\mathsf{V}}_{ \emptyset\emptyset\alpha}\Big) = p^{\frac{1}{2}(\|\alpha\|^2+\| \alpha^t\|^2)} \Big(\mathsf{V}_{   \square \emptyset  \alpha}\mathsf{V}_{ \emptyset\emptyset\alpha^t}\Big)$

\item $p^{\frac{1}{2}(\|\alpha\|^2+\| \alpha^t\|^2)-l(\alpha^t)-l(\alpha)} \Big(\widetilde{\mathsf{V}}_{  \square \emptyset  \alpha}\widetilde{\mathsf{V}}_{  \emptyset \square \alpha}\Big)= p^{\frac{1}{2}(\|\alpha\|^2+\| \alpha^t\|^2)} \Big(\mathsf{V}_{   \square \emptyset  \alpha}\mathsf{V}_{ \square\emptyset\alpha^t}\Big)$

\item $p^{\frac{1}{2}(\|\alpha\|^2+\| \alpha^t\|^2)-(l(\alpha)+l(\alpha^t)-1)} \Big(\widetilde{\mathsf{V}}_{   \square  \square  \alpha}\widetilde{\mathsf{V}}_{ \emptyset\emptyset\alpha}\Big)= p^{\frac{1}{2}(\|\alpha\|^2+\| \alpha^t\|^2)+1} \Big(\mathsf{V}_{   \square\square \alpha}\mathsf{V}_{ \emptyset \emptyset\alpha^t}\Big)$

\end{enumerate}
Similarly the factors of line (\ref{euler_decomp_4}) from above are determined by the intersections $D \cap C_3^{(i)}$ to be (the fourth comes from \ref{relative_C3_Vertex_remark}):
\begin{enumerate}[label={\arabic*)}]
\item $p^{\frac{1}{2}(\|\alpha\|^2+\| \alpha^t\|^2)} \Big(\widetilde{\mathsf{V}}_{ \emptyset \emptyset \alpha}\widetilde{\mathsf{V}}_{ \emptyset \emptyset \alpha^t}\Big) = p^{\frac{1}{2}(\|\alpha\|^2+\| \alpha^t\|^2)} \Big(\mathsf{V}_{ \emptyset \emptyset \alpha}\mathsf{V}_{ \emptyset \emptyset \alpha^t}\Big)$

\item $p^{\frac{1}{2}(\|\alpha\|^2+\| \alpha^t\|^2)-(l(\alpha^t)+l(\alpha))} \Big(\widetilde{\mathsf{V}}_{  \square \emptyset  \alpha}\widetilde{\mathsf{V}}_{ \square \emptyset  \alpha^t}\Big)= p^{\frac{1}{2}(\|\alpha\|^2+\| \alpha^t\|^2)} \Big(\mathsf{V}_{   \square \emptyset  \alpha}\mathsf{V}_{\square \emptyset\alpha^t}\Big)$

\item $p^{\frac{1}{2}(\|\alpha\|^2+\| \alpha^t\|^2)-(l(\alpha)+l(\alpha^t))} \Big(\widetilde{\mathsf{V}}_{  \emptyset   \square \alpha}\widetilde{\mathsf{V}}_{\emptyset   \square \alpha^t}\Big) = p^{\frac{1}{2}(\|\alpha\|^2+\| \alpha^t\|^2)} \Big(\mathsf{V}_{   \square \emptyset  \alpha}\mathsf{V}_{ \square \emptyset\alpha^t}\Big)$

\item $p^{\frac{1}{2}(\|\alpha\|^2+\| \alpha^t\|^2)-2l(\alpha^t)} \Big(\widetilde{\mathsf{V}}_{  \emptyset   \square \alpha}\Big)^2 = p^{\frac{1}{2}(\|\alpha\|^2+\| \alpha^t\|^2)} \Big(\mathsf{V}_{    \emptyset \square \alpha}\Big)^2$

\item $p^{\frac{1}{2}(\|\alpha\|^2+\| \alpha^t\|^2)-2(l(\alpha)+l(\alpha^t)-1)} \Big(\widetilde{\mathsf{V}}_{   \square  \square  \alpha}\widetilde{\mathsf{V}}_{ \square  \square\alpha^t}\Big)= p^{\frac{1}{2}(\|\alpha\|^2+\| \alpha^t\|^2)+2} \Big(\mathsf{V}_{   \square\square \alpha}\mathsf{V}_{ \square  \square\alpha^t}\Big)$
\end{enumerate}
\end{re}

\begin{re} \label{fibre_general_(i,j,bullet)}
We can calculate $e\big(\HilbCyc^{\bullet}(X,\cycle{q})\big)$ using the above results and notation from \ref{quot_decomposition_(i,j,bullet)_V_tilde}: 
\begin{align*}
e\Big(\HilbCyc^{\bullet}(X,\cycle{q})\Big) 
&=
\sum_{\alpha\vdash a,~\bm{\mu}\vdash\bm{m}}
p^{\chi(\O_{C_{\alpha,\bm{\mu}}}) }
e\Big(\Quot_X^\bullet (I_{\alpha,\bm{\mu}}\Big)\\
&= 
\Psi(D) \Phi_\sigma(a) \prod_{i=1}^{12} \Phi_i (m_i).
\end{align*}
where $\Phi_\sigma$ and $\Phi_i$ are determined by the intersections of $\sigma$ and $C_3^{(i)}$ respectively to be one of the following functions:
\begin{enumerate}[label={\arabic*)}]

\item $\Phi^{\emptyset,\emptyset}(a) 
:= \sum \limits_{\alpha \vdash a}p^{\frac{1}{2}(\|\alpha\|^2+\| \alpha^t\|^2)} (\mathsf{V}_{ \emptyset \emptyset \alpha}\mathsf{V}_{ \emptyset \emptyset \alpha^t})$\\

\item $\Phi^{-,\emptyset}(a) 
:=\sum \limits_{\alpha \vdash a}p^{\frac{1}{2}(\|\alpha\|^2+\| \alpha^t\|^2)} (\mathsf{V}_{ \square \emptyset \alpha}\mathsf{V}_{  \emptyset \emptyset \alpha^t})$ \\

\item $\Phi^{-,-}(a) 
:=\sum \limits_{\alpha \vdash a}p^{\frac{1}{2}(\|\alpha\|^2+\| \alpha^t\|^2)} (\mathsf{V}_{ \square \emptyset \alpha}\mathsf{V}_{ \square \emptyset \alpha^t})$ \\

\item $\Phi^{-,\,\mid\,}(a) 
:=\sum \limits_{\alpha \vdash a}p^{\frac{1}{2}(\|\alpha\|^2+\| \alpha^t\|^2)} (\mathsf{V}_{  \emptyset \square \alpha})^2$ \\

\item $\Phi^{+,\emptyset}(a)
:=  \sum \limits_{\alpha \vdash a}p^{\frac{1}{2}(\|\alpha\|^2+\| \alpha^t\|^2)+1} (\mathsf{V}_{   \square  \square \alpha}\mathsf{V}_{ \emptyset\emptyset\alpha^t})$ \\

\item $\Phi^{+,+}(a)
:=  \sum \limits_{\alpha \vdash a}p^{\frac{1}{2}(\|\alpha\|^2+\| \alpha^t\|^2)+2} (\mathsf{V}_{   \square  \square \alpha}\mathsf{V}_{  \square   \square \alpha^t}) $

\end{enumerate}

\end{re}

\subsection{Calculation for the class \texorpdfstring{$\bullet\sigma + (0,0,\bullet)$}{*s + (0,0,*)}} \label{bs+(0,0,d)_calc_section}
From lemma  \ref{classes_bs+(i,j,d)_main_chow_main_lemma} have the decomposition of $\Chow^{\bullet\sigma+(0,0,\bullet)}(X)$ into: 
\[
\mathbb{Z}_{\geq 0}  \times \Sym_{Q_3}^{\bullet}(B_{\mathsf{op}})
\]
Recall equation (\ref{bulletsigma+(i,j,bullet)_cycle_notation}) from section \ref{bulletsigma+(i,j,bullet)_cycle_preliminaries} and the notation:
\begin{align*}
\Cyc(C) =~& a \sigma + D +\sum_{i=1}^{12} m_i C_3^{(i)}. 
\end{align*}
In this class we have $D=\emptyset$. Hence we have the following summary of results from \ref{quot_decomposition_(i,j,bullet)_V_tilde} and \ref{fibre_general_(i,j,bullet)}.\\

\begin{tabu}{|@{}X[c]@{}|}
\hline
\extrarowsep=0.2em
\begin{tabu}{X[c]}
  $\chi(\O_D) =  0  $ \\
\end{tabu}\\
\hline
\begin{tabu}{@{}X[5.2cm] @{}|@{}X[cm]@{}}
$ e(\eta_\bullet^{-1}(a,\bm{m})) =
 \frac{1}{(\mathsf{V}_{\emptyset \emptyset \emptyset})^2}\cdot   Q_\sigma^a \Phi^{\emptyset,\emptyset}(a) \cdot  \prod \limits_{i=1}^{12}Q_3^{m_i}\Phi^{\emptyset,\emptyset}(m_i) $  &  \hfil \includegraphics[width=2cm]{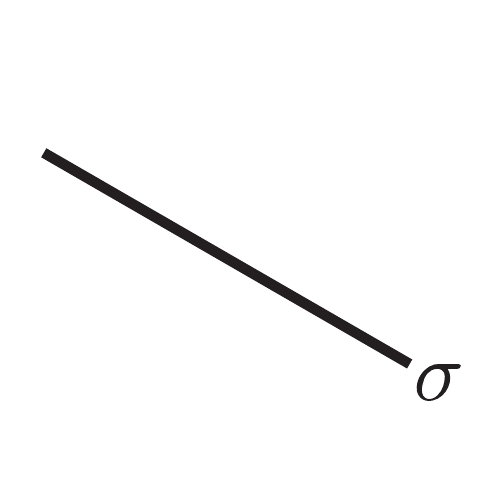} \hfil  \\
\end{tabu}\\
\hline
\end{tabu}

\vspace{0.25cm}
Now we have:
\begin{align*}
e\Big(~\mathbb{Z}_{\geq 0}  \times \Sym_{Q_3}^{\bullet}(B_{\mathsf{op}}), ~(\eta_\bullet)_*1~\Big) 
&=
  \frac{1}{(\mathsf{V}_{\emptyset \emptyset \emptyset})^2} \left(\sum_{a}  Q_\sigma^{a} \Phi^{\emptyset,\emptyset}(a) \right) \left(\sum_{m}Q_3^{m}\Phi^{\emptyset,\emptyset}(m)\right)^{12}\\
  &=
M(p)^{24} \prod\limits_{m>0}(1+p^m Q_{\sigma})^{m}(1+p^m Q_3)^{12m}.
\end{align*}\\

Where the last equality is from lemma \ref{banana_subpartition_functions} part \ref{banana_subpartition_functions_single_empty_empty} and lemma \ref{vertex_squares_lemma} part \ref{vertex_squares_lemma_eee}.\\

\subsection{Calculation for the class \texorpdfstring{$\bullet\sigma + (0,1,\bullet)$}{*s + (0,1,*)}} \label{bs+(0,1,d)_calc_section}
Recall the previously introduced notation: 
\begin{enumerate}[label={\arabic*)}]
\item  $B_{i} = \{b_{i}^1,\ldots,b_{i}^{12}\}$ is the set of the 12 points in $S_{i}$ that correspond to nodes in the fibres of the projection $\pi:S_{i} \rightarrow \P^1$.
\item $ S_{i}^\circ =  S_{i} \setminus  B_{i}$ is the complement of $ B_{i}$ in $S_{i}$
\item $\mathtt{N}_i\subset S_{i}$ are the 12 nodal fibres of $\pi:S_{i} \rightarrow \P^1$ with the nodes removed and: 
\begin{itemize}
\item[] $\mathtt{N}_i = \mathtt{N}^\sigma_i ~\amalg~ \mathtt{N}^\emptyset_i$ where   $\mathtt{N}^\sigma_i := \mathtt{N}_i\cap \sigma$ and $\mathtt{N}^\emptyset_i := \mathtt{N}_i\setminus \sigma$.
\end{itemize}
\item $ \mathtt{Sm}_i =  S_{i}^\circ \setminus  \mathtt{N}_i$ is the complement of $\mathtt{N}_i$ in $S_{i}^\circ$ and: 
\begin{itemize}
\item[] $\mathtt{Sm}_i = \mathtt{Sm}^\sigma_i ~\amalg~ \mathtt{Sm}^\emptyset_i$ where $\mathtt{Sm}^\sigma_i := \mathtt{Sm}_i\cap \sigma$ and $\mathtt{Sm}^\emptyset_i := \mathtt{Sm}_i\setminus \sigma$. \\
\end{itemize}
\end{enumerate}

Now from lemma  \ref{classes_bs+(i,j,d)_main_chow_main_lemma} we can further decompose $\Chow^{\bullet\sigma+(0,1,\bullet)}(X)$ into the four parts: 
\begin{enumerate}[label={\arabic*)}]
\item $\mathbb{Z}_{\geq 0} \times  \mathtt{Sm}^\sigma_{2}  \times \Sym_{Q_3}^{\bullet}(B_{\mathsf{op}})$
\item $\mathbb{Z}_{\geq 0} \times  \mathtt{Sm}^\emptyset_{2}  \times \Sym_{Q_3}^{\bullet}(B_{\mathsf{op}})$
\item $\mathbb{Z}_{\geq 0} \times  \mathtt{N}^\sigma_{2}  \times \Sym_{Q_3}^{\bullet}(B_{\mathsf{op}})$
\item $\mathbb{Z}_{\geq 0} \times  \mathtt{N}^\emptyset_{2}  \times \Sym_{Q_3}^{\bullet}(B_{\mathsf{op}})$
\item $\underset{k=1}{\overset{12}{\amalg}}~ \mathbb{Z}_{\geq 0} \times \Sym_{Q_3}^\bullet(\{ b_{\mathsf{op}}^{k}\}) \times \Sym_{Q_3}^\bullet(B_{\mathsf{op}}\setminus \{b_{\mathsf{op}}^{k}\})$\\
\end{enumerate}

Recall equation (\ref{bulletsigma+(i,j,bullet)_cycle_notation}) from section \ref{bulletsigma+(i,j,bullet)_cycle_preliminaries} and the notation:
\begin{align*}
\Cyc(C) =~& a \sigma + D +\sum_{i=1}^{12} m_i C_3^{(i)}. 
\end{align*}
Each part will be characterised by the type of $D$. We consider parts 1-4 separately to part 5.\\

\begin{re}\textbf{Parts 1-4:} In  parts 1-4 the curve $D$ is a fibre of the projection $\mathrm{pr}_2:X \rightarrow S$. The following table is the summary of results from \ref{quot_decomposition_(i,j,bullet)_V_tilde} and \ref{fibre_general_(i,j,bullet)} when applied to the particular $D$'s arising in each strata:
\[
\mathbb{Z}_{\geq 0} \times U  \times \Sym_{Q_3}^{\bullet}(B_{\mathsf{op}}).\\
\]

\begin{tabu}{|@{}X[c]@{}|}
\hline
\extrarowsep=0.2em
\begin{tabu}{@{}X[0.7c]|X[0.7c]|X[0.7c]@{}}
$U = \mathtt{N}_1^\sigma$ & $e(U) =  12 $ &  $\chi(\O_D) = 0$ \\
\end{tabu}\\
\hline
\begin{tabu}{@{}X[5.2cm] @{}|@{}X[cm]@{}}
$ e(\eta_\bullet^{-1}(a,x,\bm{m})) =
 Q_2 Q_3 p \frac{(\mathsf{V}_{\square \square \emptyset})}{(\mathsf{V}_{\square \emptyset \emptyset})(\mathsf{V}_{\emptyset \emptyset \emptyset})^2} \cdot  Q_\sigma^a \Phi^{-,\emptyset}(a) \cdot \prod \limits_{i=1}^{12}Q_3^{m_i}\Phi^{\emptyset,\emptyset}(m_i) $  &  \hfil \includegraphics[width=2cm]{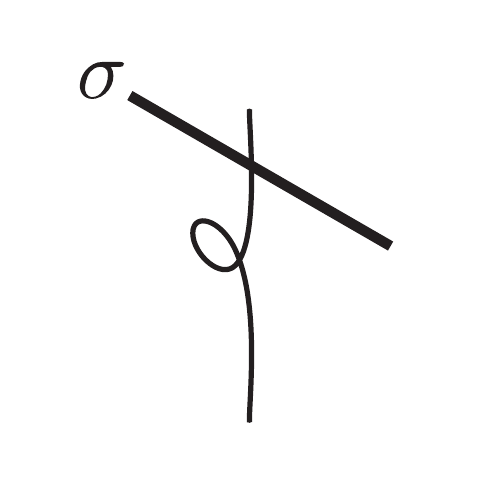} \hfil  \\
\end{tabu}\\
\hline
\extrarowsep=0.2em
\begin{tabu}{@{}X[0.7c]|X[0.7c]|X[0.7c]@{}}
$U = \mathtt{N}_1^\emptyset$ & $e(U) =  -12 $ &  $\chi(\O_D) = 0$ \\
\end{tabu}\\
\hline
\begin{tabu}{@{}X[5.2cm] @{}|@{}X[cm]@{}}
$ e(\eta_\bullet^{-1}(a,x,\bm{m})) =
 Q_2 Q_3p \frac{(\mathsf{V}_{\square \square \emptyset})}{(\mathsf{V}_{\emptyset \emptyset \emptyset})^3}   \cdot Q_\sigma^a\Phi^{\emptyset,\emptyset}(a) \cdot \prod \limits_{i=1}^{12}Q_3^{m_i}\Phi^{\emptyset,\emptyset}(m_i) $  &  \hfil \includegraphics[width=2cm]{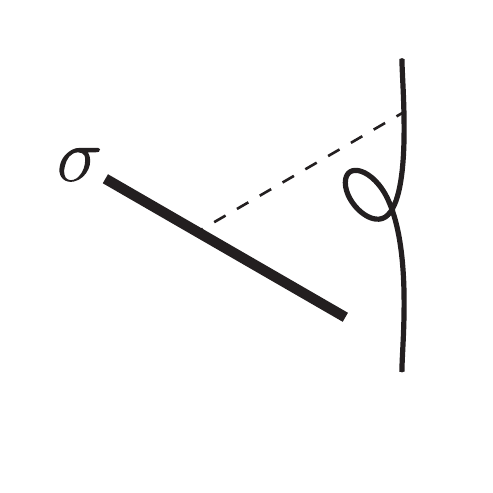} \hfil  \\
\end{tabu}\\
\hline
\extrarowsep=0.2em
\begin{tabu}{@{}X[0.7c]|X[0.7c]|X[0.7c]@{}}
$U = \mathtt{Sm}_1^\sigma$ & $e(U) =  -10 $ &  $\chi(\O_D) = 0$ \\
\end{tabu}\\
\hline
\begin{tabu}{@{}X[5.2cm] @{}|@{}X[cm]@{}}
$ e(\eta_\bullet^{-1}(a,x,\bm{m})) =
 Q_2 Q_3 \frac{1}{(\mathsf{V}_{\square \emptyset \emptyset})(\mathsf{V}_{\emptyset \emptyset \emptyset})} \cdot Q_\sigma^a  \Phi^{-,\emptyset}(a) \cdot  \prod \limits_{i=1}^{12}Q_3^{m_i}\Phi^{\emptyset,\emptyset}(m_i) $   &  \hfil \includegraphics[width=2cm]{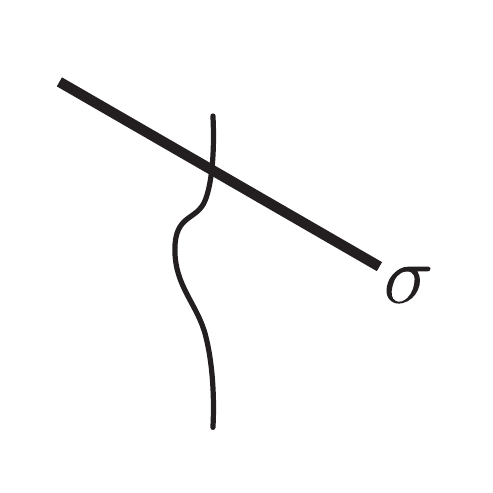} \hfil  \\
\end{tabu}\\
\hline
\extrarowsep=0.2em
\begin{tabu}{@{}X[0.7c]|X[0.7c]|X[0.7c]@{}}
$U = \mathtt{Sm}_1^\emptyset$ & $e(U) =  10 $ &  $\chi(\O_D) = 0$ \\
\end{tabu}\\
\hline
\begin{tabu}{@{}X[5.2cm] @{}|@{}X[cm]@{}}
$ e(\eta_\bullet^{-1}(a,x,\bm{m})) =
 Q_2 Q_3 \frac{1}{(\mathsf{V}_{\emptyset \emptyset \emptyset})^2}  \cdot Q_\sigma^a\Phi^{\emptyset,\emptyset}(a)  \cdot \prod \limits_{i=1}^{12} Q_3^{m_i}\Phi^{\emptyset,\emptyset}(m_i) $  &  \hfil \includegraphics[width=2cm]{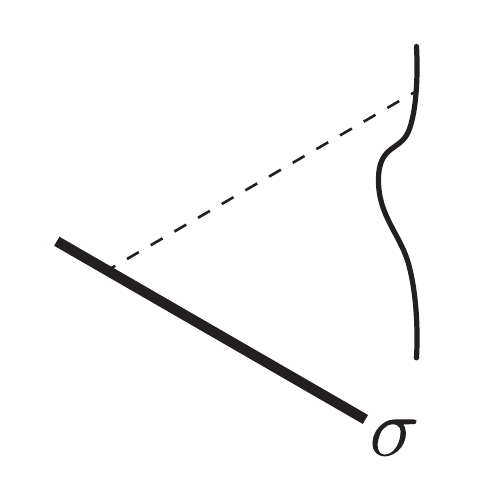} \hfil  \\
\end{tabu}\\
\hline
\end{tabu}

The union of parts 1-4 is $\mathbb{Z}_{\geq0}\times S_{2}^\circ \times \Sym_{Q_3}^{\bullet}(B_{\mathsf{op}})$ so we have:
\begin{align*}
e\Big(~\mathbb{Z}_{\geq0}&\times S_{2}^\circ \times \Sym_{Q_3}^{\bullet}(B_{\mathsf{op}}), ~(\eta_\bullet)_*1~\Big) \\
=~&
e\Big(~\mathbb{Z}_{\geq0}\times \mathtt{Sm}^\sigma_{2} \times \Sym_{Q_3}^{\bullet}(B_{\mathsf{op}}), ~(\eta_\bullet)_*1~\Big) \\
&~+~ e\Big(~\mathbb{Z}_{\geq0}\times \mathtt{Sm}^\emptyset_{2} \times \Sym_{Q_3}^{\bullet}(B_{\mathsf{op}}), ~(\eta_\bullet)_*1~\Big) \\
&~+~ e\Big(~\mathbb{Z}_{\geq0}\times \mathtt{N}^\sigma_{2} \times \Sym_{Q_3}^{\bullet}(B_{\mathsf{op}}), ~(\eta_\bullet)_*1~\Big) \\
&~+~ e\Big(~\mathbb{Z}_{\geq0}\times\mathtt{N}^\emptyset_{2}  \times \Sym_{Q_3}^{\bullet}(B_{\mathsf{op}}), ~(\eta_\bullet)_*1~\Big) 
\end{align*}
Which becomes: 
\begin{align*}
e\Big(~\mathbb{Z}_{\geq0}&\times S_{2}^\circ \times \Sym_{Q_3}^{\bullet}(B_{\mathsf{op}}), ~(\eta_\bullet)_*1~\Big) \\
=~&
e(\mathtt{Sm}^\sigma_{2})  Q_2 Q_3  p \frac{(\mathsf{V}_{\square \square \emptyset})}{(\mathsf{V}_{\square \emptyset \emptyset})(\mathsf{V}_{\emptyset \emptyset \emptyset})^2} \left(\sum_{a}  Q_\sigma^{a} \Phi^{-,\emptyset}(a) \right) \left(\sum_{m}Q_3^{m}\Phi^{\emptyset,\emptyset}(m)\right)^{12} \\
&~+~ e(\mathtt{Sm}^\emptyset_{2})   Q_2 Q_3p \frac{(\mathsf{V}_{\square \square \emptyset})}{(\mathsf{V}_{\emptyset \emptyset \emptyset})^3}   \left(\sum_{a\geq0} Q_\sigma^a\Phi^{\emptyset,\emptyset}(a) \right)\left(\sum_{m\geq 0} Q_3^{m}\Phi^{\emptyset,\emptyset}(m)\right)^{12}  \\
&~+~ e(\mathtt{N}^\sigma_{2} ) Q_2 Q_3 \frac{1}{(\mathsf{V}_{\square \emptyset \emptyset})(\mathsf{V}_{\emptyset \emptyset \emptyset})} \left(\sum_{a\geq0} Q_\sigma^a  \Phi^{-,\emptyset}(a)\right)\left(\sum_{m\geq 0}Q_3^{m}\Phi^{\emptyset,\emptyset}(m) \right)^{12} \\
&~+~ e(\mathtt{N}^\emptyset_{2}) Q_2 Q_3 \frac{1}{(\mathsf{V}_{\emptyset \emptyset \emptyset})^2} \left(\sum_{a\geq0} Q_\sigma^a\Phi^{\emptyset,\emptyset}(a) \right)\left(\sum_{m\geq 0} Q_3^{m}\Phi^{\emptyset,\emptyset}(m)\right)^{12} 
\end{align*}

From lemmas \ref{vertex_squares_lemma}, \ref{banana_subpartition_functions} and \ref{vertex_section_formulas} we have:
\begin{enumerate}[label={\arabic*)}]
\item $\mathsf{V}_{\emptyset\emptyset\emptyset} = M(p)$
\item $\mathsf{V}_{\square\emptyset\emptyset} = M(p) \frac{1}{1-p}$
\item $\mathsf{V}_{\square\square\emptyset} = M(p) \frac{p^2-p+1}{p(1-p)^2}$
\item $\sum_{m\geq 0} Q^{m}\Phi^{\emptyset,\emptyset}(m) = M(p)^2 \prod\limits_{m>0}(1+p^m Q)^{m}$
\item $\sum_{m\geq 0} Q^{m}\Phi^{-,\emptyset}(m) = M(p)^2\dfrac{1+Q}{1-p} \prod\limits_{m>0}(1+p^{m} Q)^m$
\end{enumerate}
So we have:
\begin{align*}
e\Big(~\mathbb{Z}_{\geq0}&\times S_{2}^\circ \times \Sym_{Q_3}^{\bullet}(B_{\mathsf{op}}), ~(\eta_\bullet)_*1~\Big) \\
=~& Q_\sigma Q_2 Q_3 M(p)^{24} \left( \prod\limits_{m>0}(1+p^m Q_{\sigma})^{m}(1+p^m Q_3)^{12m}\right)
\left(2 + 12\frac{p}{(1-p)^2}\right)\\
=~& Q_\sigma Q_2 Q_3 M(p)^{24} \left( \prod\limits_{m>0}(1+p^m Q_{\sigma})^{m}(1+p^m Q_3)^{12m}\right)
\left(2\psi_1 + 12\psi_0\right)
\end{align*}

\end{re}

\vspace{0.4cm}
\begin{re}\textbf{Part 5:} We have 12 separate isomorphic strata:
\[
\mathbb{Z}_{\geq 0} \times \Sym_{Q_3}^\bullet(\{ b_{\mathsf{op}}^{k}\}) \times \Sym_{Q_3}^\bullet(B_{\mathsf{op}}\setminus \{b_{\mathsf{op}}^{k}\}).
\]
These parameterise when $D = C_2^{(k)}$. The following is the summary of results from \ref{quot_decomposition_(i,j,bullet)_V_tilde} and \ref{fibre_general_(i,j,bullet)}. \\

\begin{tabu}{|@{}X[c]@{}|}
\hline
\extrarowsep=0.2em
\begin{tabu}{@{}X[0.7c]|X[0.7c]|X[0.7c]@{}}
$U = \{ k \}$ & $e(U) =  1$ &  $\chi(\O_D) = 1$  \\
\end{tabu}\\
\hline
\begin{tabu}{@{}X[5.2cm] @{}|@{}X[cm]@{}}
$ e(\eta_\bullet^{-1}(a,m_k,\bm{m})) =$
$ Q_2 p \frac{1}{(\mathsf{V}_{\emptyset \emptyset \emptyset})^2} \cdot  Q_\sigma^a  \Phi^{\emptyset,\emptyset}(a) \cdot  Q_3^{m_k} \Phi^{-,-}(m_k) 
 \prod\limits_{
 \genfrac{.}{.}{0pt}{2}{i=1}{i\neq k}
 }^{12}  
Q_3^{m_i}\Phi^{\emptyset,\emptyset}(m_i) $  
&  \hfil \includegraphics[width=2cm]{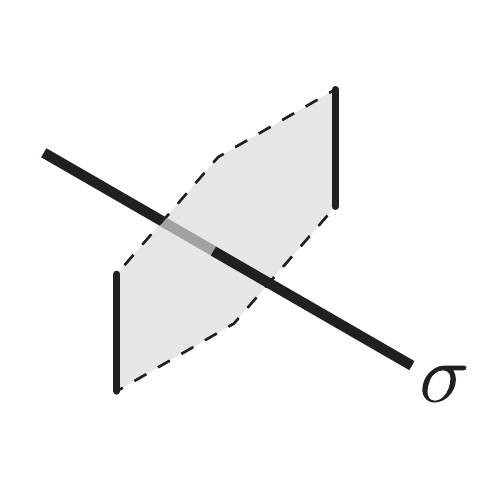} \hfil  \\
\end{tabu}\\
\hline
\end{tabu}
\mbox{}\\

From lemmas \ref{vertex_squares_lemma} and \ref{banana_subpartition_functions} we have:
\begin{enumerate}
\item $\mathsf{V}_{\emptyset\emptyset\emptyset} = M(p)$
\item $\sum_{m\geq 0} Q^{m}\Phi^{\emptyset,\emptyset}(m) = M(p)^2 \prod\limits_{m>0}(1+p^m Q)^{m}$
\item $\sum_{m\geq 0} Q^{m}\Phi^{-,-}(m) = M(p)^2 (\psi_0 +(\psi_1 +2\psi_0) Q +\psi_0 Q^2) \prod\limits_{m>0}(1+p^m Q)^{m}$
\end{enumerate}
Since the strata are isomorphic we have:
\begin{align*}
e\Big(~\underset{k=1}{\overset{12}{\amalg}}&~ \mathbb{Z}_{\geq 0} \times \Sym_{Q_3}^\bullet(\{ b_{\mathsf{op}}^{k}\}) \times \Sym_{Q_3}^\bullet(B_{\mathsf{op}}\setminus \{b_{\mathsf{op}}^{k}\}), ~(\eta_\bullet)_*1~\Big) \\
=~&
12~ e\Big(~ \mathbb{Z}_{\geq 0} \times \Sym_{Q_3}^\bullet( \{ b_{\mathsf{op}}^{k}\}) \times \Sym_{Q_3}^\bullet( B_{\mathsf{op}}\setminus \{b_{\mathsf{op}}^{k}\}), ~(\eta_\bullet)_*1~\Big)\\
=~&
 12 ~Q_2 \frac{1}{(\mathsf{V}_{\emptyset \emptyset \emptyset})^2} \left(\sum_{a\geq0} Q_\sigma^a\Phi^{\emptyset,\emptyset}(a) \right)\left(\sum_{m\geq 0} Q_3^{m}\Phi^{-,-}(m)\right) \left(\sum_{m\geq 0} Q_3^{m}\Phi^{\emptyset,\emptyset}(m)\right)^{11} \\
 =~&
 12 ~ Q_2 M(p)^{24} \left( \prod\limits_{m>0}(1+p^m Q_{\sigma})^{m}(1+p^m Q_3)^{12m}\right)\Big(\psi_0 +(\psi_1 +2\psi_0) Q_3 +\psi_0 Q_3^2\Big)
\end{align*}
\end{re}

\begin{re} 
Thus combining parts 1-5 we have that the overall formula is:
\begin{align*}
e\Big(&\Chow^{\bullet\sigma+(0,1,\bullet)}(X) , (\eta_\bullet)_*1\Big)\\
=&~ Q_2  M(p)^{24} \left( \prod\limits_{m>0}(1+p^m Q_{\sigma})^{m}(1+p^m Q_3)^{12m}\right)\\
&\cdot\Big(12\big(\psi_0 +(2\psi_0+\psi_1) Q_3 +\psi_0 Q_3^2\big) + Q_\sigma Q_3 \big(12 \psi_0 +2\psi_1\big) \Big)\\
\end{align*}

\end{re}

\subsection{Calculation for the class \texorpdfstring{$\bullet\sigma + (1,1,\bullet)$}{*s + (1,1,*)}}\label{bs+(1,1,d)_calc_section}
We have a decomposition from lemma \ref{classes_bs+(i,j,d)_main_chow_main_lemma} of $\Chow^{(1,1,\bullet)}(X)$ into the parts:
\begin{enumerate}[label={\alph*)}]
\item  $~S_{1}^\circ \times S_{2}^\circ \times \Sym_{Q_3}^{\bullet}(B_{\mathsf{op}}) $
\item  $~\underset{k=1}{\overset{12}{\amalg}}  S_{1}^\circ  \times\Sym_{Q_3}^\bullet( \{ b_{\mathsf{op}}^{k}\}) \times \Sym_{Q_3}^\bullet( B_{\mathsf{op}}\setminus \{b_{\mathsf{op}}^{k}\}) $
\item  $~\underset{k=1}{\overset{12}{\amalg}}  S_{2}^\circ  \times\Sym_{Q_3}^\bullet( \{ b_{\mathsf{op}}^{k}\}) \times \Sym_{Q_3}^\bullet( B_{\mathsf{op}}\setminus \{b_{\mathsf{op}}^{k}\}) $
\item  $\hspace{-0.58em}\underset{\mbox{\tiny$\begin{array}{c}k,l=1\\k\neq l\end{array}$}}{\overset{12}{\amalg}} \hspace{-0.58em} \Sym_{Q_3}^\bullet( \{ b_{\mathsf{op}}^{k}\})   \times\Sym_{Q_3}^\bullet( \{ b_{\mathsf{op}}^{l}\}) \times \Sym_{Q_3}^\bullet( B_{\mathsf{op}}\setminus \{b_{\mathsf{op}}^{k},b_{\mathsf{op}}^{l}\})$
\item  $~\underset{k=1}{\overset{12}{\amalg}}  \Sym_{Q_3}^\bullet( \{ b_{\mathsf{op}}^{k}\}) \times \Sym_{Q_3}^\bullet( B_{\mathsf{op}}\setminus \{b_{\mathsf{op}}^{k}\})$
\item  $~\hspace{0.4em}{\amalg} ~ \mathtt{Diag}^{\bullet}$
\end{enumerate}
We also recall the notation from equation (\ref{bulletsigma+(i,j,bullet)_cycle_notation}) from section \ref{bulletsigma+(i,j,bullet)_cycle_preliminaries} and the notation:
\begin{align*}
\Cyc(C) =~& a \sigma + D +\sum_{i=1}^{12} m_i C_3^{(i)}. 
\end{align*}
Each part will be characterised by the type of $D$. We will consider each case a-f separately and will use the following the previously introduced notation throughout: 
\begin{enumerate}[label={\arabic*)}]
\item  $B_{i} = \{b_{i}^1,\ldots,b_{i}^{12}\}$ is the set of the 12 points in $S_{i}$ that correspond to nodes in the fibres of the projection $\pi:S_{i} \rightarrow \P^1$.
\item $ S_{i}^\circ =  S_{i} \setminus  B_{i}$ is the complement of $ B_{i}$ in $S_{i}$
\item $\mathtt{N}_i\subset S_{i}$ are the 12 nodal fibres of $\pi:S_{i} \rightarrow \P^1$ with the nodes removed and: 
\begin{itemize}
\item[] $\mathtt{N}_i = \mathtt{N}^\sigma_i ~\amalg~ \mathtt{N}^\emptyset_i$ where   $\mathtt{N}^\sigma_i := \mathtt{N}_i\cap \sigma$ and $\mathtt{N}^\emptyset_i := \mathtt{N}_i\setminus \sigma$.
\end{itemize}
\item $ \mathtt{Sm}_i =  S_{i}^\circ \setminus  \mathtt{N}_i$ is the complement of $\mathtt{N}_i$ in $S_{i}^\circ$ and: 
\begin{itemize}
\item[] $\mathtt{Sm}_i = \mathtt{Sm}^\sigma_i ~\amalg~ \mathtt{Sm}^\emptyset_i$ where $\mathtt{Sm}^\sigma_i := \mathtt{Sm}_i\cap \sigma$ and $\mathtt{Sm}^\emptyset_i := \mathtt{Sm}_i\setminus \sigma$. 
\end{itemize}
\end{enumerate}
We will also use the new notation:
\[
\mathtt{D} := \big\{ (x,x) \in S_1^\circ  \times S  \big\}.
\]

\begin{re}\textbf{Part a:}
We have the following stratification of  $S^\circ_{1} \times S^\circ_{1}$:
\begin{enumerate}[label={\arabic*)}]
\item $ \Big( (\mathtt{N}_1^\sigma\times \mathtt{N}_2^\sigma)\cap \mathtt{D} ~\amalg~ (\mathtt{Sm}_1^\sigma\times \mathtt{Sm}_2^\sigma) \cap \mathtt{D}  \Big)$

\item $ ~\amalg~ \Big(\mathtt{N}_1^\sigma\times \mathtt{N}_2^\sigma\setminus \mathtt{D} ~\amalg~  \mathtt{N}_1^\sigma\times \mathtt{Sm}_2^\sigma  ~\amalg~  \mathtt{Sm}_1^\sigma \times \mathtt{N}_2^\sigma ~\amalg~  \mathtt{Sm}_1^\sigma \times \mathtt{Sm}_2^\sigma \setminus \mathtt{D}\Big)$

\item $ ~\amalg~ 
\left(\begin{array}{c}\mathtt{N}_1^\sigma\times \mathtt{N}_2^\emptyset \setminus \mathtt{D}~\amalg~ (\mathtt{N}_1^\sigma\times \mathtt{N}_2^\emptyset) \cap \mathtt{D}~\amalg~ \mathtt{N}_1^\sigma\times \mathtt{Sm}_2^\emptyset~\amalg~  \mathtt{Sm}_1^\sigma \times \mathtt{N}_2^\emptyset \\ \hspace{2.8em}
~\amalg~  \mathtt{Sm}_1^\sigma \times \mathtt{Sm}_2^\emptyset \setminus \mathtt{D}~\amalg~   (\mathtt{Sm}_1^\emptyset \times \mathtt{Sm}_2^\sigma)  \cap \mathtt{D}\end{array}\right)$

\item $~\amalg~ 
\left(\begin{array}{c}\mathtt{N}_1^\emptyset\times \mathtt{N}_2^\sigma \setminus \mathtt{D}~\amalg~ (\mathtt{N}_1^\emptyset\times \mathtt{N}_2^\sigma) \cap \mathtt{D}~\amalg~ \mathtt{N}_1^\emptyset\times \mathtt{Sm}_2^\sigma~\amalg~  \mathtt{Sm}_1^\emptyset \times \mathtt{N}_2^\sigma \\ \hspace{2.8em}~\amalg~  \mathtt{Sm}_1^\emptyset \times \mathtt{Sm}_2^\sigma \setminus \mathtt{D}~\amalg~   (\mathtt{Sm}_1^\sigma \times \mathtt{Sm}_2^\emptyset)  \cap \mathtt{D}\end{array}\right)$

\item $~\amalg~
\left(\begin{array}{c}\mathtt{N}_1^\emptyset \times \mathtt{N}_2^\emptyset \setminus \mathtt{D}~\amalg~ (\mathtt{N}_1^\emptyset \times \mathtt{N}_2^\emptyset) \cap \mathtt{D}~\amalg~\mathtt{Sm}_1^\emptyset \times \mathtt{N}_2^\emptyset ~\amalg~\mathtt{N}_1^\emptyset \times \mathtt{Sm}_2^\emptyset\\ \hspace{2.8em}
~\amalg~\mathtt{Sm}_1^\emptyset \times \mathtt{Sm}_2^\emptyset \setminus \mathtt{D}~\amalg~(\mathtt{Sm}_1^\emptyset \times \mathtt{Sm}_2^\emptyset) \cap \mathtt{D}\end{array}\right)$
\end{enumerate}

Here we have grouped by the number and type of intersection with $\sigma$.

\textit{Grouping 1:} The following table is the summary of results from \ref{quot_decomposition_(i,j,bullet)_V_tilde} and \ref{fibre_general_(i,j,bullet)} for the strata in grouping 1:
\[
\mathbb{Z}_{\geq 0} \times U  \times \Sym_{Q_3}^{\bullet}(B_{\mathsf{op}}).\\
\]

\begin{tabu}{|@{}X[c]@{}|}
\hline
\extrarowsep=0.2em
\begin{tabu}{@{}X[0.89c]|X[.55c]|X[1c]@{}}
$U = (\mathtt{N}_1^\sigma\times \mathtt{N}_2^\sigma)\cap \mathtt{D}$ & $e(U) =  12 $ &  $\chi(\O_D) =  -1  $ \\
\end{tabu}\\
\hline
\begin{tabu}{@{}X[5.2cm] @{}|@{}X[cm]@{}}
$ e(\eta_\bullet^{-1}(a,x,\bm{m})) =
 Q_1 Q_2 Q_3^{2}p \frac{(\mathsf{V}_{\square \square \emptyset})^2}{(\mathsf{V}_{\square \emptyset \emptyset})^2(\mathsf{V}_{ \emptyset  \emptyset \emptyset})^2} \cdot Q_\sigma^a \Phi^{+,\emptyset}(a) \cdot \prod \limits_{i=1}^{12}Q_3^{m_i}\Phi^{\emptyset,\emptyset}(m_i) $  &  \hfil \includegraphics[width=2cm]{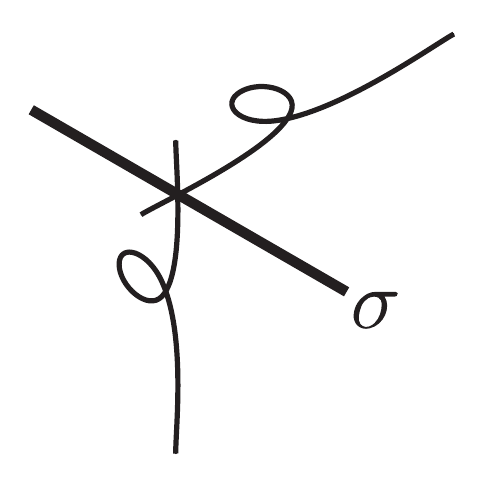} \hfil  \\
\end{tabu}\\
\hline
\extrarowsep=0.2em
\begin{tabu}{@{}X[0.89c]|X[.55c]|X[1.0c]@{}}
$U  \hspace{-0.12cm}= \hspace{-0.08cm} (\mathtt{Sm}_1^\sigma\times \mathtt{Sm}_2^\sigma) \cap \mathtt{D} $ & $e(U) \hspace{-0.05cm}=\hspace{-0.05cm}  -10 $ & $\chi(\O_D) = -1  $ \\
\end{tabu}\\
\hline
\begin{tabu}{@{}X[5.2cm] @{}|@{}X[cm]@{}}
$ e(\eta_\bullet^{-1}(a,x,\bm{m})) =
 Q_1 Q_2 Q_3^{2} p^{-1}\frac{1}{(\mathsf{V}_{\square \emptyset \emptyset})^2} \cdot Q_\sigma^a\Phi^{+,\emptyset}(a) \cdot \prod \limits_{i=1}^{12}Q_3^{m_i}\Phi^{\emptyset,\emptyset}(m_i) $  &  \hfil \includegraphics[width=2cm]{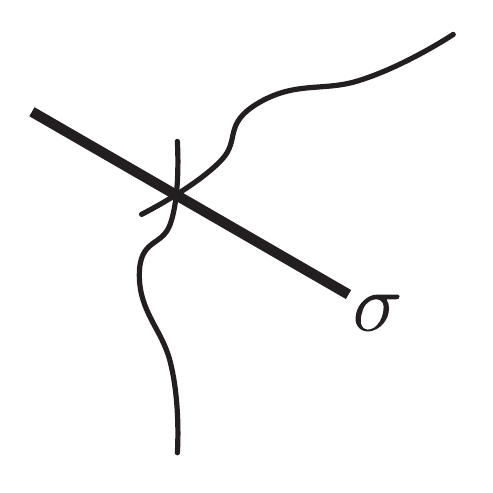} \hfil  \\
\end{tabu}\\
\hline
\end{tabu}
\mbox{}\\

From lemmas \ref{vertex_squares_lemma} and \ref{banana_subpartition_functions} we have:
\begin{enumerate}[label={\arabic*)}]
\item $\mathsf{V}_{\emptyset\emptyset\emptyset} = M(p)$
\item $\mathsf{V}_{\square\emptyset\emptyset} = M(p) \frac{1}{1-p}$
\item $\mathsf{V}_{\square\square\emptyset} = M(p) \frac{p^2-p+1}{p(1-p)^2}$
\item $\sum_{m\geq 0} Q^{m}\Phi^{\emptyset,\emptyset}(m) = M(p)^2 \prod\limits_{m>0}(1+p^m Q)^{m}$
\item $\sum_{m\geq 0} Q^{m}\Phi^{+,\emptyset}(m) = M(p)^2 (1+\psi_0 +(\psi_1 +2\psi_0) Q +\psi_0 Q^2) \prod\limits_{m>0}(1+p^m Q)^{m}$.
\end{enumerate}

So the contribution is:
\begin{align*}
&Q_1 Q_2 Q_3^2 M(p)^{24} \left( \prod\limits_{m>0}(1+ p^m Q_{\sigma})^{m}(1+p^m Q_3)^{12m}\right) (\psi_0 +(\psi_1 +2\psi_0) Q_\sigma +\psi_0 Q_\sigma^2) \\
&\cdot
\left(\frac{2 \left(p^4+8 p^3-12 p^2+8 p+1\right)}{(p-1)^2 p}\right)
\end{align*}

\textit{Grouping 2:} The following table is the summary of results from \ref{quot_decomposition_(i,j,bullet)_V_tilde} and \ref{fibre_general_(i,j,bullet)} for the strata in grouping 2:
\[
\mathbb{Z}_{\geq 0} \times U  \times \Sym_{Q_3}^{\bullet}(B_{\mathsf{op}}).
\]

\begin{tabu}{|@{}X[c]@{}|}
\hline
\extrarowsep=0.2em
\begin{tabu}{@{}X[0.7c]|X[.6c]|X[1c]@{}}
$U = \mathtt{N}_1^\sigma\times \mathtt{N}_2^\sigma\setminus \mathtt{D}$ & $e(U) = 132 $ & $\chi(\O_D) =  0  $ \\
\end{tabu}\\
\hline
\begin{tabu}{@{}X[5.2cm] @{}|@{}X[cm]@{}}
$ e(\eta_\bullet^{-1}(a,x,\bm{m})) =
 Q_1 Q_2 Q_3^{2} p^2 \frac{(\mathsf{V}_{\square \square \emptyset})^2}{(\mathsf{V}_{\square \emptyset \emptyset})^2(\mathsf{V}_{ \emptyset  \emptyset \emptyset})^2} \cdot Q_\sigma^a \Phi^{-,-}(a) \cdot \prod \limits_{i=1}^{12} Q_3^{m_i}\Phi^{\emptyset,\emptyset}(m_i) $ &  \hfil \includegraphics[width=2cm]{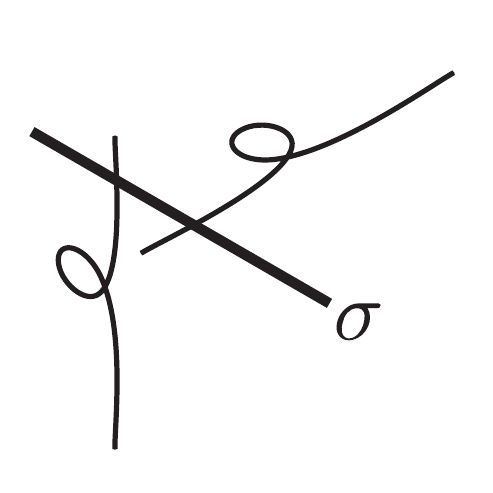} \hfil  \\
\end{tabu}\\
\hline
\extrarowsep=0.2em
\begin{tabu}{@{}X[0.7c]|X[.6c]|X[1c]@{}}
$U = \mathtt{N}_1^\sigma\times \mathtt{Sm}_2^\sigma $ & $e(U) = -120$ &$\chi(\O_D) =  0  $ \\
\end{tabu}\\
\hline
\begin{tabu}{@{}X[5.2cm] @{}|@{}X[cm]@{}}
$ e(\eta_\bullet^{-1}(a,x,\bm{m})) =
 Q_1 Q_2 Q_3^{2} p \frac{(\mathsf{V}_{\square \square \emptyset})}{(\mathsf{V}_{\square \emptyset \emptyset})^2(\mathsf{V}_{ \emptyset  \emptyset \emptyset})} \cdot Q_\sigma^a\Phi^{-,-}(a) \cdot \prod \limits_{i=1}^{12}Q_3^{m_i}\Phi^{\emptyset,\emptyset}(m_i) $  &  \hfil \includegraphics[width=2cm]{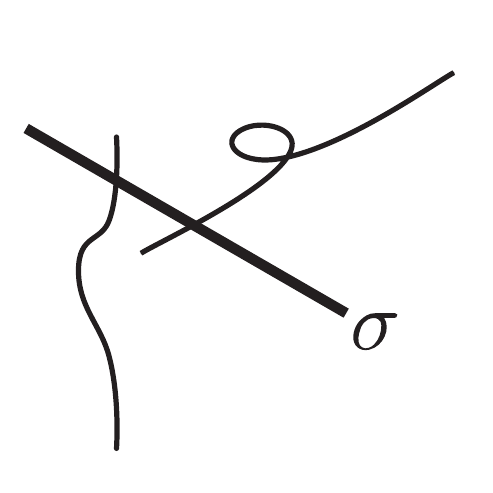} \hfil  \\
\end{tabu}\\
\hline
\extrarowsep=0.2em
\begin{tabu}{@{}X[0.7c]|X[.6c]|X[1c]@{}}
$U = \mathtt{Sm}_1^\sigma \times \mathtt{N}_2^\sigma$ & $e(U) =  -120 $ & $\chi(\O_D) = 0 $ \\
\end{tabu}\\
\hline
\begin{tabu}{@{}X[5.2cm] @{}|@{}X[cm]@{}}
$ e(\eta_\bullet^{-1}(a,x,\bm{m})) =
 Q_1 Q_2 Q_3^{2}  p\frac{(\mathsf{V}_{\square \square \emptyset})}{(\mathsf{V}_{\square \emptyset \emptyset})^2(\mathsf{V}_{ \emptyset  \emptyset \emptyset})} \cdot Q_\sigma^a\Phi^{-,-}(a) \cdot \prod \limits_{i=1}^{12}Q_3^{m_i}\Phi^{\emptyset,\emptyset}(m_i) $  &  \hfil \includegraphics[width=2cm]{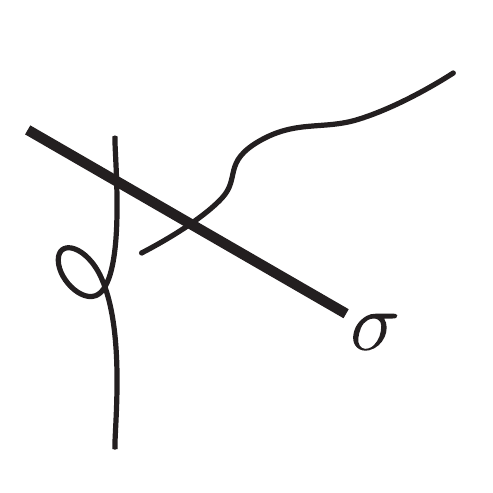} \hfil  \\
\end{tabu}\\
\hline
\extrarowsep=0.2em
\begin{tabu}{@{}X[0.7c]|X[.6c]|X[1c]@{}}
$U \hspace{-0.12cm}= \hspace{-0.06cm} \mathtt{Sm}_1^\sigma \times \mathtt{Sm}_2^\sigma \setminus \mathtt{D}$ & $e(U) = 110  $ & $\chi(\O_D) = 0  $  \\
\end{tabu}\\
\hline
\begin{tabu}{@{}X[5.2cm] @{}|@{}X[cm]@{}}
$ e(\eta_\bullet^{-1}(a,x,\bm{m})) =
 Q_1 Q_2 Q_3^{2}  \frac{1}{(\mathsf{V}_{\square \emptyset \emptyset})^2} \cdot Q_\sigma^a \Phi^{-,-}(a) \cdot \prod \limits_{i=1}^{12}Q_3^{m_i}\Phi^{\emptyset,\emptyset}(m_i) $  &  \hfil \includegraphics[width=2cm]{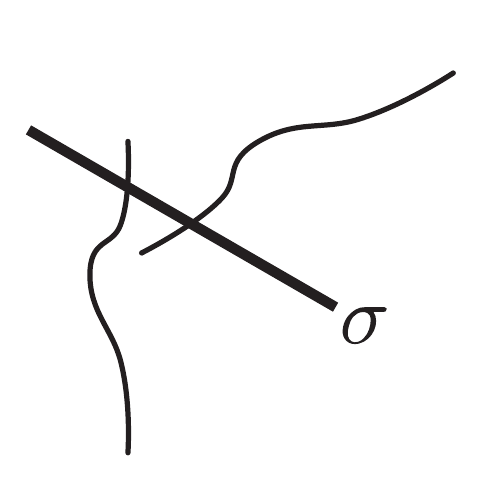} \hfil  \\
\end{tabu}\\
\hline
\end{tabu}
\mbox{}\\

From lemmas \ref{vertex_squares_lemma} and \ref{banana_subpartition_functions} we have:
\begin{enumerate}[label={\arabic*)}]
\item $\mathsf{V}_{\emptyset\emptyset\emptyset} = M(p)$
\item $\mathsf{V}_{\square\emptyset\emptyset} = M(p) \frac{1}{1-p}$
\item $\mathsf{V}_{\square\square\emptyset} = M(p) \frac{p^2-p+1}{p(1-p)^2}$
\item $\sum_{m\geq 0} Q^{m}\Phi^{\emptyset,\emptyset}(m) = M(p)^2 \prod\limits_{m>0}(1+p^m Q)^{m}$
\item $\sum_{m\geq 0} Q^{m}\Phi^{-,-}(m) = M(p)^2\dfrac{1}{p} (\psi_0 +(\psi_1 +2\psi_0) Q +\psi_0 Q^2) \prod\limits_{m>0}(1+p^m Q)^{m}$.
\end{enumerate}

So the contribution is:
\begin{align*}
&Q_1 Q_2 Q_3^2 M(p)^{24} \left( \prod\limits_{m>0}(1+p^m Q_{\sigma})^{m}(1+p^m Q_3)^{12m}\right)\frac{1}{p}(\psi_0 +(\psi_1 +2\psi_0) Q +\psi_0 Q^2) \\
&\cdot
\left(
\frac{2 \left(p^4+8 p^3+48 p^2+8 p+1\right)}{(p-1)^2}
\right)
\end{align*}

\textit{Grouping 3:} The following table is the summary of results from \ref{quot_decomposition_(i,j,bullet)_V_tilde} and \ref{fibre_general_(i,j,bullet)} for the strata in grouping 3:
\[
\mathbb{Z}_{\geq 0} \times U  \times \Sym_{Q_3}^{\bullet}(B_{\mathsf{op}}).\\
\]

\begin{tabu}{|@{}X[c]@{}|}
\hline
\extrarowsep=0.2em
\begin{tabu}{@{}X[0.7c]|X[.6c]|X[0.8c]@{}}
$U = \mathtt{N}_1^\sigma\times \mathtt{N}_2^\emptyset \setminus \mathtt{D}$ & $e(U) = -132$ &$\chi(\O_D) =  0  $ \\
\end{tabu}\\
\hline
\begin{tabu}{@{}X[5.2cm] @{}|@{}X[cm]@{}}
$ e(\eta_\bullet^{-1}(a,x,\bm{m})) =
 Q_1 Q_2 Q_3^{2} p^2 \frac{(\mathsf{V}_{\square \square \emptyset})^2}{(\mathsf{V}_{\square \emptyset \emptyset})(\mathsf{V}_{\emptyset \emptyset \emptyset})^3} \cdot Q_\sigma^a \Phi^{-,\emptyset}(a) \cdot \prod \limits_{i=1}^{12}Q_3^{m_i}\Phi^{\emptyset,\emptyset}(m_i) $  &  \hfil \includegraphics[width=2cm]{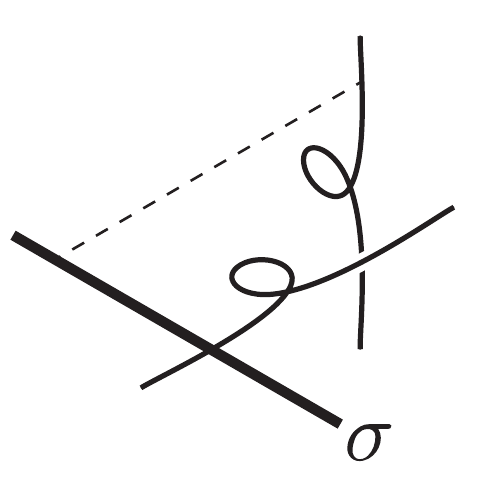} \hfil  \\
\end{tabu}\\
\hline
\extrarowsep=0.2em
\begin{tabu}{@{}X[0.7c]|X[.6c]|X[0.8c]@{}}
$U = (\mathtt{N}_1^\sigma\times \mathtt{N}_2^\emptyset) \cap  \mathtt{D}$ & $e(U) \hspace{-0.05cm}=\hspace{-0.05cm}  -12 $ & $\chi(\O_D) =  -1 $ \\
\end{tabu}\\
\hline
\begin{tabu}{@{}X[5.2cm] @{}|@{}X[cm]@{}}
$ e(\eta_\bullet^{-1}(a,x,\bm{m})) =
 Q_1 Q_2 Q_3^{2} p^{2} \frac{(\mathsf{V}_{\square \square \emptyset})^3}{(\mathsf{V}_{\square \emptyset \emptyset})^3(\mathsf{V}_{\emptyset \emptyset \emptyset})^2} \cdot Q_\sigma^a \Phi^{-,\emptyset}(a) \cdot \prod \limits_{i=1}^{12}Q_3^{m_i}\Phi^{\emptyset,\emptyset}(m_i) $  &  \hfil \includegraphics[width=2cm]{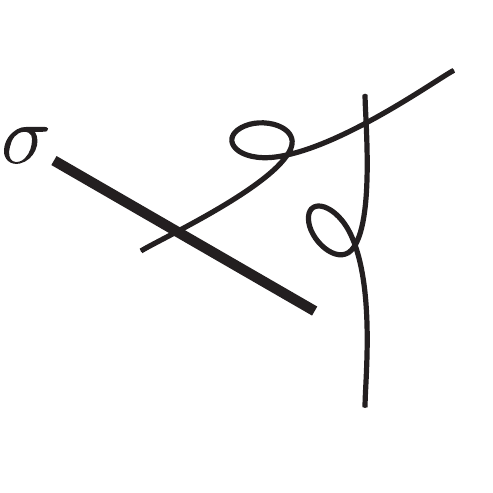} \hfil  \\
\end{tabu}\\
\hline
\extrarowsep=0.2em
\begin{tabu}{@{}X[0.7c]|X[.6c]|X[0.8c]@{}}
$U = \mathtt{N}_1^\sigma\times \mathtt{Sm}_2^\emptyset$ & $e(U) = 120$ & $\chi(\O_D) =  0  $ \\
\end{tabu}\\
\hline
\begin{tabu}{@{}X[5.2cm] @{}|@{}X[cm]@{}}
$ e(\eta_\bullet^{-1}(a,x,\bm{m})) =
 Q_1 Q_2 Q_3^{2} p \frac{(\mathsf{V}_{\square \square \emptyset})}{(\mathsf{V}_{\square \emptyset \emptyset})(\mathsf{V}_{\emptyset \emptyset \emptyset})^2} \cdot Q_\sigma^a \Phi^{-,\emptyset}(a) \cdot \prod \limits_{i=1}^{12}Q_3^{m_i}\Phi^{\emptyset,\emptyset}(m_i) $  &  \hfil \includegraphics[width=2cm]{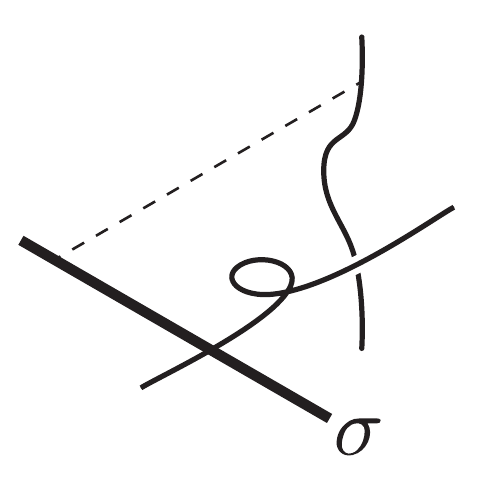} \hfil  \\
\end{tabu}\\
\hline
\extrarowsep=0.2em
\begin{tabu}{@{}X[0.7c]|X[.6c]|X[0.8c]@{}}
$U = \mathtt{Sm}_1^\sigma \times \mathtt{N}_2^\emptyset $ & $e(U) =  120 $ & $\chi(\O_D) =  0 $ \\
\end{tabu}\\
\hline
\begin{tabu}{@{}X[5.2cm] @{}|@{}X[cm]@{}}
$ e(\eta_\bullet^{-1}(a,x,\bm{m})) =
Q_1 Q_2 Q_3^{2} p \frac{(\mathsf{V}_{\square \square \emptyset})}{(\mathsf{V}_{\square \emptyset \emptyset})(\mathsf{V}_{\emptyset \emptyset \emptyset})^2} \cdot Q_\sigma^a  \Phi^{-,\emptyset}(a) \cdot \prod \limits_{i=1}^{12}Q_3^{m_i}\Phi^{\emptyset,\emptyset}(m_i) $  &  \hfil \includegraphics[width=2cm]{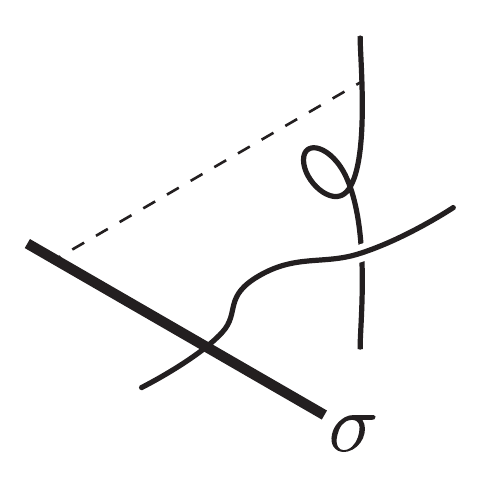} \hfil  \\
\end{tabu}\\
\hline
\extrarowsep=0.2em
\begin{tabu}{@{}X[0.7c]|X[.6c]|X[0.8c]@{}}
$U  \hspace{-0.12cm}= \hspace{-0.06cm} \mathtt{Sm}_1^\sigma \times \mathtt{Sm}_2^\emptyset \setminus  \mathtt{D}$ & $e(U) =  -110 $ & $\chi(\O_D) =  0$ \\
\end{tabu}\\
\hline
\begin{tabu}{@{}X[5.2cm] @{}|@{}X[cm]@{}}
$ e(\eta_\bullet^{-1}(a,x,\bm{m})) =
Q_1 Q_2 Q_3^{2}  \frac{1}{(\mathsf{V}_{\square \emptyset \emptyset})(\mathsf{V}_{\emptyset \emptyset \emptyset})} \cdot Q_\sigma^a  \Phi^{-,\emptyset}(a) \cdot \prod \limits_{i=1}^{12}Q_3^{m_i}\Phi^{\emptyset,\emptyset}(m_i) $  &  \hfil \includegraphics[width=2cm]{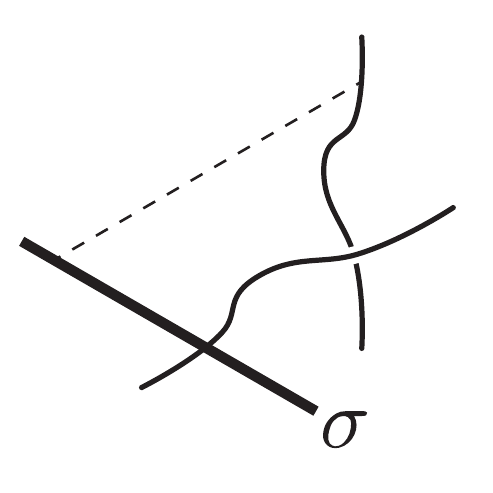} \hfil  \\
\end{tabu}\\
\hline
\extrarowsep=0.2em
\begin{tabu}{@{}X[0.7c]|X[.6c]|X[0.8c]@{}}
$U  \hspace{-0.12cm}= \hspace{-0.08cm}  (\mathtt{Sm}_1^\emptyset \times \mathtt{Sm}_2^\sigma)  \cap  \mathtt{D}$ & $e(U) = 10 $ & $\chi(\O_D) =  -1$ \\
\end{tabu}\\
\hline
\begin{tabu}{@{}X[5.2cm] @{}|@{}X[cm]@{}}
$ e(\eta_\bullet^{-1}(a,x,\bm{m})) =
 Q_1 Q_2 Q_3^{2}  \frac{(\mathsf{V}_{\square \square \emptyset})}{(\mathsf{V}_{\square \emptyset \emptyset})^3} \cdot Q_\sigma^a \Phi^{-,\emptyset}(a) \cdot \prod \limits_{i=1}^{12}Q_3^{m_i}\Phi^{\emptyset,\emptyset}(m_i) $  &  \hfil \includegraphics[width=2cm]{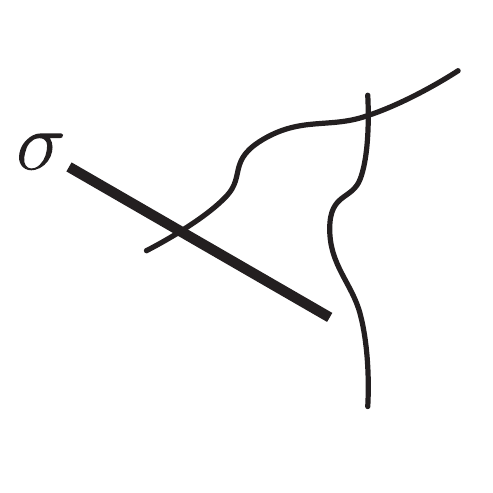} \hfil  \\
\end{tabu}\\
\hline
\end{tabu}
\mbox{}\\

From lemmas \ref{vertex_squares_lemma}, \ref{banana_subpartition_functions} and \ref{vertex_section_formulas} we have:
\begin{enumerate}[label={\arabic*)}]
\item $\mathsf{V}_{\emptyset\emptyset\emptyset} = M(p)$
\item $\mathsf{V}_{\square\emptyset\emptyset} = M(p) \frac{1}{1-p}$
\item $\mathsf{V}_{\square\square\emptyset} = M(p) \frac{p^2-p+1}{p(1-p)^2}$
\item $\sum_{m\geq 0} Q^{m}\Phi^{\emptyset,\emptyset}(m) = M(p)^2 \prod\limits_{m>0}(1+p^m Q)^{m}$
\item $\sum_{m\geq 0} Q^{m}\Phi^{-,\emptyset}(m) = M(p)^2\dfrac{1+Q}{1-p} \prod\limits_{m>0}(1+p^{m} Q)^m$
\end{enumerate}

The contribution from grouping 3 is: 
\begin{align*}
&Q_1 Q_2Q_3^2 M(p)^{24} \left( \prod\limits_{m>0}(1+p^m Q_{\sigma})^{m}(1+p^m Q_3)^{12m}\right)\left(\frac{1+Q_{\sigma}}{1-p}\right) \\
&\cdot
\left(
\frac{2 \left(p^2+10 p+1\right) \left(p^4-2 p^3+8 p^2-2 p+1\right)}{(p-1)^3 p}
\right) 
\end{align*}

\textit{Grouping 4:} The results for grouping 4 are identical to those of grouping 3 under the symmetry of the banana threefold. \\

The contribution from grouping 4 is: 
\begin{align*}
&Q_1 Q_2Q_3^2 M(p)^{24} \left( \prod\limits_{m>0}(1+p^m Q_{\sigma})^{m}(1+p^m Q_3)^{12m}\right)\left(\frac{1+Q_{\sigma}}{1-p}\right) \\
&\cdot
\left(
\frac{2 \left(p^2+10 p+1\right) \left(p^4-2 p^3+8 p^2-2 p+1\right)}{(p-1)^3 p}
\right) 
\end{align*}

\textit{Grouping 5:} The following table is the summary of results from \ref{quot_decomposition_(i,j,bullet)_V_tilde} and \ref{fibre_general_(i,j,bullet)} for the strata in grouping 5:
\[
\mathbb{Z}_{\geq 0} \times U  \times \Sym_{Q_3}^{\bullet}(B_{\mathsf{op}}).\\
\]

\begin{tabu}{|@{}X[c]@{}|}
\hline
\extrarowsep=0.2em
\begin{tabu}{@{}X[0.8c]|X[0.6c]|X[0.7c]@{}}
$U = \mathtt{N}_1^\emptyset \times \mathtt{N}_2^\emptyset \setminus  \mathtt{D}$ & $e(U) =  132 $ & $\chi(\O_D) =  0$ \\
\end{tabu}\\
\hline
\begin{tabu}{@{}X[5.2cm] @{}|@{}X[cm]@{}}
$ e(\eta_\bullet^{-1}(a,x,\bm{m})) =
 Q_1 Q_2 Q_3^{2}p^2 \frac{\mathsf{V}_{\square \square \emptyset})^2}{(\mathsf{V}_{\emptyset\emptyset\emptyset})^4} \cdot Q_\sigma^a\Psi^{\emptyset,\emptyset}(a) \cdot \prod \limits_{i=1}^{12}Q_3^{m_i}\Phi^{\emptyset,\emptyset}(m_i) $  &  \hfil \includegraphics[width=2cm]{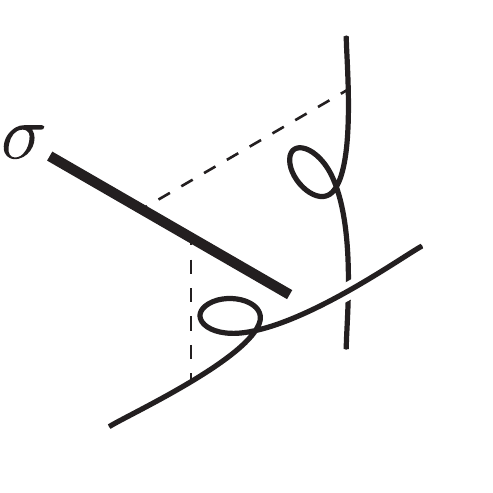}  \hfil  \\
\end{tabu}\\
\hline
\extrarowsep=0.2em
\begin{tabu}{@{}X[0.8c]|X[0.6c]|X[0.7c]@{}}
$U = (\mathtt{N}_1^\emptyset \times \mathtt{N}_2^\emptyset) \cap  \mathtt{D}$ & $e(U) =12 $ & $\chi(\O_D) =  -1$  \\
\end{tabu}\\
\hline
\begin{tabu}{@{}X[5.2cm] @{}|@{}X[cm]@{}}
$ e(\eta_\bullet^{-1}(a,x,\bm{m})) =
 Q_1 Q_2 Q_3^{2} p^{2} \frac{(\mathsf{V}_{\square \square \emptyset})^3}{(\mathsf{V}_{\square \emptyset \emptyset})^2(\mathsf{V}_{\emptyset\emptyset\emptyset})^3} \cdot Q_\sigma^a \Psi^{\emptyset,\emptyset}(a) \cdot \prod \limits_{i=1}^{12}Q_3^{m_i}\Phi^{\emptyset,\emptyset}(m_i) $  &  \hfil \includegraphics[width=2cm]{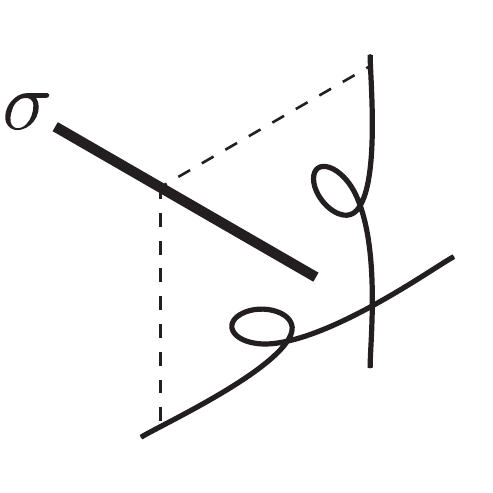} \hfil  \\
\end{tabu}\\
\hline
\end{tabu}

\begin{tabu}{|@{}X[c]@{}|}
\hline
\extrarowsep=0.2em
\begin{tabu}{@{}X[0.8c]|X[0.6c]|X[0.7c]@{}}
$U = \mathtt{Sm}_1^\emptyset \times \mathtt{N}_2^\emptyset $ & $e(U) =  -120 $ & $\chi(\O_D) =  0 $ \\
\end{tabu}\\
\hline
\begin{tabu}{@{}X[5.2cm] @{}|@{}X[cm]@{}}
$ e(\eta_\bullet^{-1}(a,x,\bm{m})) =
 Q_1 Q_2 Q_3^{2} p \frac{(\mathsf{V}_{\square \square \emptyset})}{(\mathsf{V}_{\emptyset\emptyset\emptyset})^3} \cdot Q_\sigma^a \Psi^{\emptyset,\emptyset}(a) \cdot \prod \limits_{i=1}^{12}Q_3^{m_i}\Phi^{\emptyset,\emptyset}(m_i) $   &  \hfil \includegraphics[width=2cm]{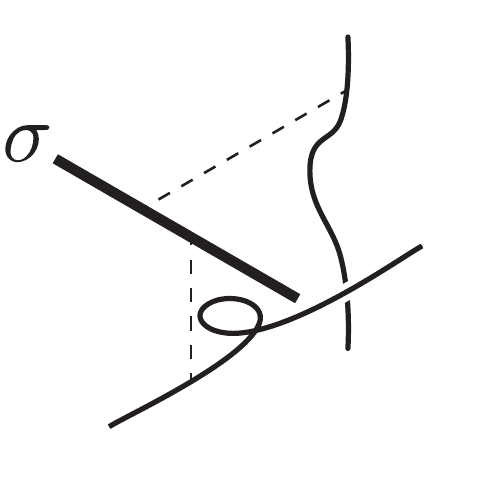} \hfil  \\
\end{tabu}\\
\hline
\extrarowsep=0.2em
\begin{tabu}{@{}X[0.8c]|X[0.6c]|X[0.7c]@{}}
$U = \mathtt{N}_1^\emptyset \times \mathtt{Sm}_2^\emptyset$ & $e(U) =  -120 $ & $\chi(\O_D) =  0 $  \\
\end{tabu}\\
\hline
\begin{tabu}{@{}X[5.2cm] @{}|@{}X[cm]@{}}
$ e(\eta_\bullet^{-1}(a,x,\bm{m})) =
 Q_1 Q_2 Q_3^{2} p \frac{(\mathsf{V}_{\square \square \emptyset})}{(\mathsf{V}_{\emptyset\emptyset\emptyset})^3} \cdot Q_\sigma^a \Psi^{\emptyset,\emptyset}(a) \cdot \prod \limits_{i=1}^{12}Q_3^{m_i}\Phi^{\emptyset,\emptyset}(m_i) $   &  \hfil \includegraphics[width=2cm]{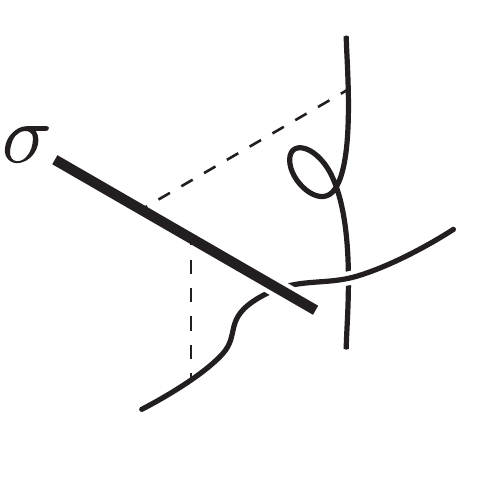}  \hfil  \\
\end{tabu}\\
\hline
\extrarowsep=0.2em
\begin{tabu}{@{}X[0.8c]|X[0.6c]|X[0.7c]@{}}
$U  \hspace{-0.12cm}= \hspace{-0.06cm} \mathtt{Sm}_1^\emptyset \times \mathtt{Sm}_2^\emptyset \setminus \mathtt{D}$ & $e(U) =110 $ & $\chi(\O_D) =  0 $ \\
\end{tabu}\\
\hline
\begin{tabu}{@{}X[5.2cm] @{}|@{}X[cm]@{}}
$ e(\eta_\bullet^{-1}(a,x,\bm{m})) =
 Q_1 Q_2 Q_3^{2}  \frac{1}{(\mathsf{V}_{\emptyset\emptyset\emptyset})^2} \cdot  Q_\sigma^a\Psi^{\emptyset,\emptyset}(a) \cdot \prod \limits_{i=1}^{12}Q_3^{m_i}\Phi^{\emptyset,\emptyset}(m_i) $  &  \hfil \includegraphics[width=2cm]{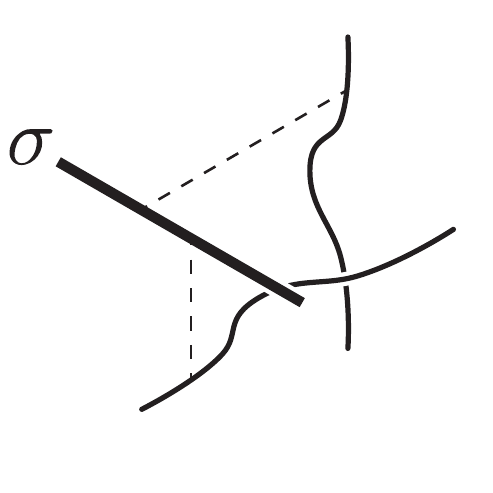} \hfil  \\
\end{tabu}\\
\hline
\extrarowsep=0.2em
\begin{tabu}{@{}X[0.8c]|X[0.6c]|X[0.7c]@{}}
$U  \hspace{-0.12cm}= \hspace{-0.08cm}  (\mathtt{Sm}_1^\emptyset \times \mathtt{Sm}_2^\emptyset) \cap  \mathtt{D}$ & $e(U) \hspace{-0.05cm}=\hspace{-0.05cm}  - 10 $ & $\chi(\O_D) =  -1  $ \\
\end{tabu}\\
\hline
\begin{tabu}{@{}X[5.2cm] @{}|@{}X[cm]@{}}
$ e(\eta_\bullet^{-1}(a,x,\bm{m})) =
 Q_1 Q_2 Q_3^{2}  \frac{(\mathsf{V}_{\square \square \emptyset})}{(\mathsf{V}_{\square \emptyset \emptyset})^2(\mathsf{V}_{\emptyset\emptyset\emptyset})} \cdot Q_\sigma^a\Psi^{\emptyset,\emptyset}(a) \cdot \prod \limits_{i=1}^{12}Q_3^{m_i}\Phi^{\emptyset,\emptyset}(m_i) $  &  \hfil \includegraphics[width=2cm]{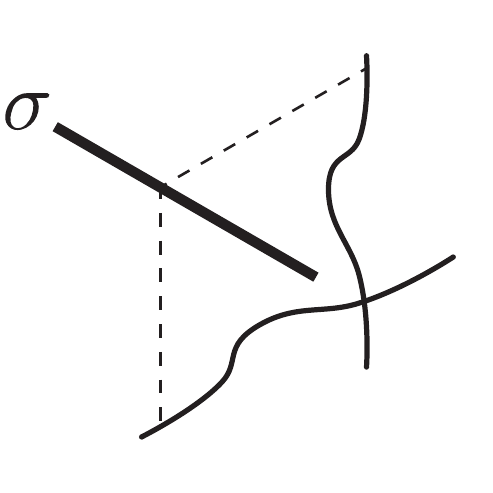} \hfil  \\
\end{tabu}\\
\hline
\end{tabu}
\mbox{}\\

From lemmas \ref{vertex_squares_lemma} and \ref{banana_subpartition_functions} we have:
\begin{enumerate}[label={\arabic*)}]
\item $\mathsf{V}_{\emptyset\emptyset\emptyset} = M(p)$
\item $\mathsf{V}_{\square\emptyset\emptyset} = M(p) \frac{1}{1-p}$
\item $\mathsf{V}_{\square\square\emptyset} = M(p) \frac{p^2-p+1}{p(1-p)^2}$
\item $\sum_{m\geq 0} Q^{m}\Phi^{\emptyset,\emptyset}(m) = M(p)^2 \prod\limits_{m>0}(1+p^m Q)^{m}$
\end{enumerate}

Summing the contributions from the above groupings we arrive at the overall contribution from part a:
\begin{align*}
&e\Big(\mathbb{Z}_{\geq0} \times S_{1}^\circ \times S_{2}^\circ \times \Sym_{Q_3}^{\bullet}(B_{\mathsf{op}}) , ~ (\eta_\bullet)_*1 \Big) \\
&=
Q_1 Q_2Q_3^2 M(p)^{24} \left( \prod\limits_{m>0}(1+p^m Q_{\sigma})^{m}(1+p^m Q_3)^{12m}\right) \\
&\hspace{1em} \cdot
\left(
\frac{2 \left(p^2+10 p+1\right) \left(p^4-2 p^3+8 p^2-2 p+1\right)}{(p-1)^4 p}
\right) \\
&=
Q_1 Q_2Q_3^2 Q_\sigma M(p)^{24} \left( \prod\limits_{m>0}(1+p^m Q_{\sigma})^{m}(1+p^m Q_3)^{12m}\right)\\
&\hspace{1em} \cdot \left(120  \psi_0+Q_\sigma \left(144 \psi_0^2+48 \psi_0+4\right)\right)
\end{align*}

\end{re}

\vspace{0.2cm}
\begin{re}\textbf{Part b-c:}
By the symmetry of $X$ we only need to consider part b, with part c being completely analogous. For each $k\in\{1,\ldots,12\}$ we begin by decomposing $S_{1}^\circ$ into the following six parts:
\[
\mathtt{Sm}_1^\sigma ~\amalg~ \mathtt{Sm}_1^\emptyset ~\amalg~ \mathtt{N}_1^{\sigma,(k)} ~\amalg~ \mathtt{N}_1^{\sigma,c} ~\amalg~ \mathtt{N}_1^{\emptyset,(k)} ~\amalg~ \mathtt{N}_1^{\emptyset,c}
\]
where $\mathtt{N}_1^{\sigma,(k)}$ is the connected component of $\mathtt{N}_1^{\sigma}$ corresponding the the $k$th banana configuration and $\mathtt{N}_1^{\sigma,c}$ is its complement in $\mathtt{N}_1^{\sigma}$. The same definition is true for $\mathtt{N}_1^{\emptyset}$. We use the above six-part decomposition for
\begin{align*}
 \mathbb{Z}_{\geq0} \times S_{1}^\circ  \times\Sym_{Q_3}^\bullet( \{ b_{\mathsf{op}}^{k}\}) \times \Sym_{Q_3}^\bullet( B_{\mathsf{op}}\setminus \{b_{\mathsf{op}}^{k}\}). 
\end{align*}
The following is the summary of results from \ref{quot_decomposition_(i,j,bullet)_V_tilde} and \ref{fibre_general_(i,j,bullet)} for this stratification. \\

\begin{tabu}{|@{}X[c]@{}|}
\hline
\extrarowsep=0.2em
\begin{tabu}{@{}X[0.7c]|X[0.8c]|X[0.8c]@{}}
$U  = \mathtt{N}_1^{\sigma,c}  $ & $e(U) = 11 $ & $\chi(\O_D) =  1  $ \\
\end{tabu}\\
\hline
\begin{tabu}{@{}X[5.2cm] @{}|@{}X[cm]@{}}
$ e(\eta_\bullet^{-1}(a,x,n,\bm{m})) =
 Q_1 Q_2 Q_3 p^2 \frac{(\mathsf{V}_{\square \square \emptyset})}{(\mathsf{V}_{\square \emptyset \emptyset})(\mathsf{V}_{\emptyset \emptyset \emptyset})^2} \cdot Q_\sigma^a\Phi^{-,\emptyset}(a) \cdot Q_3^{n}\Phi^{-,-}(n) 
  \prod\limits_{
 \genfrac{.}{.}{0pt}{2}{i=1}{i\neq k}
 }^{12}  
 Q_3^{m_i}\Phi^{\emptyset,\emptyset}(m_i)  $  &  \hfil \includegraphics[width=2cm]{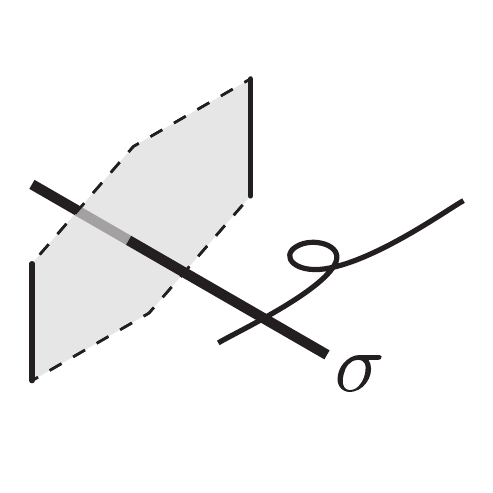} \hfil  \\
\end{tabu}\\
\hline
\extrarowsep=0.2em
\begin{tabu}{@{}X[0.7c]|X[0.8c]|X[0.8c]@{}}
$U  = \mathtt{N}_1^{\emptyset,c}  $ & {\small$ e(U)= -11 $} & $\chi(\O_D) =  1  $ \\
\end{tabu}\\
\hline
\begin{tabu}{@{}X[5.2cm] @{}|@{}X[cm]@{}}
$ e(\eta_\bullet^{-1}(a,x,n,\bm{m})) =
 Q_1 Q_2 Q_3 p^2 \frac{(\mathsf{V}_{\square \square \emptyset})}{(\mathsf{V}_{\emptyset \emptyset \emptyset})^3} \cdot Q_\sigma^a\Phi^{\emptyset,\emptyset}(a) \cdot Q_3^{n} \Phi^{-,-}(n) 
  \prod\limits_{
 \genfrac{.}{.}{0pt}{2}{i=1}{i\neq k}
 }^{12}  
 Q_3^{m_i}\Phi^{\emptyset,\emptyset}(m_i)  $   &  \hfil \includegraphics[width=2cm]{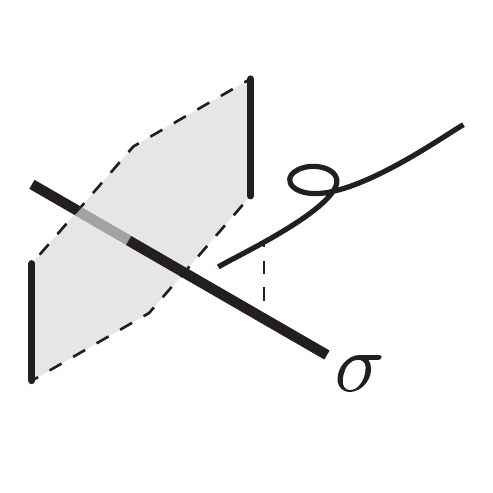} \hfil  \\
\end{tabu}\\
\hline
\extrarowsep=0.2em
\begin{tabu}{@{}X[0.7c]|X[0.8c]|X[0.8c]@{}}
$U   =\mathtt{N}_1^{\sigma,(k)}  $ & $e(U) = 1 $ & $\chi(\O_D) =  0   $ \\
\end{tabu}\\
\hline
\begin{tabu}{@{}X[5.2cm] @{}|@{}X[cm]@{}}
$ e(\eta_\bullet^{-1}(a,x,n,\bm{m})) =
 Q_1 Q_2 Q_3p^2 \frac{ (\mathsf{V}_{\square \square \square})}{(\mathsf{V}_{\square\emptyset \emptyset})^2(\mathsf{V}_{\emptyset\emptyset \emptyset})} \cdot Q_\sigma^a \Phi^{-,\emptyset}(a) \cdot Q_3^{n}\Phi^{-,-}(n) 
  \prod\limits_{
 \genfrac{.}{.}{0pt}{2}{i=1}{i\neq k}
 }^{12}  
 Q_3^{m_i}\Phi^{\emptyset,\emptyset}(m_i) $  &  \hfil \includegraphics[width=2cm]{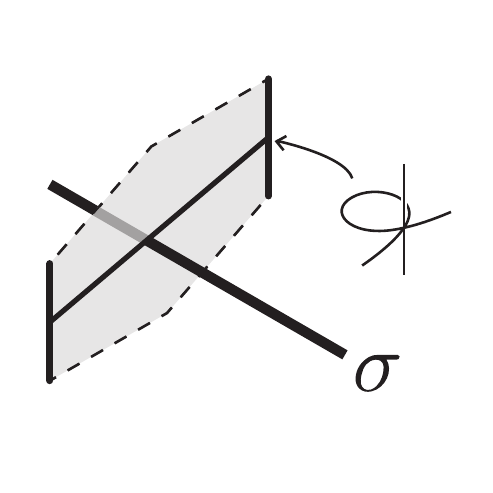} \hfil  \\
\end{tabu}\\
\hline
\extrarowsep=0.2em
\begin{tabu}{@{}X[0.7c]|X[0.8c]|X[0.8c]@{}}
$U = \mathtt{N}_1^{\emptyset,(k)} $ & $e(U)= -1$ & $\chi(\O_D) =  0    $ \\
\end{tabu}\\
\hline
\begin{tabu}{@{}X[5.2cm] @{}|@{}X[cm]@{}}
$ e(\eta_\bullet^{-1}(a,x,n,\bm{m})) =
 Q_1 Q_2 Q_3 p^2\frac{ (\mathsf{V}_{\square \square \square})}{(\mathsf{V}_{\square\emptyset \emptyset})(\mathsf{V}_{\emptyset \emptyset \emptyset})^2} \cdot Q_\sigma^a\Phi^{\emptyset,\emptyset}(a) \cdot Q_3^{n}\Phi^{-,-}(n) 
  \prod\limits_{
 \genfrac{.}{.}{0pt}{2}{i=1}{i\neq k}
 }^{12}  
 Q_3^{m_i}\Phi^{\emptyset,\emptyset}(m_i)  $&  \hfil \includegraphics[width=2cm]{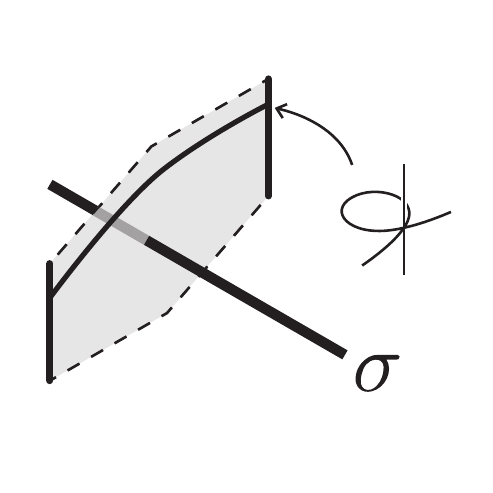} \hfil  \\
\end{tabu}\\
\hline
\extrarowsep=0.2em
\begin{tabu}{@{}X[0.7c]|X[0.8c]|X[0.8c]@{}}
$U   = \mathtt{Sm}_1^\sigma$ & $e(U) = -10 $ & $\chi(\O_D) =  1  $  \\
\end{tabu}\\
\hline
\begin{tabu}{@{}X[5.2cm] @{}|@{}X[cm]@{}}
$ e(\eta_\bullet^{-1}(a,x,n,\bm{m})) =
 Q_1 Q_2 Q_3 p \frac{1}{(\mathsf{V}_{\square \emptyset \emptyset})(\mathsf{V}_{\emptyset \emptyset \emptyset})} \cdot Q_\sigma^a\Phi^{-,\emptyset}(a) \cdot Q_3^{n}\Phi^{-,-}(n) 
  \prod\limits_{
 \genfrac{.}{.}{0pt}{2}{i=1}{i\neq k}
 }^{12}  
 Q_3^{m_i}\Phi^{\emptyset,\emptyset}(m_i)  $  &  \hfil \includegraphics[width=2cm]{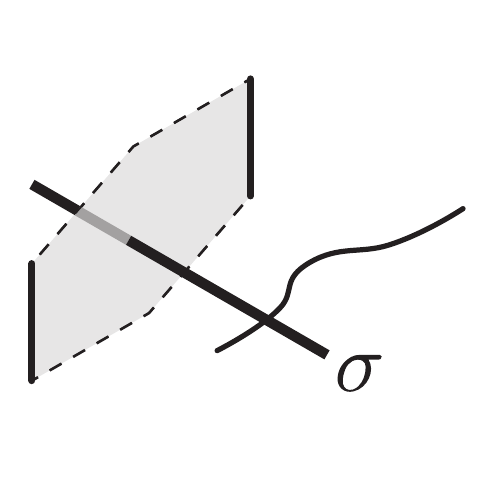} \hfil  \\
\end{tabu}\\
\hline
\extrarowsep=0.2em
\begin{tabu}{@{}X[0.7c]|X[0.8c]|X[0.8c]@{}}
$U   = \mathtt{Sm}_1^\emptyset$ & $e(U)=  10 $ & $\chi(\O_D) =  1  $ \\
\end{tabu}\\
\hline
\begin{tabu}{@{}X[5.2cm] @{}|@{}X[cm]@{}}
$ e(\eta_\bullet^{-1}(a,x,n,\bm{m})) =$
$Q_1 Q_2 Q_3 p\frac{1}{(\mathsf{V}_{\emptyset \emptyset \emptyset})^2}  \cdot Q_\sigma^a\Phi^{\emptyset,\emptyset}(a) \cdot Q_3^{n}\Phi^{-,-}(n) 
 \prod\limits_{
 \genfrac{.}{.}{0pt}{2}{i=1}{i\neq k}
 }^{12}  
Q_3^{m_i}\Phi^{\emptyset,\emptyset}(m_i) $  &  \hfil \includegraphics[width=2cm]{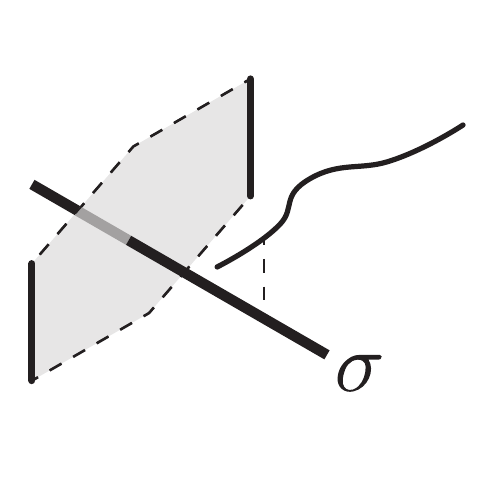} \hfil  \\
\end{tabu}\\
\hline
\end{tabu}

From lemmas \ref{vertex_squares_lemma}, \ref{banana_subpartition_functions} and \ref{vertex_section_formulas} we have:
\begin{enumerate}[label={\arabic*)}]
\item $\mathsf{V}_{\emptyset\emptyset\emptyset} = M(p)$
\item $\mathsf{V}_{\square\emptyset\emptyset} = M(p) \frac{1}{1-p}$
\item $\mathsf{V}_{\square\square\emptyset} = M(p) \frac{p^2-p+1}{p(1-p)^2}$
\item $\mathsf{V}_{\square\square\square} = M(p) \frac{p^4-p^3+p^2-p+1}{ p^2(1-p)^3}$
\item $\sum_{m\geq 0} Q^{m}\Phi^{\emptyset,\emptyset}(m) = M(p)^2 \prod\limits_{m>0}(1+p^m Q)^{m}$
\item $\sum_{m\geq 0} Q^{m}\Phi^{-,\emptyset}(m) = M(p)^2\dfrac{1+Q}{1-p} \prod\limits_{m>0}(1+p^{m} Q)^m$
\item $\sum_{m\geq 0} Q^{m}\Phi^{-,-}(m) = M(p)^2\dfrac{1}{p} (\psi_0 +(\psi_1 +2\psi_0) Q +\psi_0 Q^2) \prod\limits_{m>0}(1+p^m Q)^{m}$.
\end{enumerate}

There are 12 singular fibres of $\mathrm{pr}$. So, we have that the combined contribution from parts c and d is:

\begin{align*}
e\Big(&\underset{k=1}{\overset{12}{\amalg}}  S_{1}^\circ  \times\Sym_{Q_3}^\bullet( \{ b_{\mathsf{op}}^{k}\}) \times \Sym_{Q_3}^\bullet( B_{\mathsf{op}}\setminus \{b_{\mathsf{op}}^{k}\}) ,~ (\eta_\bullet)_*1 \Big)\\
+&e\Big(\underset{k=1}{\overset{12}{\amalg}}  S_{2}^\circ  \times\Sym_{Q_3}^\bullet( \{ b_{\mathsf{op}}^{k}\}) \times \Sym_{Q_3}^\bullet( B_{\mathsf{op}}\setminus \{b_{\mathsf{op}}^{k}\})  ,~ (\eta_\bullet)_*1 \Big)\\
=~&24Q_1 Q_2 Q_3 Q_\sigma M(p)^{24} \left( \prod\limits_{m>0}(1+p^m Q_{\sigma})^{m}(1+p^m Q_3)^{12m}\right)\\
&\cdot
(\psi_0 +(\psi_1 +2\psi_0) Q_3 +\psi_0 Q_3^2)
 \left(
12 \psi_0+4 \psi_1+\psi_2
\right)
\end{align*}

\end{re}

\vspace{0.5cm}
\begin{re}\textbf{Part d-e:}
Parts d and e parametrise the cases when $D$ is the union of $C_2^{(k)}$ and $C_2^{(l)}$. We have the spaces:
\begin{enumerate}[label={\arabic*)}]
\item  $\hspace{-0.58em}\underset{\mbox{\tiny$\begin{array}{c}k,l=1\\k\neq l\end{array}$}}{\overset{12}{\amalg}} \hspace{-0.58em} \Sym_{Q_3}^\bullet( \{ b_{\mathsf{op}}^{k}\})   \times\Sym_{Q_3}^\bullet( \{ b_{\mathsf{op}}^{l}\}) \times \Sym_{Q_3}^\bullet( B_{\mathsf{op}}\setminus \{b_{\mathsf{op}}^{k},b_{\mathsf{op}}^{l}\})$,
\item  $~\underset{k=1}{\overset{12}{\amalg}}  \Sym_{Q_3}^\bullet( \{ b_{\mathsf{op}}^{k}\}) \times \Sym_{Q_3}^\bullet( B_{\mathsf{op}}\setminus \{b_{\mathsf{op}}^{k}\})$.
\end{enumerate}

The following table is the summary of results from \ref{quot_decomposition_(i,j,bullet)_V_tilde} and \ref{fibre_general_(i,j,bullet)} for this stratification.\\

\begin{tabu}{|@{}X[c]@{}|}
\hline
\extrarowsep=0.2em
\begin{tabu}{@{}X[1.2c]|X[.7c]|X[1c]@{}}
$U  =\{(k,l)\}$, $k\neq l$& $e(U) = 1$ & $\chi(\O_D) =  2  $ \\
\end{tabu}\\
\hline
\begin{tabu}{@{}X[5.2cm] @{}|@{}X[cm]@{}}
{\small $ e(\eta_\bullet^{-1}(a,c,d,\bm{m})) =
 Q_1 Q_2 p^2 \frac{1}{(\mathsf{V}_{\emptyset \emptyset \emptyset})^2} \cdot Q_\sigma^a\Phi^{\emptyset,\emptyset}(a)  \cdot Q_3^{c} \Phi^{-,-}(c)\cdot Q_3^{d} \Phi^{-,-}(d) 
  \prod\limits_{
 \genfrac{.}{.}{0pt}{2}{i=1}{i\neq k,l}
 }^{12}  
 Q_3^{m_i}\Phi^{\emptyset,\emptyset}(m_i)  $ } &  \hfil \includegraphics[width=2cm]{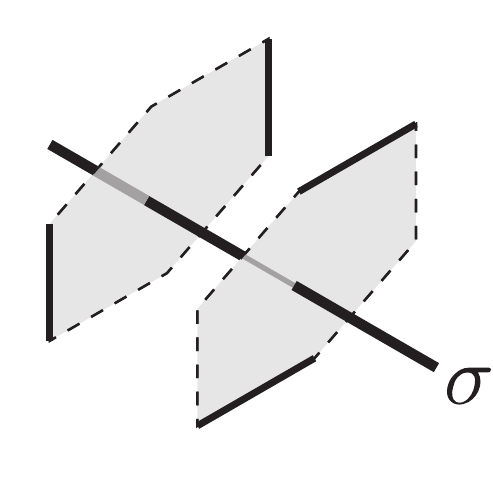} \hfil  \\
\end{tabu}\\
\hline
\extrarowsep=0.2em
\begin{tabu}{@{}X[1.2c]|X[.7c]|X[1c]@{}}
$U  =\{(k,k)\}$& $e(U) = 1 $ & $\chi(\O_D) = 0  $ \\
\end{tabu}\\
\hline
\begin{tabu}{@{}X[5.2cm] @{}|@{}X[cm]@{}}
$ e(\eta_\bullet^{-1}(a,m_k,\bm{m})) =
 Q_1 Q_2 \frac{1}{(\mathsf{V}_{\emptyset \emptyset \emptyset})^2} \cdot Q_\sigma^a\Psi^{\emptyset,\emptyset}(a) \cdot Q_3^{m_k}\Phi^{+,+}(m_k)
  \prod\limits_{
 \genfrac{.}{.}{0pt}{2}{i=1}{i\neq k}
 }^{12}  
 Q_3^{m_i}\Phi^{\emptyset,\emptyset}(m_i)  $  &  \hfil \includegraphics[width=2cm]{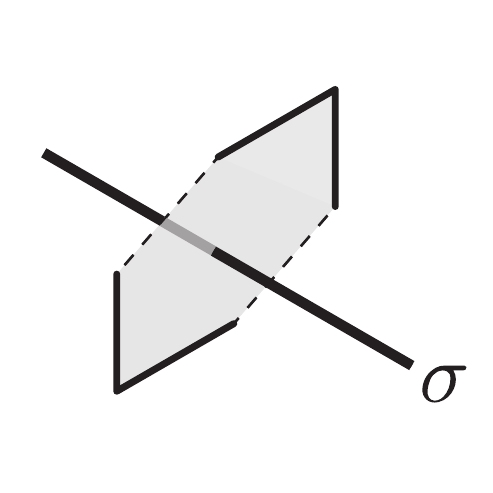} \hfil  \\
\end{tabu}\\
\hline
\end{tabu}
\mbox{}\\

From lemmas \ref{vertex_squares_lemma}, \ref{banana_subpartition_functions} and \ref{vertex_section_formulas} we have:
\begin{enumerate}[label={\arabic*)}]
\item $\mathsf{V}_{\emptyset\emptyset\emptyset} = M(p)$
\item $\sum_{m\geq 0} Q^{m}\Phi^{\emptyset,\emptyset}(m) = M(p)^2 \prod\limits_{m>0}(1+p^m Q)^{m}$
\item $\sum_{m\geq 0} Q^{m}\Phi^{-,-}(m) = M(p)^2\dfrac{1}{p} (\psi_0 +(\psi_1 +2\psi_0) Q +\psi_0 Q^2) \prod\limits_{m>0}(1+p^m Q)^{m}$.
\item $\sum_{m\geq 0} Q^{m}\Phi^{+,+}(m) $\\
$=M(p)^2 \prod\limits_{m>0}(1+p^m Q)^{m}\Big(Q^4 (2 \psi_0+\psi_1)+Q^3 (8 \psi_0+6 \psi_1+\psi_2)+Q^2 (12 \psi_0\hspace{2em}$\\
\mbox{}\hspace{12em}$+10 \psi_1+2 \psi_2)+Q (8 \psi_0+6 \psi_1+\psi_2)+(2 \psi_0+\psi_1)\Big)$
\end{enumerate}

There are $136$ choices for two distinct fibres. Hence the contribution from part d is:
\begin{align*}
e\Big(&\hspace{-0.58em}\underset{\mbox{\tiny$\begin{array}{c}k,l=1\\k\neq l\end{array}$}}{\overset{12}{\amalg}} \hspace{-0.58em} \Sym_{Q_3}^\bullet( \{ b_{\mathsf{op}}^{k}\})   \times\Sym_{Q_3}^\bullet( \{ b_{\mathsf{op}}^{l}\}) \times \Sym_{Q_3}^\bullet( B_{\mathsf{op}}\setminus \{b_{\mathsf{op}}^{k},b_{\mathsf{op}}^{l}\}),~ (\eta_\bullet)_*1 \Big)\\
=~&132Q_1 Q_2 M(p)^{24} \left( \prod\limits_{m>0}(1+p^m Q_{\sigma})^{m}(1+p^m Q_3)^{12m}\right)(\psi_0 +(\psi_1 +2\psi_0) Q_3 +\psi_0 Q_3^2)^2.
\end{align*}

The 12 singular fibres give the contribution of part e as:
\begin{align*}
e\Big(&\underset{k=1}{\overset{12}{\amalg}}  \Sym_{Q_3}^\bullet( \{ b_{\mathsf{op}}^{k}\}) \times \Sym_{Q_3}^\bullet( B_{\mathsf{op}}\setminus \{b_{\mathsf{op}}^{k}\}),~ (\eta_\bullet)_*1 \Big)\\
=~&12Q_1 Q_2 M(p)^{24} \left( \prod\limits_{m>0}(1+p^m Q_{\sigma})^{m}(1+p^m Q_3)^{12m}\right)\\
&\cdot \Big(\big( Q_3^2 \psi_0 + Q_3 (2 \psi_0 +\psi_1)+\psi_0 \big)^2 +  \big( Q_3^4 (2 \psi_0+\psi_1)+ Q_3^3 (8 \psi_0+6 \psi_1+\psi_2)\\
&\hspace{5.75cm} + Q_3^2 (12 \psi_0+10 \psi_1+2 \psi_2) \\
&\hspace{5.75cm}+ Q_3 (8 \psi_0+6 \psi_1+\psi_2) + (2 \psi_0+\psi_1)\big)\Big).
\end{align*}

Summing the contributions of parts d and e we have:
\begin{align*}
&Q_1 Q_2 M(p)^{24} \left( \prod\limits_{m>0}(1+p^m Q_{\sigma})^{m}(1+p^m Q_3)^{12m}\right)\\
&\cdot \Big(144\big( Q_3^2 \psi_0 + Q_3 (2 \psi_0 +\psi_1)+\psi_0 \big)^2 +  12\big( Q_3^4 (2 \psi_0+\psi_1)+ Q_3^3 (8 \psi_0+6 \psi_1+\psi_2)\\
&\hspace{6.5cm} + Q_3^2 (12 \psi_0+10 \psi_1+2 \psi_2) \\
&\hspace{6.5cm}+ Q_3 (8 \psi_0+6 \psi_1+\psi_2) + (2 \psi_0+\psi_1)\big)\Big).
\end{align*}

\end{re}

\begin{re}\textbf{Part f:} Recall from lemma \ref{classes_bs+(i,j,d)_main_chow_main_lemma} that part f, $\mathtt{Diag}^{\bullet}$ has the further decomposition:
\begin{enumerate}
\item[g)] ~$\mathtt{Sm}_{1} \times \Sym_{Q_3}^\bullet( B_{\mathsf{op}})$
\item[h)] ~$\amalg~ \mathtt{Sm}_{2} \times \Sym_{Q_3}^\bullet( B_{\mathsf{op}}) $
\item[i)] ~$\hspace{-0.3em}\underset{y\in\mathtt{J}}{\amalg} ~E_{\pi(y)} \times \widetilde{\mathrm{Aut}}(E_{\pi(y)}) \times \Sym_{Q_3}^\bullet( B_{\mathsf{op}}) $
\item[j)] ~$\hspace{-0.35em}\underset{k=1}{\overset{12}{\amalg}}~ \mathtt{L} \times \Sym_{Q_3}^\bullet( \{ b_{\mathsf{op}}^{k}\}) \times \Sym_{Q_3}^\bullet( B_{\mathsf{op}}\setminus \{b_{\mathsf{op}}^{k}\})$.
\end{enumerate}
Where we have used the notation:
\begin{enumerate}[label={\arabic*)}]
\item $\mathtt{J}^{0}$ and $\mathtt{J}^{1728}$ to be the subsets of points $x\in \P^1$ such that $\pi^{-1}(x)$ has $j$-invariant $0$ or $1728$ respectively and $\mathtt{J} = \mathtt{J}^{0}~\amalg ~\mathtt{J}^{1728}$.
\item $\mathtt{L}$ to be the linear system $|f_1+f_2|$ on $\P^1\times \P^1$ with the singular divisors removed where $f_1$ and $f_2$ are fibres of the two projection maps. 
\item $ \widetilde{\mathrm{Aut}}(E):= \mathrm{Aut}(E) \setminus\{\pm1\}$. 
\end{enumerate}
\end{re}

\begin{re}\textbf{Parts g-i:} 

The results for parts g-i will all be very similar. The key differences are:
\begin{enumerate}[label={\arabic*)}]
\item The overall factor of $Q_3$ may be different. 
\item The Euler characteristics of the space parametrising the $D$'s may be different. 
\end{enumerate}

We define $U$ to be one of
\begin{enumerate}
\item[g)] ~$\mathtt{Sm}_{1} $ noting that $e(\mathtt{Sm}_{1} \cap \{\sigma \} )= -10$ and $e(\mathtt{Sm}_{1} \setminus \{\sigma \} )= 10$.
\item[h)] ~$\mathtt{Sm}_{2} $ noting that $e(\mathtt{Sm}_{2} \cap \{\sigma \} )= -10$ and $e(\mathtt{Sm}_{2} \setminus \{\sigma \} )= 10$.
\item[i)] ~$E_{\pi(y)}$ for $y\in\mathtt{J}$ noting that $e(E_{\pi(y)}\cap \{\sigma \} )= 1$ and  $e(E_{\pi(y)} \setminus \{\sigma \} )= -1$.\\
\end{enumerate}

\begin{tabu}{|@{}X[c]@{}|}
\hline
\extrarowsep=0.2em
\begin{tabu}{@{}X[1c]|X[1c]@{}}
$U\cap\{\sigma\}$  & $\chi(\O_D) =  0 $ \\
\end{tabu}\\
\hline
\begin{tabu}{@{}X[5.2cm] @{}|@{}X[cm]@{}}
$ e(\eta_\bullet^{-1}(a,x,m_k,\bm{m})) =
 Q_1 Q_2 Q_3^{n} \frac{1}{(\mathsf{V}_{\square \emptyset \emptyset})(\mathsf{V}_{\emptyset \emptyset \emptyset})} \cdot Q_\sigma^a \Phi^{-,\emptyset}(a) \cdot\prod \limits_{i=1}^{12}Q_3^{m_i}\Phi^{\emptyset,\emptyset}(m_i)   $  &  \hfil \includegraphics[width=2cm]{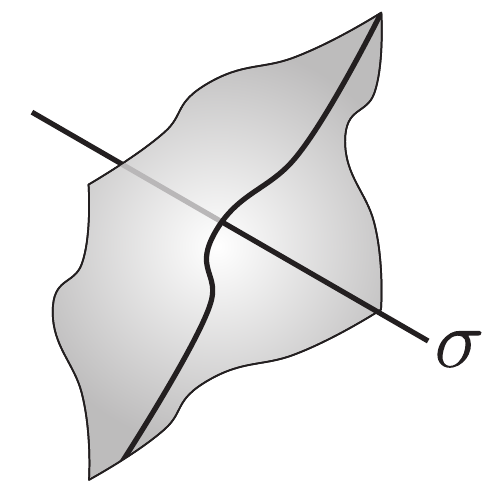} \hfil  \\
\end{tabu}\\
\hline
\extrarowsep=0.2em
\begin{tabu}{@{}X[1c]|X[1c]@{}}
$U\setminus \{\sigma\}$   & $\chi(\O_D) =  0 $ \\
\end{tabu}\\
\hline
\begin{tabu}{@{}X[5.2cm] @{}|@{}X[cm]@{}}
$ e(\eta_\bullet^{-1}(a,x,m_k,\bm{m})) =
 Q_1 Q_2 Q_3^{n}  \frac{1}{(\mathsf{V}_{\emptyset \emptyset \emptyset})^2}\cdot Q_\sigma^a\Phi^{\emptyset,\emptyset}(a) \cdot\prod \limits_{i=1}^{12}Q_3^{m_i}\Phi^{\emptyset,\emptyset}(m_i)   $ &  \hfil \includegraphics[width=2cm]{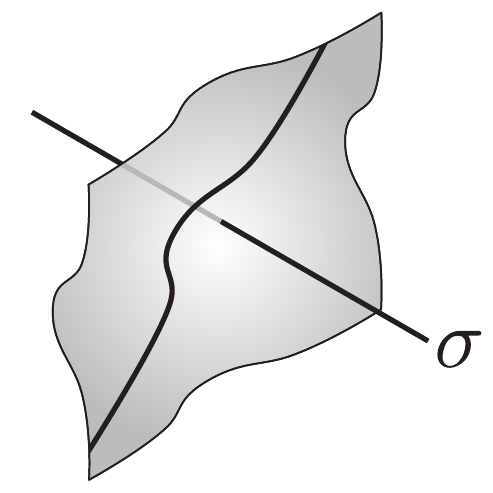} \hfil  \\
\end{tabu}\\
\hline
\end{tabu}
\mbox{}\\

From lemmas \ref{vertex_squares_lemma}, \ref{banana_subpartition_functions} and \ref{vertex_section_formulas} we have:
\begin{enumerate}[label={\arabic*)}]
\item $\mathsf{V}_{\emptyset\emptyset\emptyset} = M(p)$
\item $\mathsf{V}_{\square\emptyset\emptyset} = M(p) \frac{1}{1-p}$
\item $\sum_{m\geq 0} Q^{m}\Phi^{\emptyset,\emptyset}(m) = M(p)^2 \prod\limits_{m>0}(1+p^m Q)^{m}$
\item $\sum_{m\geq 0} Q^{m}\Phi^{-,\emptyset}(m) = M(p)^2\dfrac{1+Q}{1-p} \prod\limits_{m>0}(1+p^{m} Q)^m$
\end{enumerate}

The overall factors of $Q_3^n$ are calculated in \ref{ExE_curve_class_lemma2} to be:
\begin{enumerate}[label={\arabic*)}]
\item $n=4$ for (g) and $n=0$ for (h). 
\item
 If $j(E)=1728$ and $E \cong \C/i$ then
\begin{enumerate}
\item[-] $n=2$ occurs when $D$ is a translation of the graph $\{(x,\pm i x)\}$. 
\end{enumerate}
\item If $j(E)=0$ and $E \cong \C/\tau$ with $\tau=\frac{1}{2}(1+i \sqrt{3})$ then
\begin{enumerate}
\item[-] $n=1$ occurs when $D$ is a translation of the graph $\{(x,-\tau x)\}$ or the graph $\{(x, (\tau-1) x)\}$. 
\item[-] $n=3$ occurs when $D$ is a translation of the graph $\{(x,\tau x)\}$ or the graph $\{(x, (-\tau+1) x)\}$. 
\end{enumerate}
\end{enumerate}
Lastly, in a generic pencil we have $e(\mathtt{J}^0) = 4$ and $e(\mathtt{J}^{1728}) = 6$. \\

Hence the contribution for parts g-i is:
\begin{align*}
&Q_1 Q_2 Q_\sigma M(p)^{24} \left( \prod\limits_{m>0}(1+p^m Q_{\sigma})^{m}(1+p^m Q_3)^{12m}\right)
\left( -10 +8Q_3+12Q_3^2 + 8Q_3^3-10Q_3^4 \right).
\end{align*}\\

\end{re}

\begin{re}\textbf{Part j:} 

In the appendix \ref{L_decomposition} we give the following decomposition for $\mathtt{L}$ into groupings:
\begin{enumerate}[label={\arabic*)}]
\item $ \mathtt{L}_{(0,0), (\infty,\infty)}^\sigma ~\amalg~\mathtt{L}_{(0,0), (\infty,\infty)}^\emptyset ~\amalg~ \mathtt{L}_{(0,\infty), (\infty,0)}^\sigma~\amalg~ \mathtt{L}_{(0,\infty), (\infty,0)}^\emptyset$
\item $ \amalg~ \mathtt{L}_{(0,0)}^\sigma ~\amalg~\mathtt{L}_{(0,0)}^\emptyset ~\amalg~ \mathtt{L}_{(\infty,\infty)}^\sigma  ~\amalg~ \mathtt{L}_{(\infty,\infty)}^\emptyset$
\item $\amalg~ \mathtt{L}_{(0,\infty)}^\sigma ~\amalg~\mathtt{L}_{(0,\infty)}^\emptyset ~\amalg~ \mathtt{L}_{(\infty,0)}^\sigma ~\amalg~ \mathtt{L}_{(\infty,0)}^\emptyset$
\item $ \amalg~ \mathtt{L}_{\emptyset}^\sigma \amalg~ \mathtt{L}_{\emptyset}^\emptyset.$
\end{enumerate}
The Euler characteristics of the parts of this decomposition are computed in \ref{f1+f2_linear_system_decomp} and the overall factors of $Q_1$, $Q_2$ and $Q_3$ are calculated in lemma \ref{FSing_curve_class_lemma2}.\\

\textit{Grouping 1:} The following table is the summary of results from \ref{quot_decomposition_(i,j,bullet)_V_tilde} and \ref{fibre_general_(i,j,bullet)} for the strata in grouping 1:
\[
\mathbb{Z}_{\geq 0} \times U  \times \Sym_{Q_3}^\bullet( \{ b_{\mathsf{op}}^{k}\}) \times \Sym_{Q_3}^\bullet( B_{\mathsf{op}}\setminus \{b_{\mathsf{op}}^{k}\}).
\]

Note that the vertex is different for $ \mathtt{L}_{(0,0), (\infty,\infty)}$ as described in \ref{relative_C3_Vertex_remark}.\\

\begin{tabu}{|@{}X[c]@{}|}
\hline
\extrarowsep=0.2em
\begin{tabu}{@{}X[0.8c]|X[.6c]|X[1c]@{}}
$U  =\mathtt{L}_{(0,0), (\infty,\infty)}^\sigma$ & $e(U) \hspace{-0.05cm}=\hspace{-0.05cm}  1 $ & $\chi(\O_D) =  1 $ \\
\end{tabu}\\
\hline
\begin{tabu}{@{}X[5.2cm] @{}|@{}X[cm]@{}}
$ e(\eta_\bullet^{-1}(a,x,m_k,\bm{m})) =
 Q_1 Q_2 Q_3^{2} p \frac{1}{(\mathsf{V}_{\square \emptyset \emptyset})(\mathsf{V}_{\emptyset \emptyset \emptyset})}  Q_\sigma^a  \Phi^{-,\emptyset}(a) Q_3^{m_k}\Phi^{-,\,\mid\,}(m_k) 
  \prod\limits_{
 \genfrac{.}{.}{0pt}{2}{i=1}{i\neq k}
 }^{12}  
 Q_3^{m_i}\Phi^{\emptyset,\emptyset}(m_i)   $   &  \hfil \includegraphics[width=2cm]{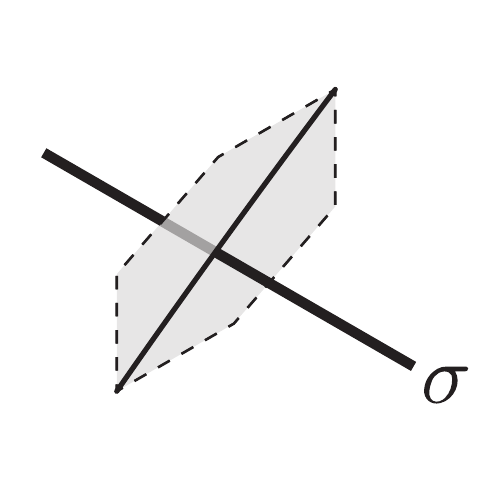} \hfil  \\
\end{tabu}\\
\hline
\extrarowsep=0.2em
\begin{tabu}{@{}X[0.8c]|X[.6c]|X[1c]@{}}
$U  = \mathtt{L}_{(0,0), (\infty,\infty)}^\emptyset $ & $e(U) \hspace{-0.05cm}=\hspace{-0.05cm}  -1 $ & $\chi(\O_D) =  1   $ \\
\end{tabu}\\
\hline
\begin{tabu}{@{}X[5.2cm] @{}|@{}X[cm]@{}}
$ e(\eta_\bullet^{-1}(a,x,m_k,\bm{m})) =
 Q_1 Q_2 Q_3^{2} p \frac{1}{(\mathsf{V}_{\emptyset \emptyset \emptyset})^2}  Q_\sigma^a  \Phi^{\emptyset,\emptyset}(a) Q_3^{m_k}\Phi^{-,\,\mid\,}(m_k) 
  \prod\limits_{
 \genfrac{.}{.}{0pt}{2}{i=1}{i\neq k}
 }^{12}  
 Q_3^{m_i}\Phi^{\emptyset,\emptyset}(m_i)   $   &  \hfil \includegraphics[width=2cm]{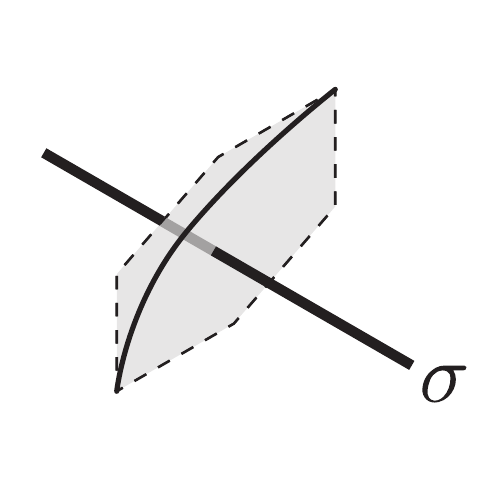} \hfil  \\
\end{tabu}\\
\hline
\extrarowsep=0.2em
\begin{tabu}{@{}X[0.8c]|X[.6c]|X[1c]@{}}
$U  \hspace{-0.05cm}=\hspace{-0.05cm}\mathtt{L}_{(0,\infty), (\infty,0)}^\sigma $ & $e(U) = 1 $ & $\chi(\O_D) =  0  $ \\
\end{tabu}\\
\hline
\begin{tabu}{@{}X[5.2cm] @{}|@{}X[cm]@{}}
$ e(\eta_\bullet^{-1}(a,x,m_k,\bm{m})) =
 Q_1 Q_2 \frac{1}{(\mathsf{V}_{\square \emptyset \emptyset})(\mathsf{V}_{\emptyset \emptyset \emptyset})} \cdot Q_\sigma^a \Phi^{-,\emptyset}(a) \cdot Q_3^{m_k}\Phi^{+,\emptyset}(m_k) 
  \prod\limits_{
 \genfrac{.}{.}{0pt}{2}{i=1}{i\neq k}
 }^{12}  
 Q_3^{m_i}\Phi^{\emptyset,\emptyset}(m_i) $  &  \hfil \includegraphics[width=2cm]{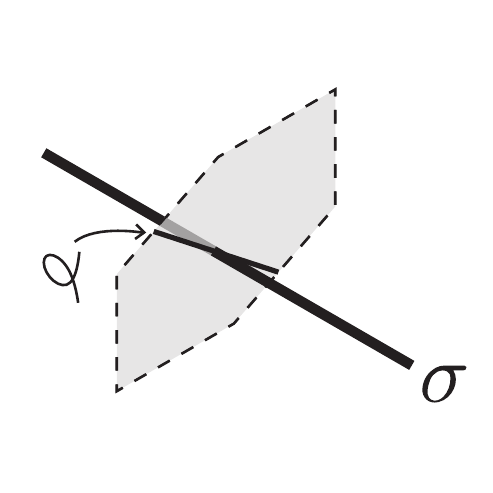} \hfil  \\
\end{tabu}\\
\hline
\extrarowsep=0.2em
\begin{tabu}{@{}X[0.8c]|X[.6c]|X[1c]@{}}
$U  \hspace{-0.05cm}=\hspace{-0.05cm} \mathtt{L}_{(0,\infty), (\infty,0)}^\emptyset $ & $e(U)  =-1$ & $\chi(\O_D) =  0 $ \\
\end{tabu}\\
\hline
\begin{tabu}{@{}X[5.2cm] @{}|@{}X[cm]@{}}
$ e(\eta_\bullet^{-1}(a,x,m_k,\bm{m})) =
Q_1 Q_2 \frac{1}{(\mathsf{V}_{\emptyset \emptyset \emptyset})^2}  \cdot Q_\sigma^b  \Phi^{\emptyset,\emptyset}(b) \cdot Q_3^{m_k}\Phi^{+,\emptyset}(m_k) 
 \prod\limits_{
 \genfrac{.}{.}{0pt}{2}{i=1}{i\neq k}
 }^{12}  
Q_3^{d_i}\Phi^{\emptyset,\emptyset}(d_i)   $  &  \hfil \includegraphics[width=2cm]{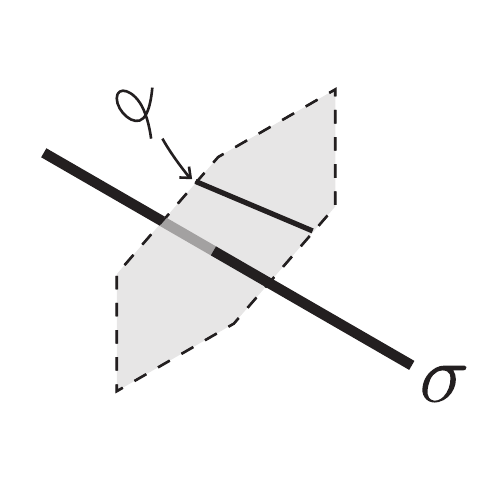} \hfil  \\
\end{tabu}\\
\hline
\end{tabu}

From lemmas \ref{vertex_squares_lemma}, \ref{banana_subpartition_functions} and \ref{vertex_section_formulas} we have:
\begin{enumerate}[label={\arabic*)}]
\item $\mathsf{V}_{\emptyset\emptyset\emptyset} = M(p)$
\item $\mathsf{V}_{\square\emptyset\emptyset} = M(p) \frac{1}{1-p}$
\item $\sum_{m\geq 0} Q^{m}\Phi^{\emptyset,\emptyset}(m) = M(p)^2 \prod\limits_{m>0}(1+p^m Q)^{m}$
\item $\sum_{m\geq 0} Q^{m}\Phi^{-,\emptyset}(m) = M(p)^2\dfrac{1+Q}{1-p} \prod\limits_{m>0}(1+p^{m} Q)^m$
\item $\sum_{m\geq 0} Q^{m}\Phi^{-,-}(m) = M(p)^2\frac{1}{p} (\psi_0 +(\psi_1 +2\psi_0) Q +\psi_0 Q^2) \prod\limits_{m>0}(1+p^m Q)^{m}$
\item {\small $\sum_{m\geq 0} Q^{m}\Phi^{-,\,\mid\,}(m) = M(p)^2(\psi_0 +(2\psi_0+ \psi_1 ) Q +(\psi_0+\psi_1) Q^2) \prod\limits_{m>0}(1+p^m Q)^{m}$}
\item $\sum_{m\geq 0} Q^{m}\Phi^{+,\emptyset}(m) =M(p)^2(\psi_1+\psi_0 +(\psi_1 +2\psi_0) Q +\psi_0 Q^2) \prod\limits_{m>0}(1+p^m Q)^{m}$.
\end{enumerate}

So the after accounting for the 12 singular fibres we have the contribution from grouping 1 as:
\begin{align*}
& Q_1 Q_2 M(p)^{24} \left( \prod\limits_{m>0}(1-p^m Q_{\sigma})^{m}(1-p^m Q_3)^{12m}\right)\\
&\cdot 12 Q_\sigma Q_3^2 
\left( (\psi_0+\psi_1)+ Q_3 (2 \psi_0+\psi_1)+2 Q_3^2 \psi_0+Q_3^3 (2 \psi_0+\psi_1)+Q_3^4 (\psi_0+\psi_1)
\right)
\end{align*}

\textit{Grouping 2:} We compute the results for $\mathtt{L}_{(0,0)}$ with $\mathtt{L}_{(\infty,\infty)}$ being completely analogous. The following table is the summary of results from \ref{quot_decomposition_(i,j,bullet)_V_tilde} and \ref{fibre_general_(i,j,bullet)} for the strata in grouping 2:
\[
\mathbb{Z}_{\geq 0} \times U  \times \Sym_{Q_3}^\bullet( \{ b_{\mathsf{op}}^{k}\}) \times \Sym_{Q_3}^\bullet( B_{\mathsf{op}}\setminus \{b_{\mathsf{op}}^{k}\}).
\]

\begin{tabu}{|@{}X[c]@{}|}
\hline
\extrarowsep=0.2em
\begin{tabu}{@{}X[0.89c]|X[.55c]|X[1c]@{}}
$U  =\mathtt{L}_{(0,0)}^\sigma$ & $e(U) \hspace{-0.05cm}=\hspace{-0.05cm}  -1 $ & $\chi(\O_D) =  1 $ \\
\end{tabu}\\
\hline
\begin{tabu}{@{}X[5.2cm] @{}|@{}X[cm]@{}}
$ e(\eta_\bullet^{-1}(a,x,m_k,\bm{m})) =
Q_1 Q_2 Q_3^2 p\frac{1}{(\mathsf{V}_{\emptyset \emptyset \emptyset})^2} Q_\sigma^a \Phi^{-,\emptyset}(a) Q_3^{m_k}\Phi^{-,\emptyset}(m_k) 
 \prod\limits_{
 \genfrac{.}{.}{0pt}{2}{i=1}{i\neq k}
 }^{12}  
Q_3^{m_i}\Phi^{\emptyset,\emptyset}(m_i) $ 
    &  \hfil \includegraphics[width=2cm]{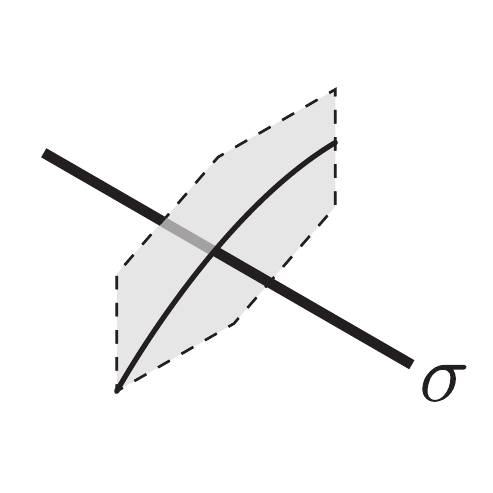} \hfil  \\
\end{tabu}\\
\hline
\extrarowsep=0.2em
\begin{tabu}{@{}X[0.89c]|X[.55c]|X[1c]@{}}
$U  = \mathtt{L}_{(0,0)}^\emptyset $ & $e(U) \hspace{-0.05cm}=\hspace{-0.05cm}  1 $ & $\chi(\O_D) =  1   $ \\
\end{tabu}\\
\hline
\begin{tabu}{@{}X[5.2cm] @{}|@{}X[cm]@{}}
$ e(\eta_\bullet^{-1}(a,x,m_k,\bm{m})) =
   Q_1 Q_2 Q_3^2 p\frac{(\mathsf{V}_{\square \emptyset \emptyset})}{(\mathsf{V}_{\emptyset \emptyset \emptyset})^3} Q_\sigma^a \Phi^{\emptyset,\emptyset}(a) Q_3^{m_k}\Phi^{-,\emptyset}(m_k) 
    \prod\limits_{
 \genfrac{.}{.}{0pt}{2}{i=1}{i\neq k}
 }^{12}  
   Q_3^{m_i}\Phi^{\emptyset,\emptyset}(m_i) $ 
      &  \hfil \includegraphics[width=2cm]{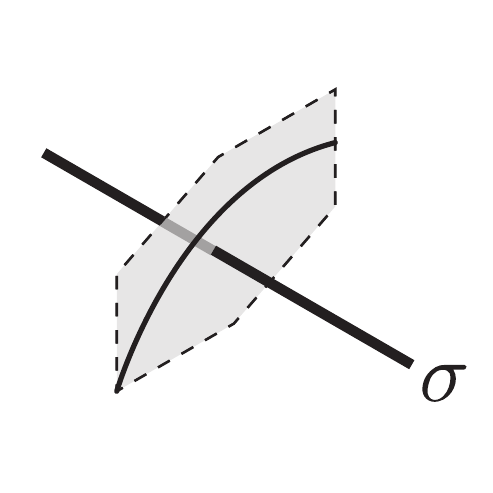} \hfil  \\
\end{tabu}\\
\hline
\end{tabu}
\mbox{}\\

From lemmas \ref{vertex_squares_lemma}, \ref{banana_subpartition_functions} and \ref{vertex_section_formulas} we have:
\begin{enumerate}[label={\arabic*)}]
\item $\mathsf{V}_{\emptyset\emptyset\emptyset} = M(p)$
\item $\mathsf{V}_{\square\emptyset\emptyset} = M(p) \frac{1}{1-p}$
\item $\sum_{m\geq 0} Q^{m}\Phi^{\emptyset,\emptyset}(m) = M(p)^2 \prod\limits_{m>0}(1+p^m Q)^{m}$
\item $\sum_{m\geq 0} Q^{m}\Phi^{-,\emptyset}(m) = M(p)^2\dfrac{1+Q}{1-p} \prod\limits_{m>0}(1+p^{m} Q)^m$
\end{enumerate}

Accounting for both $\mathtt{L}_{(0,0)}$ and $\mathtt{L}_{(\infty,\infty)}$,  the contribution for grouping 2 is:
\begin{align*}
&Q_1 Q_2  M(p)^{24} \left(\prod\limits_{m>0}(1+p^m Q_\sigma)^{m}(1+p^m Q_3)^{12m}\right)\\
&\cdot (-24)  Q_\sigma Q_3^2
 \left(\psi_0 + Q_3 \psi_0 \right)
\end{align*}

\textit{Grouping 3:} We compute the results for $\mathtt{L}_{(0,\infty)}$ with $\mathtt{L}_{(\infty,0)}$ being completely analogous. The following table is the summary of results from \ref{quot_decomposition_(i,j,bullet)_V_tilde} and \ref{fibre_general_(i,j,bullet)} for the strata in grouping 3:
\[
\mathbb{Z}_{\geq 0} \times U  \times \Sym_{Q_3}^\bullet( \{ b_{\mathsf{op}}^{k}\}) \times \Sym_{Q_3}^\bullet( B_{\mathsf{op}}\setminus \{b_{\mathsf{op}}^{k}\}).
\]

\begin{tabu}{|@{}X[c]@{}|}
\hline
\extrarowsep=0.2em
\begin{tabu}{@{}X[0.89c]|X[.55c]|X[1c]@{}}
$U  =\mathtt{L}_{(0,\infty)}^\sigma$ & $e(U) \hspace{-0.05cm}=\hspace{-0.05cm}  -1 $ & $\chi(\O_D) =  1 $ \\
\end{tabu}\\
\hline
\begin{tabu}{@{}X[5.2cm] @{}|@{}X[cm]@{}}
$ e(\eta_\bullet^{-1}(a,x,m_k,\bm{m})) =
Q_1 Q_2 Q_3 p\frac{1}{(\mathsf{V}_{\emptyset \emptyset \emptyset})^2} Q_\sigma^a \Phi^{-,\emptyset}(a) Q_3^{m_k}\Phi^{-,\emptyset}(m_k)
 \prod\limits_{
 \genfrac{.}{.}{0pt}{2}{i=1}{i\neq k}
 }^{12}  
Q_3^{m_i}\Phi^{\emptyset,\emptyset}(m_i) $ 
   &  \hfil \includegraphics[width=2cm]{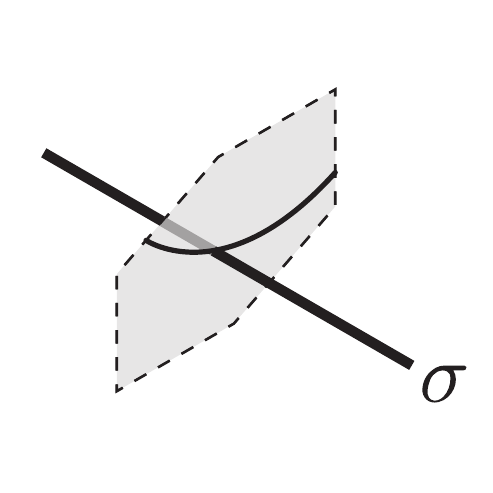} \hfil  \\
\end{tabu}\\
\hline
\extrarowsep=0.2em
\begin{tabu}{@{}X[0.89c]|X[.55c]|X[1c]@{}}
$U  = \mathtt{L}_{(0,\infty)}^\emptyset $ & $e(U) \hspace{-0.05cm}=\hspace{-0.05cm}  1 $ & $\chi(\O_D) =  1   $ \\
\end{tabu}\\
\hline
\begin{tabu}{@{}X[5.2cm] @{}|@{}X[cm]@{}}
$ e(\eta_\bullet^{-1}(a,x,m_k,\bm{m})) =
   Q_1 Q_2 Q_3 p\frac{(\mathsf{V}_{\square \emptyset \emptyset})}{(\mathsf{V}_{\emptyset \emptyset \emptyset})^3} Q_\sigma^a \Phi^{\emptyset,\emptyset}(a) Q_3^{m_k}\Phi^{-,\emptyset}(m_k) 
    \prod\limits_{
 \genfrac{.}{.}{0pt}{2}{i=1}{i\neq k}
 }^{12}  
   Q_3^{m_i}\Phi^{\emptyset,\emptyset}(m_i) $ 
  &  \hfil \includegraphics[width=2cm]{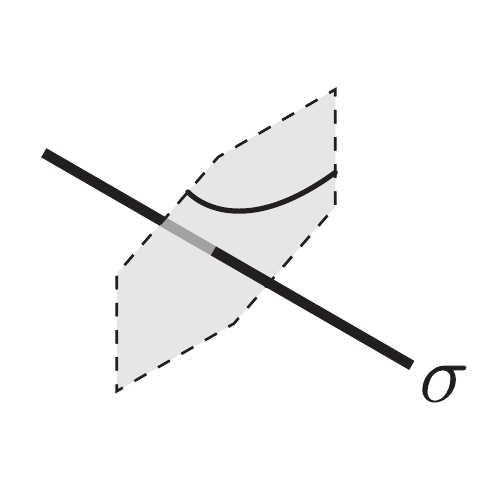} \hfil  \\
\end{tabu}\\
\hline
\end{tabu}
\mbox{}\\

From lemmas \ref{vertex_squares_lemma}, \ref{banana_subpartition_functions} and \ref{vertex_section_formulas} we have:
\begin{enumerate}[label={\arabic*)}]
\item $\mathsf{V}_{\emptyset\emptyset\emptyset} = M(p)$
\item $\mathsf{V}_{\square\emptyset\emptyset} = M(p) \frac{1}{1-p}$
\item $\sum_{m\geq 0} Q^{m}\Phi^{\emptyset,\emptyset}(m) = M(p)^2 \prod\limits_{m>0}(1+p^m Q)^{m}$
\item $\sum_{m\geq 0} Q^{m}\Phi^{-,\emptyset}(m) = M(p)^2\dfrac{1+Q}{1-p} \prod\limits_{m>0}(1+p^{m} Q)^m$
\end{enumerate}

Accounting for both $\mathtt{L}_{(0,\infty)}$ and $\mathtt{L}_{(\infty,0)}$,  the contribution for grouping 3 is:
\begin{align*}
&Q_1 Q_2  M(p)^{24} \left(\prod\limits_{m>0}(1+p^m Q_\sigma)^{m}(1+p^m Q_3)^{12m}\right)\\
&\cdot (-24)  Q_\sigma Q_3^2
 \left(\psi_0 + Q_3 \psi_0 \right)
\end{align*}

\textit{Grouping 4:} The following table is the summary of results from \ref{quot_decomposition_(i,j,bullet)_V_tilde} and \ref{fibre_general_(i,j,bullet)} for the strata in grouping 4:
\[
\mathbb{Z}_{\geq 0} \times U  \times \Sym_{Q_3}^\bullet( \{ b_{\mathsf{op}}^{k}\}) \times \Sym_{Q_3}^\bullet( B_{\mathsf{op}}\setminus \{b_{\mathsf{op}}^{k}\}).
\]

\begin{tabu}{|@{}X[c]@{}|}
\hline
\extrarowsep=0.2em
\begin{tabu}{@{}X[0.89c]|X[.55c]|X[1c]@{}}
$U  =\mathtt{L}_{\emptyset}^\sigma$ & $e(U) \hspace{-0.05cm}=\hspace{-0.05cm}  2 $ & $\chi(\O_D) =  1 $ \\
\end{tabu}\\
\hline
\begin{tabu}{@{}X[5.2cm] @{}|@{}X[cm]@{}}
$ e(\eta_\bullet^{-1}(a,x,m_k,\bm{m})) =
Q_1 Q_2 Q_3^2 p\frac{(\mathsf{V}_{\square \emptyset \emptyset})}{(\mathsf{V}_{\emptyset \emptyset \emptyset})^3} Q_\sigma^a \Phi^{-,\emptyset}(a) \prod\limits_{i=1}^{12}  Q_3^{m_i}\Phi^{\emptyset,\emptyset}(m_i) $ 
   &  \hfil \includegraphics[width=2cm]{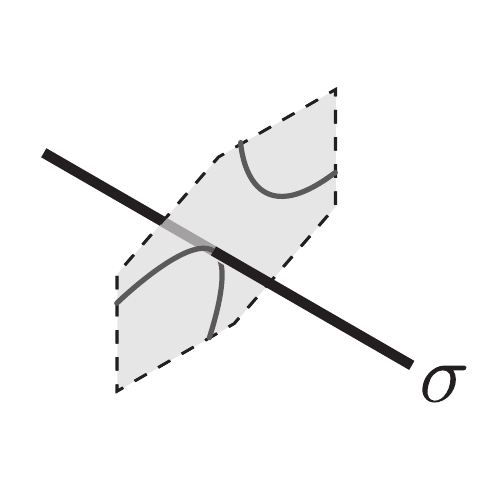} \hfil  \\
\end{tabu}\\
\hline
\extrarowsep=0.2em
\begin{tabu}{@{}X[0.89c]|X[.55c]|X[1c]@{}}
$U  = \mathtt{L}_{\emptyset}^\emptyset $ & $e(U) \hspace{-0.05cm}=\hspace{-0.05cm}  -2 $ & $\chi(\O_D) =  1   $ \\
\end{tabu}\\
\hline
\begin{tabu}{@{}X[5.2cm] @{}|@{}X[cm]@{}}
$ e(\eta_\bullet^{-1}(a,x,m_k,\bm{m})) =
Q_1 Q_2 Q_3^2 p\frac{(\mathsf{V}_{\square \emptyset \emptyset})^2}{(\mathsf{V}_{\emptyset \emptyset \emptyset})^4} Q_\sigma^a \Phi^{\emptyset,\emptyset}(a) \prod\limits_{i=1}^{12}  Q_3^{m_i}\Phi^{\emptyset,\emptyset}(m_i) $ 
  &  \hfil \includegraphics[width=2cm]{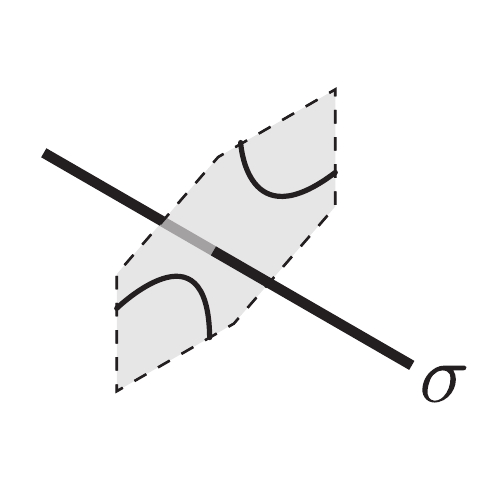} \hfil  \\
\end{tabu}\\
\hline
\end{tabu}
\mbox{}\\

From lemmas \ref{vertex_squares_lemma}, \ref{banana_subpartition_functions} and \ref{vertex_section_formulas} we have:
\begin{enumerate}[label={\arabic*)}]
\item $\mathsf{V}_{\emptyset\emptyset\emptyset} = M(p)$
\item $\mathsf{V}_{\square\emptyset\emptyset} = M(p) \frac{1}{1-p}$
\item $\sum_{m\geq 0} Q^{m}\Phi^{\emptyset,\emptyset}(m) = M(p)^2 \prod\limits_{m>0}(1+p^m Q)^{m}$
\item $\sum_{m\geq 0} Q^{m}\Phi^{-,\emptyset}(m) = M(p)^2\dfrac{1+Q}{1-p} \prod\limits_{m>0}(1+p^{m} Q)^m$
\end{enumerate}

So the contribution for grouping 4 is:
\begin{align*}
&Q_1 Q_2 M(p)^{24} \left(\prod\limits_{m>0}(1+p^m Q_\sigma)^{m}(1+p^m Q_3)^{12m}\right)\cdot 24 Q_\sigma Q_3^2 \psi_0
\end{align*}

Combining groupings 1-4 we have the overall contribution for part j is:
\begin{align*}
&Q_1 Q_2 M(p)^{24} \left(\prod\limits_{m>0}(1+p^m Q_\sigma)^{m}(1+p^m Q_3)^{12m}\right)\\
&\cdot 12 Q_\sigma  \Big((\psi_0 +\psi_1) + Q_3 \psi_1 +Q_3^3 \psi_1+Q_3^4(\psi_0 +\psi_1)\Big)
\end{align*}\\

\end{re}

\section{Appendix}

\subsection{Connected Invariants and their Partition Functions}\label{connected_section}

For the rank four sub-lattice $\Gamma \subset H_2(X, \mathbb{Z})$ generated by a section and banana curves, we can consider the \textit{connected} unweighted Pandharipande-Thomas invariants. They are defined formally via the following partition function
\[
\widehat{Z}^{\mathsf{PT}, \mathsf{Con}}_{\Gamma} (X) := \log\left(\frac{ \widehat{Z}_{\Gamma}(X) }{ \widehat{Z}_{(0,\bullet,\bullet)}|_{Q_i=0} }\right). 
\]
For the partition function in theorem \ref{main_DT_calc_theorem_A} we consider the first terms of the expansion in $Q_\sigma$ and $Q_1$:
\begin{align*}
\frac{ \widehat{Z}_{\Gamma}(X) }{ \widehat{Z}_{(0,\bullet,\bullet)}|_{Q_i=0} }
& = \frac{\widehat{Z}_{(0,\bullet,\bullet)} }{ \widehat{Z}_{(0,\bullet,\bullet)}|_{Q_i=0} } + Q_\sigma \frac{\widehat{Z}_{\sigma + (0,\bullet,\bullet)} }{ \widehat{Z}_{(0,\bullet,\bullet)}|_{Q_i=0} }+ \cdots \\
& =  \frac{\widehat{Z}_{(0,\bullet,\bullet)} }{ \widehat{Z}_{(0,\bullet,\bullet)}|_{Q_i=0} } \Big( 1 + Q_\sigma\frac{\widehat{Z}_{\sigma + (0,\bullet,\bullet)} }{\widehat{Z}_{(0,\bullet,\bullet)}}+ \cdots\Big).
\end{align*}
So the first terms of the expansion in $Q_\sigma$ and $Q_1$ of the connected partition function are:
\[
\widehat{Z}^{\mathsf{PT},\mathsf{Con}}_{\Gamma}(X)  
=
 \frac{\widehat{Z}_{(0,\bullet,\bullet)} }{ \widehat{Z}_{(0,\bullet,\bullet)}|_{Q_i=0} } 
-  Q_\sigma \frac{\widehat{Z}_{\sigma + (0,\bullet,\bullet)} }{\widehat{Z}_{(0,\bullet,\bullet)}}
+ \cdots. 
\]
In particular we have the connected version of $\widehat{Z}_{\sigma + (0,\bullet,\bullet)} $ as:
\begin{align*}
\widehat{Z}^{\mathsf{PT},\mathsf{Con}}_{\sigma + (0,\bullet,\bullet)}  
 = \frac{-p}{(1-p)^2} \prod_{m>0} \frac{1}{(1-Q_2^m  Q_3^m)^{8}(1-pQ_2^m  Q_3^m)^{2}(1-p^{-1}Q_2^m  Q_3^m)^{2}},
\end{align*}
proving corollary \ref{PT_sigma+(0,bullet,bullet)}. For the partition function in theorem \ref{main_DT_calc_theorem_B} we consider the first terms of the expansion in $Q_1$ and $Q_2$:
\begin{align*}
&\frac{ \widehat{Z}_{\Gamma}(X) }{ \widehat{Z}_{(0,\bullet,\bullet)}|_{Q_i=0} }\\
& = \frac{\widehat{Z}_{(0,\bullet,\bullet)} }{ \widehat{Z}_{(0,\bullet,\bullet)}|_{Q_i=0} } \Big(1  + Q_1 \frac{\widehat{Z}_{\bullet\sigma + (1,0,\bullet)}}{\widehat{Z}_{\bullet\sigma + (0,0,\bullet)}}  + Q_2 \frac{\widehat{Z}_{\bullet\sigma + (0,1,\bullet)}}{\widehat{Z}_{\bullet\sigma + (0,0,\bullet)}}  + Q_1 Q_2 \frac{\widehat{Z}_{\bullet\sigma + (1,1,\bullet)}}{\widehat{Z}_{\bullet\sigma + (0,0,\bullet)}} + \cdots \Big).
\end{align*}
So the first terms of the expansion in $Q_1$ and $Q_2$ of the connected partition function are:
\begin{align*}
\widehat{Z}^{\mathsf{PT},\mathsf{Con}}_{\Gamma}(X)  
=~&  \log\left(\frac{\widehat{Z}_{(0,\bullet,\bullet)} }{ \widehat{Z}_{(0,\bullet,\bullet)}|_{Q_i=0} } \right)
- Q_1 \frac{\widehat{Z}_{\bullet\sigma + (1,0,\bullet)}}{\widehat{Z}_{\bullet\sigma + (0,0,\bullet)}}  - Q_2 \frac{\widehat{Z}_{\bullet\sigma + (0,1,\bullet)}}{\widehat{Z}_{\bullet\sigma + (0,0,\bullet)}}\\
&  + Q_1 Q_2\left(\frac{\widehat{Z}_{\bullet\sigma + (1,0,\bullet)}\widehat{Z}_{\bullet\sigma + (0,1,\bullet)}}{\big(\widehat{Z}_{\bullet\sigma + (0,0,\bullet)}\big)^2} -  \frac{\widehat{Z}_{\bullet\sigma + (1,1,\bullet)}}{\widehat{Z}_{\bullet\sigma + (0,0,\bullet)}}  \right)+ \cdots
\end{align*}
In particular we have the connected version of $\widehat{Z}_{\bullet \sigma + (0,0,\bullet)} $ as
\begin{align*}
\widehat{Z}^{\mathsf{PT}, \mathsf{Con}}_{\bullet \sigma + (0,0,\bullet)}
=~&
\log\left(\frac{\widehat{Z}_{\bullet\sigma + (0,0,\bullet)}}{\widehat{Z}_{(0,\bullet,\bullet)}|_{Q_i =0}}\right)\\
=~&
\log\left( \prod\limits_{m>0}(1+p^m Q_{\sigma})^{m}(1+p^m Q_3)^{12m}\right)\\
=~&
\sum_{n>0}  \frac{p^n}{(1-p^n)^2}\frac{(-Q_\sigma)^n}{n}
+ \sum_{n>0} 12  \frac{p^n}{(1-p^n)^2}\frac{(-Q_3)^n}{n}\\
=~&
\sum_{n>0}  \psi_0(p^n) \frac{(-Q_\sigma)^n}{n}
+ \sum_{n>0} 12 \psi_0(p^n) \frac{(-Q_3)^n}{n}
\end{align*}
and the connected version of $\widehat{Z}_{\bullet \sigma + (0,1,\bullet)} $ (and also of $\widehat{Z}_{\bullet \sigma + (1,0,\bullet)}$) given by:
\[
\widehat{Z}^{\mathsf{PT},\mathsf{Con}}_{\bullet \sigma + (0,1,\bullet)}  = -\Big(\big(12\psi_0 +Q_3 (24\psi_0+12\psi_1)  +Q_3^2 (12\psi_0) \big) + Q_\sigma Q_3 \big(\psi_0 +2\psi_1\big) \Big)
\]
and the connected version of $\widehat{Z}_{\bullet \sigma + (1,1,\bullet)}$ given by:
 \begin{align*}
\widehat{Z}&^{\mathsf{PT},\mathsf{Con}}_{\bullet \sigma + (1,1,\bullet)}\\
=&~ \Big( 12\big( Q_3^4 (2 \psi_0+\psi_1) + Q_3^3 (8 \psi_0+6 \psi_1+\psi_2)+Q_3^2 (12 \psi_0+10 \psi_1+2 \psi_2) 
\\
&\hspace{1.7cm}
+ Q_3 (8 \psi_0+6 \psi_1+\psi_2) + (2 \psi_0+\psi_1)\big)\Big)
\\
&
+Q_\sigma\Big(\big( 12\psi_0+2\psi_1\big) + Q_3\big(48 \psi_0+44\psi_1 \big) 
+ Q_3^2\big(216\psi_0+108\psi_1+24 \psi_2\big)\\
&\hspace{1.7cm}+ Q_3^3\big(48 \psi_0+44\psi_1 \big)  + Q_3^4\big(12\psi_0+2\psi_1\big)  \Big).
\end{align*}

Corollary \ref{GV_bsigma+(i,i,bullet)}  now follows immediately. 

\subsection{Linear System in \texorpdfstring{$\P^1\times \P^1$}{P1xP1}}

In this section we consider a stratification of the following linear system in $\P^1\times \P^1$ with strata determined by the intersections of the associated divisors with a collection of points. \\

Consider the fibres of the projection maps $\mathrm{pr}_i: \P^1\times \P^1 \rightarrow \P^1$ and a fibre from each $f_i$. The linear system in $\P^1\times \P^1$ defined by the sum of a fibre from each map is $| f_1 +f_2 | = \P^3$. This is the collection of bi-homogeneous polynomials of degree $(1,1)$: 
\[
\Big\{~ a x_0 y_0 + b x_0 y_1 + c x_1 y_0 + d x_1 y_1 =0 ~\Big|~ [a:b:c:d] \in \P^3~\Big\}
\]

\begin{re}
There are five points in $\P^1\times\P^1$ that are of interest to us: 
\[
\sigma = \big([1:1], [1:1]\big) 
\hspace{0.5cm}\mbox{and}\hspace{0.5cm}
\mathsf{P}:= \big\{(0,0), (0,\infty), (\infty,0), (\infty,\infty) \Big\}. 
\]
where we have used the standard notation $0=[0:1]$ and $\infty = [1:0]$. We will decompose $| f_1 +f_2 |$ into strata based on which points the divisor intersects. Consider a divisor $D\in |f_1+f_2|$. Then $D$ passes through:
\begin{enumerate}[label={\arabic*)}]
\item $\big(0,0\big)$ if and only if $d=0$;
\item $\big(0,\infty\big)$ if and only if $c=0$;
\item $\big(\infty,0\big)$ if and only if $b=0$;
\item $\big(\infty,\infty\big)$ if and only if $a=0$.
\end{enumerate}
\end{re}

\begin{re} \label{L_decomposition}
Define the following convenient notation for $y,x \in \mathsf{P}$: 
\begin{enumerate}[label={\arabic*)}]
\item $\mathsf{Sing} \subset | f_1 +f_2 | $ is the subset of singular divisors.
\item $\mathtt{L}_{\emptyset} \subset \big(| f_1 +f_2 | \setminus  \mathsf{Sing} \big)$ is the subset of smooth curves not passing through any points of $\mathsf{P}$ .
\item $\mathtt{L}_{x} \subset \big(| f_1 +f_2 | \setminus  \mathsf{Sing} \big)$ is the subset of smooth curve passing through $x$ but no other points of $\mathsf{P}$ .
\item $\mathtt{L}_{x,y} \subset \big(| f_1 +f_2 | \setminus  \mathsf{Sing} \big)$ is the subset of smooth curve passing through $x$ and $y$ but no other points of $\mathsf{P}$.
\item Also let $\mathtt{L}^\sigma_{\emptyset}$, $\mathtt{L}^\sigma_{x}$ and $\mathtt{L}^\sigma_{x,y}$ be subsets of $\mathtt{L}_{\emptyset}$, $\mathtt{L}_{x}$ and $\mathtt{L}_{x,y}$ respectively with the further condition that the curves pass through $\sigma$. 
\item Let $\mathtt{L}^\emptyset_{\emptyset}$, $\mathtt{L}^\emptyset_{x}$ and $\mathtt{L}^\emptyset_{x,y}$ be the complements of $\mathtt{L}^\sigma_{\emptyset}$, $\mathtt{L}^\sigma_{x}$ and $\mathtt{L}^\sigma_{x,y}$ in $\mathtt{L}_{\emptyset}$, $\mathtt{L}_{x}$ and $\mathtt{L}_{x,y}$ respectively. 
\end{enumerate}

\begin{figure}
\centering\vspace{-1cm}
\begin{tabu}{ccc}
  $\mathtt{L}^\sigma_{(0,0), (\infty,\infty)}  : \hspace{-1em}\begin{array}{c}\includegraphics[width=2.5cm]{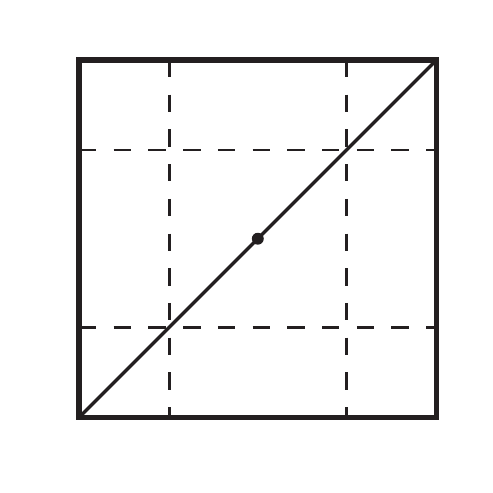}\end{array}$
&
 $\mathtt{L}^\sigma_{(0,0)}  : \hspace{-1em}\begin{array}{c}\includegraphics[width=2.5cm]{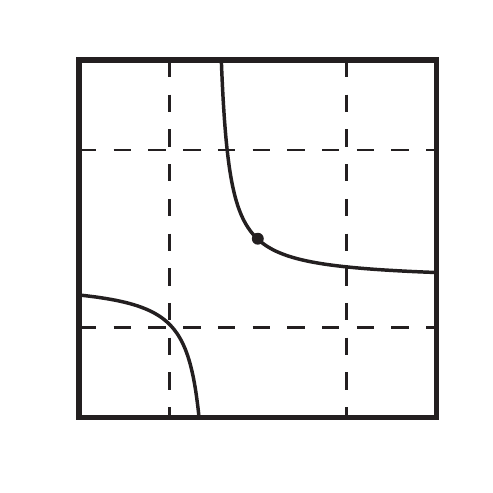}\end{array}$
 &
 $\mathtt{L}^\sigma_\emptyset : \hspace{-1em}\begin{array}{c}\includegraphics[width=2.5cm]{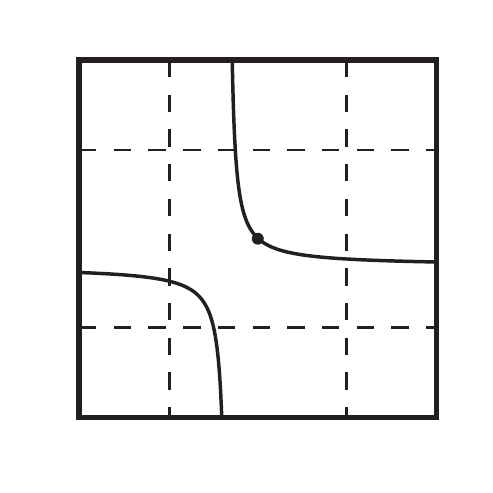} 
\end{array}$
 \vspace{-1cm}\\
$\mathtt{L}^\emptyset_{(0,0), (\infty,\infty)} : \hspace{-1em}\begin{array}{c}\includegraphics[width=2.5cm]{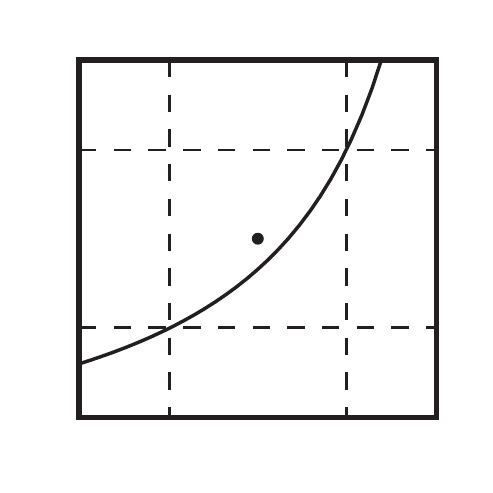} 
\end{array}$
&
$\mathtt{L}^\emptyset_{(0,0)} : \hspace{-1em}\begin{array}{c}\includegraphics[width=2.5cm]{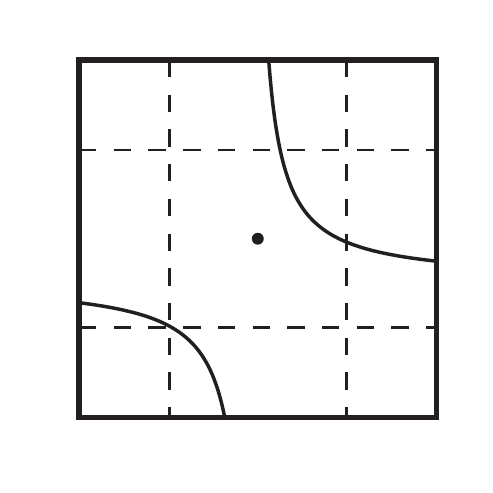} 
\end{array}$
&
$\mathtt{L}^\emptyset_\emptyset : \hspace{-1em}\begin{array}{c}\includegraphics[width=2.5cm]{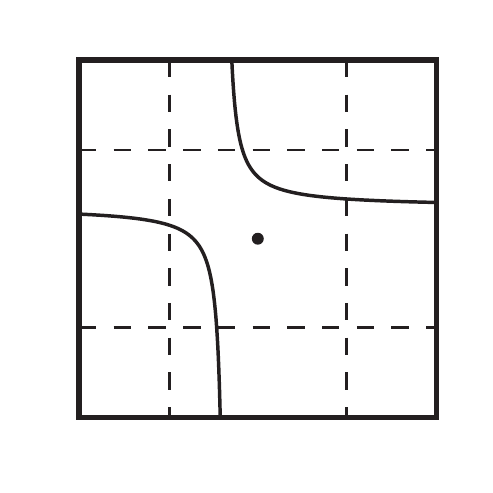} 
\end{array}$ \vspace{-0.2cm}
\end{tabu}
  \caption[Linear system $|f_1 +f_2|$ on $\P^1\times \P^1$.]{Depictions of the curves in the decomposition of the linear system $|f_1 +f_2|$ on $\P^1\times \P^1$.}
\end{figure}

With this notation we have the following decomposition of $| f_1 +f_2 |$:
\begin{align*}
| f_1 +f_2 |  
=&~
\mathsf{Sing} ~\amalg~ \mathtt{L}_{(0,0), (\infty,\infty)} ~\amalg~ \mathtt{L}_{(0,\infty), (\infty,0)}\\
&~ \hspace{2.2em}\amalg~ \mathtt{L}_{(0,0)} ~\amalg~ \mathtt{L}_{(0,\infty)} ~\amalg~ \mathtt{L}_{(\infty,0)} ~\amalg~ \mathtt{L}_{(\infty,\infty)}\\
&~ \hspace{2.2em}\amalg~ \mathtt{L}_{\emptyset}.
\end{align*}
\end{re}

\begin{re} \label{f1+f2_linear_system_decomp}
We now consider the strata of this collection and their Euler characteristics:
\begin{itemize}
\item[$\mathsf{ban}$:]  A curve in $|f_1+f_2|$ is singular if and only if the equation for the curve factorises:
\[
a x_0 y_0 + b x_0 y_1 + c x_1 y_0 + d x_1 y_1 = (\alpha x_0 +\beta x_1) (\gamma y_0 + \delta y_1) = 0
\]
where $[\alpha:\beta], [\gamma:\delta] \in \P^1$. Hence $\mathsf{Sing} \cong \P^1 \times \P^1$ and the Euler characteristic is $e(\mathsf{ban}) = e(|f_1+f_2|)=4$. 

\item[$ \mathtt{L}_{x, y}$:] We consider for $x=(0,0)$ and $y=(\infty,\infty)$ with the case $(0,\infty)$ and $(\infty,0)$ being completely analogous. The points $[a:b:d:c]\in |f_1+f_2|$ correspond to a curve passing through $x$ and $y$ if and only if $a=d=0$. Moreover, this is singular when either $b=0$ or $c=0$. Hence $ \mathtt{L}_{x, y}\cong \P^1\setminus\{0,\infty\}$ and $e(\mathtt{L}_{x, y}) = 0$.\\

\vspace{-0.1cm}
\noindent The set $\mathtt{L}^\sigma_{x, y}$ is when $b+c =0$, which is a point in $\P^1$. So we have $e(\mathtt{L}^\sigma_{x, y}) = 1$ and $e(\mathtt{L}^\emptyset_{x, y}) = -1$. 

\item[$ \mathtt{L}_{x}$:] We consider the case $x=(0,0)$ with the other cases being completely analogous. So the subspace of all divisors passing through $x$ is $[a:b:d:c]\in |f_1+f_2|$ where $d=0$. This is a $\P^2\subset \P^3$. The subspace where the curve doesn't pass through one of the other points is where $a,b,c\neq 0$ which is given by $\C^*\times \C^* \cong \P^2 \setminus \big(\{a=0\} \cup \{b=0\} \cup \{c=0\}\big)$. None of the equations for these curves factorise since such a factorisation would require either $b=0$ or $c=0$. Hence, $\mathtt{L}_{x} \cong \C^*\times \C^*$ and $e(\mathtt{L}_{x}) = 0$. \\

\vspace{-0.1cm}
\noindent The subset $\mathtt{L}^\sigma_{x}$ is defined by the further condition $a+b+c=0$ which gives
\[
\mathtt{L}^\sigma_{x} = \Big\{~ [a:b:c]\in \P^2 ~\Big|~ a,b,c \neq 0 \mbox{ and } a+b=1~\Big\} \cong \C^* \setminus \mathrm{pt}
\]
Hence we have the Euler characteristics $e(\mathtt{L}^\sigma_{x}) =-1$ and $e(\mathtt{L}^\emptyset_{x}) = 1$. 

\item[$ \mathtt{L}_{\emptyset}$:] The set of curves not passing through any points of $\mathsf{P}$ is given by 
\begin{align*}
\Big\{~  [a:b:c:d] \in |f_1+f_2| ~\Big|~ a,b,c,d \neq 0 ~\Big\}  \cong \Big\{~  b,c,d\in(\C^{*})^3  ~\Big\}
\end{align*}
The singular curves are given by the factorisation condition:
\begin{align*}
 x_0 y_0 + b x_0 y_1 + c x_1 y_0 + d x_1 y_1 = ( x_0 +\beta x_1) (y_0 + \delta y_1) 
\end{align*}
which is the condition that $d=bc$. So the subspace of curves which are singular is $(\C^*)^2 \subset (\C^*)^3$. Hence $\mathtt{L}_{\emptyset} \cong \{ (b,c,d) \in (\C^*)^3 | b\neq dc \}$ and $e(\mathtt{L}_{\emptyset}) = 0$. \\

\noindent  $ \mathtt{L}^\sigma_{\emptyset}$ is given by the further condition that $b+c+d =0$, so we have: 
\begin{align*}
 \mathtt{L}^\sigma_{\emptyset}  &\cong\Big\{~  (b,c,d)\in  (\C^*)^3 ~\Big|~ d\neq bc \mbox{ and } 1+b+c+d = 0 ~\Big\} \\
& \cong \Big\{~  (b,c)\in  \big(\C^*\big)^2 ~\Big|~ (b+1)(c+1) \neq 0 \mbox{ and } b+c \neq -1~\Big\}\\
& \cong \Big\{~  (b,c)\in \big(\C^* \setminus\{-1\}\big)^2  ~\Big|~  b+c \neq -1~\Big\}\\
& \cong   \big(\C^* \setminus\{-1\}\big)^2 - \big(\C \setminus\{2\mathrm{pt}\}\big).
\end{align*}
Hence we have the Euler characteristics $e(\mathtt{L}^\sigma_{\emptyset}) =2$ and $e(\mathtt{L}^\emptyset_{\emptyset}) = -2$. 
\end{itemize}
\end{re}

\subsection{Topological Vertex Formulas}

In this section of the appendix we collect some useful formulas for partition functions involving the topological vertex. 

Define the ``MacMahon'' notation:
\[
M(p,Q) = \prod_{m>0} (1-p^m Q)^{-m}
\]
and the simpler version $M(p) = M(p,1)$. 

\begin{lemma} \label{vertex_splitting_lemma}
We have the equality: 
\[
\mathsf{V}_{\lambda \square \square}\mathsf{V}_{\lambda \emptyset \emptyset} = \frac{1}{p}\mathsf{V}_{\lambda \emptyset \emptyset}\mathsf{V}_{\lambda \emptyset \emptyset} +\mathsf{V}_{\lambda \square \emptyset} \mathsf{V}_{\lambda \emptyset \square}
\]
\end{lemma}
\begin{proof}
We prove the equivalent equation:
\[
\frac{\mathsf{V}_{ \square \square\nu}}{\mathsf{V}_{ \emptyset \emptyset \nu}}  = \frac{1}{p} +\frac{\mathsf{V}_{\square \emptyset \nu} \mathsf{V}_{ \emptyset \square\nu}}{(\mathsf{V}_{ \emptyset \emptyset \nu})^2}
\]
From the definition we have:
\begin{align*}
\frac{\mathsf{V}_{ \square \square\nu}}{\mathsf{V}_{ \emptyset \emptyset \nu}} 
=&
\frac{1}{p} \sum_{\eta \subset \square} S_{\square/\eta}(p^{-\nu-\rho})S_{\square/\eta}(p^{-\nu^t-\rho})\\
=& 
\frac{1}{p} \left(S_{\square/\emptyset}(p^{-\nu-\rho})S_{\square/\emptyset}(p^{-\nu^t-\rho})+ S_{\square/\square}(p^{-\nu-\rho})S_{\square/\square}(p^{-\nu^t-\rho})\right)\\
=& 
\frac{\mathsf{V}_{\square \emptyset \nu} \mathsf{V}_{ \emptyset \square\nu}}{(\mathsf{V}_{ \emptyset \emptyset \nu})^2} +\frac{1}{p} 
\end{align*}
\end{proof}

\begin{lemma}\label{vertex_squares_lemma}
We have
\begin{enumerate}[label={\arabic*)}]
\item\label{vertex_squares_lemma_eee} $\mathsf{V}_{\emptyset\emptyset\emptyset} = M(p)$
\item\label{vertex_squares_lemma_bee} $\mathsf{V}_{\square\emptyset\emptyset} = M(p) \frac{1}{1-p}$
\item\label{vertex_squares_lemma_bbe} $\mathsf{V}_{\square\square\emptyset} = M(p) \frac{p^2-p+1}{p(1-p)^2}$
\item\label{vertex_squares_lemma_bbb} $\mathsf{V}_{\square\square\square} = M(p) \frac{p^4-p^3+p^2-p+1}{ p^2(1-p)^3}$
\end{enumerate}

\end{lemma}
\begin{proof}
Part \ref{vertex_squares_lemma_eee} is immediate from the definition. For part \ref{vertex_squares_lemma_bee} we have:
\begin{align*}
\mathsf{V}_{\square\emptyset\emptyset} 
&= M(p) p^{-\frac{1}{2}}   S_{\emptyset}(p^{-\rho}) \sum_{\eta} S_{\square/\eta}(p^{-\rho})S_{\emptyset/\eta}(p^{-\rho})\\
&= M(p) \frac{1}{1-p}
\end{align*}
For part \ref{vertex_squares_lemma_bbe} we have:
\begin{align*}
\mathsf{V}_{\square\square\emptyset} 
&= M(p) p^{-1}   S_{\emptyset}(p^{-\rho}) \sum_{\eta} S_{\square/\eta}(p^{-\rho})S_{\square/\eta}(p^{-\rho})\\
&= M(p)p^{-1} \Big(S_{\square/\emptyset}(p^{-\rho})S_{\square/\emptyset}(p^{-\rho}) + S_{\square/\square}(p^{-\rho})S_{\square/\square}(p^{-\rho}) \Big)\\
&= M(p)p^{-1} \Big(\frac{p}{(1-p)^2} +1\Big)\\
&= M(p) \frac{p^2-p+1}{p(1-p)^2}
\end{align*}
Part \ref{vertex_squares_lemma_bbb} follows from parts \ref{vertex_squares_lemma_bee} and \ref{vertex_squares_lemma_bbe} and lemma \ref{vertex_splitting_lemma}:
\[
\mathsf{V}_{\square \square \square}= \frac{1}{p}\mathsf{V}_{\square \emptyset \emptyset} +\frac{\mathsf{V}_{\square \square \emptyset} \mathsf{V}_{\square \emptyset \square}}{\mathsf{V}_{\square \emptyset \emptyset} }
\]

\end{proof}

\begin{re}\label{full_banana_configuration_partition_re} It is shown in \cite[\S 4.3]{Bryan_Banana} that:
\begin{align*}
&\sum\limits_{\nu, \alpha, \mu} Q_1^{|\nu|}Q_2^{|\alpha|}Q_3^{|\mu|}p^{\frac{1}{2}(\|\nu\|^2+\| \nu^t\|^2+\|\alpha\|^2+\| \alpha^t\|^2 +\|\mu\|^2+\| \mu^t\|^2)} (\mathsf{V}_{ \nu \mu \alpha}\mathsf{V}_{  \nu^t \mu^t  \alpha^t}).\\
&= \prod_{d_1,d_2,d_3\geq 0} \prod_{k} (1- (-Q_1)^{d_1}(-Q_2)^{d_2}(-Q_3)^{d_3} p^k )^{- c(\|\bm{d}\|,k)}
\end{align*}
where $\bm{d} = (d_1,d_2,d_3)$ and the second product is over $k\in \mathbb{Z}$ unless $\bm{d} = (0,0,0)$ in which case $k>0$. The powers $ c(\|\bm{d}\|,k)$ are defined by
\[
\sum_{a=-1}^\infty \sum_{k\in\mathbb{Z}} c(a,k)Q^a y^k := \frac{\sum_{k\in \mathbb{Z}} Q^{k^2} (-y)^k}{\left(\sum_{k\in \mathbb{Z}+\frac{1}{2}} Q^{2k^2} (-y)^k\right)^2} = \frac{\vartheta_4(2\tau,z)}{\vartheta_1(4\tau,z)^2}
\]
and $\|\bm{d}\|:= 2d_1d_2 +2 d_2d_3 +2d_3d_1 -d_1^2 -d_2^2-d_3^2$. Also, if we recall the notation that
\[
\psi_g := \left(\frac{p}{(1-p)^2}\right)^{1-g}
\]
then we have the following corollary.

\end{re}

\begin{lemma}\label{banana_subpartition_functions}
We have:
\begin{enumerate}[label={\arabic*)}]
\item \label{banana_subpartition_functions_double}
$\sum\limits_{\alpha, \mu} Q_2^{|\alpha|}Q_3^{|\mu|}p^{\frac{1}{2}(\|\alpha\|^2+\| \alpha^t\|^2 +\|\mu\|^2+\| \mu^t\|^2)} (\mathsf{V}_{ \emptyset \mu \alpha}\mathsf{V}_{  \emptyset \mu^t  \alpha^t})$\\
$= M(p)^{2} \prod\limits_{m>0} \dfrac{M(p,Q_2^m Q_3^m)^{2}}{(1-Q_2^m  Q_3^m) M(p,-Q_2^{m-1}Q_3^m) M(p,-Q_2^m Q_3^{m-1})}$

\item \label{banana_subpartition_functions_single_empty_empty}
$\sum\limits_{\alpha} Q^{|\alpha|} p^{\frac{1}{2}(\|\alpha\|^2+\| \alpha^t\|^2)} (\mathsf{V}_{ \emptyset \emptyset \alpha}\mathsf{V}_{ \emptyset \emptyset \alpha^t})  =M(p)^2 \prod\limits_{m>0}(1+p^m Q)^{m}$

\item
$\sum\limits_{\alpha} Q^{|\alpha|} p^{\frac{1}{2}(\|\alpha\|^2+\| \alpha^t\|^2)+1} (\mathsf{V}_{ \square \emptyset \alpha}\mathsf{V}_{ \square \emptyset \alpha^t})$\\ 
$= M(p)^2 (\psi_0 +(\psi_1 +2\psi_0) Q +\psi_0 Q^2) \prod\limits_{m>0}(1+p^m Q)^{m}$

\item
$\sum\limits_{\alpha} Q^{|\alpha|} p^{\frac{1}{2}(\|\alpha\|^2+\| \alpha^t\|^2)+2} (\mathsf{V}_{ \square \square \alpha}\mathsf{V}_{ \square \square \alpha^t})$\\
$= M(p)^2 \prod\limits_{m>0}(1+p^m Q)^{m}\Big(Q^4 (2 \psi_0+\psi_1)+Q^3 (8 \psi_0+6 \psi_1+\psi_2)+Q^2 (12 \psi_0\hspace{2em}$\\
\mbox{}\hspace{12em}$+10 \psi_1+2 \psi_2)+Q (8 \psi_0+6 \psi_1+\psi_2)+(2 \psi_0+\psi_1)\Big)$\\
\end{enumerate}
\end{lemma}
\begin{proof}
These are all coefficients of the partition function in \ref{full_banana_configuration_partition_re}. For example part 3) is the coefficient of $Q_1^1Q_2^0$. 
\end{proof}

\begin{lemma}\label{vertex_section_formulas}
We have the following equalities:
\begin{enumerate}[label={\arabic*)}]
\item
$\sum\limits_{\alpha} Q^{|\alpha|} p^{\frac{1}{2}(\|\alpha\|^2+\| \alpha^t\|^2)} (\mathsf{V}_{ \square \emptyset \alpha}\mathsf{V}_{  \emptyset  \emptyset \alpha^t}) 
= M(p)^2\dfrac{1+Q}{1-p} \prod\limits_{m>0}(1+p^{m} Q)^m  $
\item
$\sum\limits_{\alpha} Q^{|\alpha|} p^{\frac{1}{2}(\|\alpha\|^2+\| \alpha^t\|^2)+1} (\mathsf{V}_{ \square \square \alpha}\mathsf{V}_{  \emptyset  \emptyset \alpha^t})$\\
$= M(p)^2((\psi_0+\psi_1) +(2\psi_0 +\psi_1) Q +\psi_0 Q^2) \prod\limits_{m>0}(1+p^m Q)^{m}$
\item
$\sum\limits_{\alpha} Q^{|\alpha|} p^{\frac{1}{2}(\|\alpha\|^2+\| \alpha^t\|^2)+1} (\mathsf{V}_{\square \emptyset  \alpha})^2$\\
$=M(p)^2(\psi_0 +(2\psi_0+ \psi_1 ) Q +(\psi_0+\psi_1) Q^2) \prod\limits_{m>0}(1+p^m Q)^{m}$
\end{enumerate}

\end{lemma}
\begin{proof}
Part 1) is given by:
\begin{align*}
&\sum\limits_{\alpha} Q^{|\alpha|} p^{\frac{1}{2}(\|\alpha\|^2+\| \alpha^t\|^2)} (\mathsf{V}_{\alpha \square \emptyset }\mathsf{V}_{ \alpha^t \emptyset  \emptyset }) \\
&= p^{-\frac{1}{2}}M(p)^2\sum\limits_{\alpha} Q^{|\alpha|}   \sum_{\eta} S_{\alpha^t/\eta}(p^{-\rho})S_{\square/\eta}(p^{-\rho})\sum_{\delta} S_{\alpha/\delta}(p^{-\rho})S_{\emptyset/\delta}(p^{-\rho})\\
&=p^{-\frac{1}{2}} M(p)^2\sum\limits_{\alpha} Q^{|\alpha|}   \Big(S_{\alpha^t}(p^{-\rho})S_{\square}(p^{-\rho})+ S_{\alpha^t/\square}(p^{-\rho}) \Big)S_{\alpha}(p^{-\rho})\\
&=p^{-\frac{1}{2}} M(p)^2  \Big(S_{\square}(p^{-\rho}) \sum\limits_{\alpha\supset \emptyset} S_{\alpha^t/\emptyset}(p^{-\rho})S_{\alpha/\emptyset}(Qp^{-\rho})+ \sum\limits_{\alpha\supset \square} S_{\alpha^t/\square}(p^{-\rho}) S_{\alpha/\emptyset}(Qp^{-\rho})\Big)
\end{align*}\\

\vspace{-0.25cm}
After applying \cite[Eqn. 2, pg. 96]{Macdonald} the equation becomes
\begin{align*}
& p^{-\frac{1}{2}}M(p)^2 \prod_{i,j>0}(1+p^{i+j} Q) \\
& \Big(S_{\square}(p^{-\rho}) \sum\limits_{\tau\subset\emptyset} S_{\emptyset/\tau}(p^{-\rho})S_{\emptyset/\tau}(Qp^{-\rho})+ \sum\limits_{\tau\subset\emptyset} S_{\emptyset/\tau^t}(p^{-\rho}) S_{\square/\tau}(Qp^{-\rho})\Big)\\
&= p^{-\frac{1}{2}}M(p)^2 \prod_{m>0}(1+p^{m} Q)^m  (1+Q)\frac{p^{\frac{1}{2}}}{1-p}.
\end{align*}\\

\vspace{-0.25cm}
Part 2) follows from lemmas  \ref{vertex_splitting_lemma} and \ref{banana_subpartition_functions}:
\begin{align*}
&\sum\limits_{\alpha} Q^{|\alpha|} p^{\frac{1}{2}(\|\alpha\|^2+\| \alpha^t\|^2)+1} (\mathsf{V}_{ \square \square \alpha}\mathsf{V}_{  \emptyset  \emptyset \alpha^t})\\
&=
\sum\limits_{\alpha} Q^{|\alpha|} p^{\frac{1}{2}(\|\alpha\|^2+\| \alpha^t\|^2)} (\mathsf{V}_{ \emptyset  \emptyset \alpha}\mathsf{V}_{  \emptyset  \emptyset \alpha^t})
+\sum\limits_{\alpha} Q^{|\alpha|} p^{\frac{1}{2}(\|\alpha\|^2+\| \alpha^t\|^2)+1} (\mathsf{V}_{ \square \emptyset \alpha}\mathsf{V}_{   \square \emptyset \alpha^t}).
\end{align*}\\

\vspace{-0.25cm}
Part 3) is given by:
\begin{align*}
&\sum\limits_{\alpha} Q^{|\alpha|} p^{\frac{1}{2}(\|\alpha\|^2+\| \alpha^t\|^2)+1} (\mathsf{V}_{  \square \emptyset \alpha})^2 \\
&=
\sum\limits_{\alpha} Q^{|\alpha|} p^{\frac{1}{2}(\|\alpha\|^2+1)}(\mathsf{V}_{ \alpha \square \emptyset })~ p^{ \frac{1}{2}(\| \alpha^t\|^2+1)} (\mathsf{V}_{  \emptyset \alpha\square }) \\
&=
M(p)^2\sum\limits_{\alpha} Q^{|\alpha|}   
S_{\square}(p^{-\rho})\sum_{\delta} S_{\alpha/\delta}(p^{-\square-\rho})S_{\emptyset/\delta}(p^{-\square-\rho})\\
&\hspace{7.8em}S_{\emptyset}(p^{-\rho}) \sum_{\eta} S_{\alpha^t/\eta}(p^{-\rho})S_{\square/\eta}(p^{-\rho})\\ 
&=
M(p)^2 S_{\square}(p^{-\rho}) \sum\limits_{\alpha}    S_{\alpha}(Q p^{-\square-\rho})
\Big(
 S_{\alpha^t}(p^{-\rho})S_{\square}(p^{-\rho}) + S_{\alpha^t/\square}(p^{-\rho})
\Big)
\end{align*}
After applying \cite[Eqn. 2, pg. 96]{Macdonald} the equation becomes
\begin{align*}
M(p)^2 S_{\square}(p^{-\rho}) (1+Q)\prod_{m>0}(1+Qp^m)^m 
\Big(
S_{\square}(p^{-\rho}) + S_{\square}(p^{-\square-\rho})
\Big).
\end{align*}
The result follows from a quick computation involving $\mathsf{V}_{  \emptyset  \square\square}=\mathsf{V}_{  \square\square \emptyset}$ showing that 
\[
S_{\square}(p^{-\rho})S_{\square}(p^{-\square-\rho}) = S_{\square}(p^{-\rho})^2 +1.
\]\end{proof}

\begin{lemma}\label{vertex_trace_formulas}
The following are true
\begin{enumerate}[label={\arabic*)}]
\item
$\sum\limits_\alpha Q^{|\alpha|}  = \prod\limits_{d>0}\dfrac{1}{(1-Q^d)}$
\item 
$\sum\limits_\alpha Q^{|\alpha|}\dfrac{(\mathsf{V}_{\alpha \square \emptyset})}{(\mathsf{V}_{\alpha\emptyset\emptyset})}  =\dfrac{1}{1-p} \prod\limits_{d>0}\dfrac{(1-Q^d)}{(1-p Q^d)(1-p^{-1}Q^d)}$
\item 
$\sum\limits_\alpha p^{\|\alpha\|^2} Q^{|\alpha|}\dfrac{(\mathsf{V}_{\alpha \alpha^t \emptyset})}{(\mathsf{V}_{\emptyset\emptyset\emptyset})}  =\prod\limits_{d>0}\dfrac{M(p,Q^d)}{(1-Q^d)}$
\item
$\sum\limits_\alpha p^{\|\alpha\|^2} Q^{|\alpha|}\dfrac{(\mathsf{V}_{\alpha \alpha^t \emptyset})(\mathsf{V}_{\alpha \square \emptyset})}{(\mathsf{V}_{\alpha\emptyset\emptyset})(\mathsf{V}_{\emptyset\emptyset\emptyset})}  =\dfrac{1}{1-p}\prod\limits_{d>0}\dfrac{M(p,Q^d)}{(1-p Q^d)(1-p^{-1}Q^d)}$
\end{enumerate}
\end{lemma}
\begin{proof}
The first is a classical result and the other three are the content of \cite[Thm. 3]{BKY}. 
\end{proof}

\end{document}